\documentclass[reqno,12pt,a4paper]{amsart}

\usepackage{amsmath,amssymb,amsthm,graphicx,mathrsfs,url}
\usepackage{a4wide}
\usepackage{enumitem}
\usepackage{mathtools}
\usepackage[utf8]{inputenc}
\usepackage[usenames,dvipsnames]{xcolor}
\usepackage[colorlinks=true,linkcolor=Red,citecolor=Green]{hyperref}
\hypersetup{pdfstartview=XYZ}
\usepackage{tikz-cd}
\usepackage{mathdots}

\usepackage{extarrows}

\theoremstyle{plain}
\newtheorem{introthm}{Theorem}

\newtheorem{thm}{Theorem}
\newtheorem{lem}{Lemma}[section]

\newtheorem{prop}[lem]{Proposition}

\newtheorem{cor}[lem]{Corollary}

\theoremstyle{definition}
\newtheorem{definition}[lem]{Definition}

\theoremstyle{remark}
\newtheorem{rem}{Remark}[section]
 
\numberwithin{equation}{section}

\newcommand{\C}{\mathbb{C}}
\newcommand{\R}{\mathbb{R}}

\newcommand{\Z}{\mathbb{Z}}
\newcommand{\D}{\mathcal{D}}
\newcommand{\N}{\mathbb{N}}

\newcommand{\eps}{\varepsilon}

\newcommand{\norm}[1]{\left\Vert#1\right\Vert}

\let\Im=\Imag

\let\Re=\Real

\DeclareMathOperator{\SL}{SL}
\DeclareMathOperator{\supp}{supp}

\DeclareMathOperator{\Tr}{Tr}

\def\Ddots{\mathinner{\mkern1mu\raise\p@
\vbox{\kern7\p@\hbox{.}}\mkern2mu
\raise4\p@\hbox{.}\mkern2mu\raise7\p@\hbox{.}\mkern1mu}}
\makeatother

\newcommand{\eklm}[1]{\left\langle #1 \right\rangle}

\renewcommand{\d}{\,d}

\newcommand{\M}{{\mathcal M}}

\renewcommand{\epsilon}{\vararepsilon}

\newcommand{\bdm}{\begin{displaymath}}
\newcommand{\edm}{\end{displaymath}}
\newcommand{\bq}{\begin{equation}}
\newcommand{\eq}{\end{equation}}
\newcommand{\bqn}{\begin{equation*}}
\newcommand{\eqn}{\end{equation*}}

\newcommand{\CT}{{C^{\infty}_c}}

\renewcommand{\L}{{\mathbb L}}

\newcommand{\SO}{\mathrm{SO}}

\newcommand{\Ad}{\mathrm{Ad}}

\newcommand{\id}{\mathrm{id}\,}

\renewcommand{\Im}{\mathrm{Im}\,}
\renewcommand{\Re}{\mathrm{Re}\,}
\definecolor{darkblue}{rgb}{0,0,0.4}

\newcommand{\set}[2]{\left\lbrace #1 \,\middle|\, #2 \right\rbrace}

\newcommand{\trprod}{{\mathbb P}(V)\overset{\pitchfork}{\times}{\mathbb P}(V^*)}

\newcommand{\transverse}{\pitchfork}
\newcommand{\bbL}{\mathbb{L}}
\newcommand{\bbP}{\mathbb{P}}

\newcommand{\bbH}{\mathbb{H}}
\newcommand{\bbS}{\mathbb{S}}

\newcommand{\pscal}[2]{\left\langle  #1 \middle| #2\right\rangle}
\newcommand{\map}[4]{\left\lbrace \begin{array}{ccc} #1 & \to & #2 \\ #3 & \mapsto & #4 \end{array} \right.}
\renewcommand{\tilde}{\widetilde}
\newcommand{\cK}{\mathcal{K}}
\newcommand{\cC}{\mathcal{C}}
\newcommand{\cH}{\mathcal{H}}

\newcommand{\cG}{\mathcal{G}}

\newcommand{\cM}{\mathcal{M}}
\newcommand{\cN}{\mathcal{N}}

\author[B.~Delarue]{Benjamin Delarue}
\email{bdelarue@math.upb.de}
\address{Universit\"at Paderborn, Warburger Str.~100, 33098 Paderborn, Germany}

\author[D.~Monclair]{Daniel Monclair}
\email{daniel.monclair@universite-paris-saclay.fr}
\address{Université Paris-Saclay, CNRS, Laboratoire de mathématiques d’Orsay, 91405, Orsay, France}

\author[A.~Monclair]{Andrew Sanders}
\email{andrew.sanders.2@bc.edu}
\address{Boston College, Visiting Scholar, 02467, Chestnut Hill, Massachusetts, United States}

\title[Locally homogeneous Axiom A flows I]{Locally homogeneous Axiom A flows I: projective Anosov subgroups and exponential mixing}

\begin{document}
\begin{abstract}
By constructing a non-empty domain of discontinuity in a suitable homogeneous space, we prove that every torsion-free projective Anosov subgroup is the monodromy group of a locally homogeneous contact Axiom A dynamical system with a unique basic hyperbolic set on which the flow is conjugate to the refraction flow of Sambarino.  Under the assumption of irreducibility, we utilize the work of Stoyanov to establish spectral estimates for the associated complex Ruelle transfer operators, and by way of corollary: exponential mixing, exponentially decaying error term in the prime orbit theorem, and a spectral gap for the Ruelle zeta function.  With no irreducibility assumption, results of Dyatlov-Guillarmou imply the global meromorphic continuation of zeta functions with smooth weights, as well as the existence of a discrete spectrum of Ruelle-Pollicott resonances and (co)-resonant states.  We apply our results to space-like geodesic flows for the convex cocompact pseudo-Riemannian manifolds of Danciger-Gu\'{e}ritaud-Kassel, and the Benoist-Hilbert geodesic flow for strictly convex real projective manifolds. 
\end{abstract}

\maketitle


\section{Introduction}

Given a torsion-free uniform lattice $\Gamma<\mathrm{Isom}_{0}(\mathbb{H}^{n})\simeq \mathrm{SO}_{0}(n,1),$ the discrete group action $\Gamma \curvearrowright T^{1}\mathbb{H}^{n}$ is free, properly discontinuous and cocompact, and the geodesic flow
$$
\varphi^{t}: T^{1}(\Gamma\backslash \mathbb{H}^{n}) \rightarrow T^{1}(\Gamma\backslash \mathbb{H}^{n})
$$
on the unit tangent bundle of  the closed hyperbolic manifold $\Gamma\backslash \mathbb{H}^{n}$   is the Anosov flow par excellence: it mixes exponentially for all Gibbs equilibrium measures, and in particular for the measure of maximal entropy which is equal to the Liouville measure induced by the canonical Riemannian contact structure on $T^{1}(\Gamma\backslash \mathbb{H}^{n})$.

From a dynamical systems perspective, the most natural weakening of the uniform lattice property for $\Gamma<\mathrm{SO}_{0}(n,1)$ is to ask that $(T^{1}(\Gamma\backslash \mathbb{H}^{n}),\varphi^t)$ is a (not necessarily compact) Axiom A system in the sense of Smale \cite{smale67}.  It is well known that the class of ``Axiom A" subgroups $\Gamma<\mathrm{SO}_{0}(n,1)$ is equivalent to the class of convex-cocompact subgroups. They are characterized by the existence of a compact and geodesically convex submanifold (potentially with non-empty boundary) $C_{\Gamma}\subset \Gamma\backslash \mathbb{H}^{n}$ such that the inclusion is a homotopy equivalence.  Every convex-cocompact $\Gamma<\mathrm{SO}_{0}(n,1)$ is Gromov hyperbolic and there is a unique basic hyperbolic set
$\mathcal{K}_{\Gamma}\subset T^{1}(\Gamma\backslash \mathbb{H}^{n})$ of the geodesic flow satisfying $\pi(\mathcal{K}_{\Gamma})\subset C_{\Gamma}$ where $\pi: T^{1}(\Gamma\backslash \mathbb{H}^{n})\rightarrow \Gamma\backslash \mathbb{H}^{n}$ is the natural projection.

The above situation is a dynamical manifestation of negative sectional curvature, and therefore replacing $\mathrm{SO}_{0}(n,1)$ by another semisimple Lie group $G$ of real rank one causes no essential difficulties.  Meanwhile, Lie groups $G$ of real rank greater than one present a conundrum: the Riemannian symmetric space $\mathbb{X}$ for which $G\simeq \mathrm{Isom}_{0}(\mathbb{X})$ is only non-positively curved. It gets worse from here: the $G$-action on the unit tangent bundle $T^{1}\mathbb{X}$ is no longer transitive, and while every $G$-orbit is preserved by the geodesic flow, the flow has dynamically trivial central stable leaves of dimension greater than one manifested by the existence of maximal flats $F\subset \mathbb{X}$ of dimension greater than one.  

The final nail in the coffin of the naive geometro-dynamical 
extensions of convex compactness to higher rank was hammered by the rigidity result of Kleiner-Leeb \cite{KL06} and Quint \cite{Quint05}:  the only convex-cocompact Zariski dense discrete subgroups $\Gamma<G$ of a simple real algebraic group of rank at least two are the uniform lattices, none of which are Gromov hyperbolic, and therefore even the coarse approach  to negative curvature pioneered by Gromov \cite{Gr} was left without a place to get started.  Moreover, the direct geometric techniques invented by Thurston \cite{THU22} (e.g.\ convex hulls and pleated surfaces) which were so fruitful in rank one were equally left without analogue for higher rank $G.$  

This situation has changed dramatically in the previous two decades.  The revolution started with the article \cite{labourie} of Labourie who showed that the limit set focused approach of Gromov-Thurston can be carried out in higher rank.  This work produced the first systematic treatment of rank one behavior in higher rank, with many hints provided by Benoist's prior theory of strictly convex divisible domains \cite{BEN08} in real projective space.  In fact, Labourie's original approach was primarily dynamical and defined an \emph{Anosov} subgroup $\Gamma<G$ through the existence of (rank one) hyperbolic dynamics in an appropriate bundle (with connection) $E_{\Gamma}\rightarrow T^{1}\mathcal{N}_{\Gamma}$ where $\mathcal{N}_{\Gamma}$ is a closed Riemannian manifold of negative sectional curvature with $\pi_{1}(\mathcal{N}_{\Gamma})\simeq \Gamma.$  

This breakthrough has led to an elaborate geometric theory of discrete subgroups $\Gamma<G$ of semisimple Lie groups.  The sustained work of a myriad of mathematicians has produced multiple distinct approaches to Labourie's theory of Anosov homomorphisms, putting the (coarse) geometric  elements largely on par with the classical theory of convex-cocompact subgroups. To name a few, the coarse-geometric viewpoint was developed in a series of self-contained articles by Kapovich-Leeb-Porti \cite{KLP17, KLP18, KLP2}.  Meanwhile Sambarino, utilizing the methods and works of Benoist \cite{BEN97, BEN00}, Ledrappier \cite{LED95}, and Quint \cite{QUI02C, QUI02D}, led the advancement of the dynamical study (see the survey paper \cite{SAM24}), in particular obtaining many quantitative counting results for periods associated to Anosov subgroups \cite{SAM14}.

Despite this spirited activity, some foundational goals concerning hyperbolic dynamical systems, in particular spectral estimates for Ruelle transfer operators and exponential mixing, have so far been out of reach.  This is largely due to the low regularity (fractal in general) of the available models for the Gromov geodesic flow \cite{Gr}, \cite{Mineyev2005}.  The relevance of the Gromov geodesic flow to the theory of Anosov subgroups was first elucidated by the foundational work of Guichard-Wienhard \cite{GW12}, followed by Sambarino's \cite{SAM14, SAM24} explicit construction of models with prescribed periods; adhering to \cite{SAM24}, the model of interest in this article is called the \emph{refraction flow}.

By way of example regarding low regularity, it has been known since Benoist's result \cite{BEN04} that the Hilbert geodesic flow on the unit tangent bundle $T^{1}\mathcal{N}$ of a closed, strictly convex real projective manifold $\mathcal{N}$ is a $C^{1, \alpha}$-topologically transitive Anosov flow, but spectral estimates for Ruelle transfer operators and exponential mixing have remained elusive. Influenced by the ideas of Chernov \cite{CHE98}, Dolgopyat's groundbreaking discovery \cite{DOL98} of the coercive relationship between  the transverse regularity of the (un)stable foliations and questions of exponential mixing indicates the challenges posed by the low regularity of the Benoist-Hilbert geodesic flow, and Sambarino's refraction flows in general.

This article constitutes the first in a series where we develop techniques to resolve these issues around regularity by embedding Sambarino's refraction flows as basic hyperbolic sets in real analytic locally homogeneous contact Axiom A systems.  It is important to note that these Axiom A systems are almost never compact, with the exception being uniform lattices $\Gamma<\mathrm{SL}_{2}(\mathbb{R})$ in the projective Anosov case.

Fix a finite-dimensional real vector space $V$ of dimension $d>1$.  Let $\mathbb{L}\subset \mathbb{P}\left(V\times V^{*}\right)$ denote the open submanifold given by the projectivization of the affine quadric hypersurface $\{(v, \alpha)\in V\times V^{*}\ | \ \alpha(v)=1\}.$ The group $\mathrm{SL}(V)$ acts diagonally on $\mathbb{P}\left(V\times V^{*}\right),$ preserves $\mathbb{L},$ and commutes with the flow $\phi^{t}: \mathbb{L}\rightarrow \mathbb{L}$ induced by $\phi^{t}(v, \alpha)=(e^{t}v, e^{-t}\alpha).$ Moreover, this flow is the Reeb flow of an $\mathrm{SL}(V)$-invariant contact one-form on $\mathbb{L}$ descended from the tautological Liouville one-form on $T^{\ast}V=V\times V^{*}.$  The $\mathrm{SL}(V)$-action on $\mathbb{L}$ is transitive and the tangent bundle of $\mathbb{L}$ admits an   $\mathrm{SL}(V)\times \{d\phi^t\}_{t\in\R}$-equivariant real analytic splitting 
\bq
T_{[v:\alpha]}\mathbb{L}\simeq \mathrm{ker}(\alpha)\oplus E^{0} \oplus \mathrm{ker}(\iota_{v}) \label{eq:splittingofTL}
\eq
where $E^{0}$ is the central flow direction and $\iota_{v}: V^{*}\rightarrow \mathbb{R}$ is the contraction with $v\in V.$  
\begin{introthm}\label{THM A}
Suppose $\Gamma<\mathrm{SL}(V)$ is a torsion-free projective Anosov subgroup.  
\begin{enumerate}
\item \label{DOD} There exists a non-empty $\Gamma\times \{\phi^{t}\}_{t\in \R}$-invariant open subset $\widetilde{\mathcal{M}_{\Gamma}}\subset \mathbb{L}$ upon which $\Gamma$ acts freely and properly discontinuously.  
Therefore, the quotient 
$$
\mathcal{M}_{\Gamma}:=\Gamma\backslash \widetilde{\mathcal{M}_{\Gamma}}
$$
is a real analytic locally homogeneous  $\left(\mathrm{SL}(V), \mathbb{L}\right)$-manifold, endowed with a real analytic complete flow $\phi^t:\cM_\Gamma\to\cM_\Gamma$ preserving a real analytic contact one-form for which $\phi^t$ is the Reeb flow.

\item \label{THM A: AxiomA} The flow $\phi^t:\cM_\Gamma\to\cM_\Gamma$ is Axiom A in the sense of Smale and preserves a unique basic hyperbolic set $$\mathcal{K}_{\Gamma}\subset \mathcal{M}_{\Gamma}.$$

\item \label{real analytic fol} The real analytic splitting \eqref{eq:splittingofTL} descends to a real analytic splitting
$$
T\mathcal{M}_{\Gamma}=E^{\mathrm{u}}\oplus E^{0} \oplus E^{\mathrm{s}}
$$
whose restriction to the basic hyperbolic set $\mathcal{K}_{\Gamma}\subset \mathcal{M}_{\Gamma}$ is the splitting into unstable-central-stable directions guaranteed by hyperbolicity.  In particular, each of the four stable, unstable, central stable, and central unstable foliations along $\mathcal{K}_{\Gamma}$ are the restriction of global real analytic foliations on $\mathcal{M}_{\Gamma}.$  
 \end{enumerate}
\end{introthm}
 In Theorem \ref{THM A}.(\ref{DOD}), a locally homogeneous $(\mathrm{SL}(V), \mathbb{L})$-manifold is a real analytic manifold with an $\mathbb{L}$-valued real analytic atlas whose coordinate changes are locally given by the restriction of elements of $\mathrm{SL}(V).$

By a judicious use of representation theory, Guichard-Wienhard showed \cite[Prop. 4.3, Rem.~4.12]{GW12} that every Anosov subgroup $\Gamma<G$ of a semisimple Lie group $G$ may be made projective Anosov after post-composition with a linear representation $G\rightarrow \mathrm{SL}(W)$ for some finite-dimensional real vector space $W$.  Therefore, Theorem \ref{THM A} establishes that all torsion-free Anosov subgroups are monodromy groups of (generally non-compact) Axiom A dynamical systems. By Selberg's Lemma, the torsion-free property can always be achieved by passing to a finite index subgroup, which preserves the Anosov property \cite[Cor.~1.3]{GW12}.

A choice of norm on $V$ induces a real analytic trivialization of the principal $\R$-bundle 
$$
\mathbb{L}\rightarrow \mathbb{P}(V)\overset{\pitchfork}\times \mathbb{P}(V^{\ast})
$$
over the space $\mathbb{P}(V)\overset{\pitchfork}\times \mathbb{P}(V^{\ast})$ of all transverse line/hyperplane pairs $([v], [\alpha])\in \mathbb{P}(V)\times \mathbb{P}(V^{\ast}),$ i.e. those satisfying $\alpha(v)\neq 0$.  With respect to this (or any) trivialization, the $\mathrm{SL}(V)$-action becomes a real analytic cocycle 
$$
\mathscr{H}: \mathrm{SL}(V)\times \mathbb{P}(V)\overset{\pitchfork}\times \mathbb{P}(V^{\ast})\rightarrow \mathbb{R}
$$ 
prescribing the $\mathrm{SL}(V)$-action in the choice of trivialization.  This cocycle was initially utilized in the study of discrete group actions by Quint \cite{QUI02C, QUI02D}, and after the work \cite{SAM14} appears throughout the literature on Anosov subgroups/homomorphisms.

A projective Anosov subgroup $\Gamma<\SL(V)$ comes with a pair of $\Gamma$-equivariant limit maps $\xi:\partial_{\infty}\Gamma\to\mathbb{P}(V)$, $\xi^\ast:\partial_{\infty}\Gamma\to\mathbb{P}(V^\ast)$ defined on the Gromov boundary $\partial_{\infty}\Gamma$. Denoting the complement of the diagonal in $\partial_{\infty}\Gamma\times \partial_{\infty}\Gamma$ by $\partial_{\infty}\Gamma^{(2)}$, Sambarino (see \cite{SAM14, SAM24}) observed that the product map
$$
\xi\times\xi^{*}:\partial_{\infty}\Gamma^{(2)}\longrightarrow \mathbb{P}(V)\overset{\pitchfork}\times \mathbb{P}(V^{\ast})
$$ 
can be post-composed by the cocycle $\mathscr{H}$ to obtain a 
H\"{o}lder cocycle 
$$
\mathscr{H}_{\Gamma}: \Gamma\times \partial_{\infty}\Gamma^{(2)}\rightarrow \mathbb{R}.
$$
Upon leveraging the results of Ledrappier \cite{LED95}, Sambarino proved that the induced $\Gamma$-action on 
$\partial_{\infty}\Gamma^{(2)}\times \mathbb{R}$ is properly discontinuous and cocompact. He denoted the quotient by this action  $\chi_{\Gamma}$  and called the induced flow on $\chi_{\Gamma}$ the \emph{refraction flow}  \cite{SAM24}.

Our second result establishes that the refraction flow space $\chi_{\Gamma}$ appears as the basic hyperbolic set of the real analytic contact Axiom A system from Theorem \ref{THM A}.
\begin{introthm}\label{SAMHYP}
Suppose $\Gamma<\mathrm{SL}(V)$ is a torsion-free projective Anosov subgroup.  In the setting of Theorem \ref{THM A}, there is a canonical bi-H\"{o}lder continuous homeomorphism
$$
F: \chi_{\Gamma} \xlongrightarrow{\simeq}  \mathcal{K}_{\Gamma}
$$
such that $F\circ  \phi_{S}^{t} = \phi^{t}\circ F$ where $\phi_{S}^{t}: \chi_{\Gamma}\rightarrow \chi_{\Gamma}$ is Sambarino's refraction flow.
\end{introthm}
It is notable that through the work of Sambarino (e.g. \cite{SAM14}, \cite{BCLS}), the $\mathrm{SL}(V)$-equivariant affine line bundle $\mathbb{L}$ and basic hyperbolic set $\chi_{\Gamma}\simeq\mathcal{K}_{\Gamma}$ have already been central objects in the study of Anosov subgroups/homomorphisms for over a decade, but no domain of discontinuity for the $\Gamma$-action on $\mathbb{L}$ containing the lift $\widetilde{\mathcal{K}}_{\Gamma}\subset \mathbb{L}$ of the basic set $\mathcal{K}_{\Gamma}$ had previously been found.  

We now begin our demonstration of what is possible when one realizes a Smale flow on a compact metric space (e.g.\ the refraction flow, c.f.\ \cite{POL87}) as a basic hyperbolic set in a (real analytic) Axiom A system.  It is here that the high regularity of the associated foliations in Theorem \ref{THM A}.(\ref{real analytic fol}) are of utmost importance.

Theorem \ref{THM A} allows Stoyanov's \cite{St11} remarkable extension of the techniques of Dolgopyat \cite{DOL98} to be brought to bear in studying Ruelle transfer operators and exponential mixing for the refraction flow $\phi_{S}^{t}: \chi_{\Gamma}\rightarrow \chi_{\Gamma}$.  We view the following result as a \emph{proof of concept} that Theorem \ref{THM A} opens the door to a deeper analysis of the dynamics of the refraction flow associated to projective Anosov subgroups.
\begin{introthm}\label{EXP MIX}
Suppose $\Gamma<\mathrm{SL}(V)$ is a torsion-free irreducible projective Anosov subgroup.  Let $\mathcal{M}_{\Gamma}$ be the Axiom A dynamical system provided by Theorem \ref{THM A} with unique basic hyperbolic set $\mathcal{K}_{\Gamma}\subset \mathcal{M}_{\Gamma}.$  

For every $0<\alpha<1$ and $U\in C^{\alpha}(\mathcal{K}_{\Gamma}, \mathbb{R}),$ there exist $c_{\alpha}, C_{\alpha}>0,$ depending only on $U$ and $\alpha,$ such that for every $F,G\in C^{\alpha}(\mathcal{K}_{\Gamma}, \mathbb{R})$ the correlations decay exponentially:
\begin{multline*}
\forall\,t\in \R: \left\lvert \int_{z\in \mathcal{K}_{\Gamma}} F(z)\cdot G(\phi^{t}(z)) \ d\mu_{U}(z) - \int_{z\in \mathcal{K}_{\Gamma}} F(z) \ d\mu_{U}(z)\int_{z\in \mathcal{K}_{\Gamma}} G(z)\ d\mu_{U}(z)\right\rvert \\
\leq C_{\alpha}e^{-c_{\alpha} |t|}\lVert F\rVert_{\alpha}\lVert G\rVert_{\alpha}.
\end{multline*}
Above, the measure of integration $d\mu_{U}$ is induced from the unique Gibbs equilibrium state $\mu_{U}$ of maximal topological pressure for the H\"{o}lder potential $U.$  In particular, the flow $\phi^{t}: \mathcal{K}_{\Gamma}\rightarrow \mathcal{K}_{\Gamma}$ is exponentially mixing for all Gibbs states and H\"{o}lder observables, and therefore applying Theorem \ref{SAMHYP}, the refraction flow $\phi_{S}^{t}: \chi_{\Gamma}\rightarrow \chi_{\Gamma}$ is exponentially mixing for all Gibbs states and H\"{o}lder observables.
\end{introthm}
The existence of Gibbs states for the refraction flow is originally due to Sambarino (see \cite{SAM24}), though in light of Theorems \ref{THM A} and \ref{SAMHYP} it is also an immediate consequence of the results of Bowen-Ruelle \cite{BR75}.

\subsection{Further consequences}

In this section we will explain some additional dynamical consequences of Theorem \ref{THM A}. The Axiom A system $(\mathcal{M}_{\Gamma},\phi^t)$ associated  to a non-trivial torsion-free projective Anosov subgroup $\Gamma<\mathrm{SL}(V)$ by Theorem \ref{THM A} carries a period function
\begin{align}\label{periods notation}
\lambda_{1}: \Gamma\rightarrow [0, +\infty)
\end{align}
which is conjugacy invariant and positive, i.e. $\lambda_{1}(\gamma)>0$ for all non-trivial $e\neq \gamma\in \Gamma.$  As suggested by \eqref{periods notation}, $\lambda_{1}(\gamma)$ is equal to the natural logarithm of the spectral radius of the unimodular linear map $\gamma\in \mathrm{SL}(V).$  The associated Ruelle zeta function is formally defined by the Euler product
\begin{align}\label{Euler product intro}
\zeta_{\Gamma}(s)=\prod_{[\gamma]\in \left[\Gamma\right]_{\mathrm{prim}}} \left(1-e^{-s\lambda_{1}([\gamma])}\right)^{-1},
\end{align}
where $\left[\Gamma\right]_{\mathrm{prim}}$ is the set of non-trivial conjugacy classes of primitive elements in the hyperbolic group $\Gamma.$  The Euler product \eqref{Euler product intro} converges for $\Re(s)>h_{\mathrm{top}}(\Gamma)$, where $h_{\mathrm{top}}(\Gamma)$ denotes the topological entropy of the flow $\phi^{t}: \mathcal{K}_{\Gamma}\rightarrow \mathcal{K}_{\Gamma}$ restricted to the unique basic hyperbolic set $\mathcal{K}_{\Gamma}\subset \mathcal{M}_{\Gamma}.$  Note that $h_{\mathrm{top}}(\Gamma)>0$ if and only if $\Gamma$ is non-abelian.

In light of Theorem \ref{THM A}, the resolution of Smale's conjecture due to Dyatlov-Guillarmou \cite{DG16, DG18} and Borns-Weil-Shen \cite{BWS21} immediately implies the global meromorphic continuation of $\zeta_{\Gamma}.$ Moreover, provided $\Gamma<\SL(V)$ is irreducible, in Theorem \ref{contracting ruelle} we adapt the spectral estimates on Ruelle transfer operators achieved by Stoyanov \cite{St11}, which when combined with the work Pollicott-Sharp \cite{PS98} (see also Dolgopyat-Pollicott \cite{DP98}) implies a zero-free strip to the left of the simple pole $h_{\mathrm{top}}(\Gamma)\in \C$:
\begin{introthm}\label{zeta fcn}
Suppose $\Gamma<\mathrm{SL}(V)$ is a non-trivial torsion-free projective Anosov subgroup.  Then the associated Ruelle zeta function $\zeta_{\Gamma}(s)$ admits a meromorphic continuation to all $s\in \mathbb{C}$ with a simple pole at $s=h_{\mathrm{top}}(\Gamma).$

If $\Gamma<\mathrm{SL}(V)$ is irreducible, then $\zeta_{\Gamma}$ has a zero-free spectral gap: there exists $\varepsilon>0$ such that $\zeta_{\Gamma}$ is holomorphic and nowhere vanishing in the strip $h_{\mathrm{top}}(\Gamma)-\varepsilon<\mathrm{Re}(s)<h_{\mathrm{top}}(\Gamma).$ 
\end{introthm}
We actually obtain more general results for weighted zeta functions, see the more detailed Theorems \ref{zeta potential strip}, \ref{thm:thmzetasmooth} and \ref{RH} in Section \ref{sec:Resonances}.

A result similar to Theorem \ref{zeta fcn} appeared recently in the article  \cite{PS24} by  Pollicott-Sharp, where they prove the meromorphic continuation of $\zeta_{\Gamma}$ for all projective Anosov surface groups $\Gamma\simeq \pi_{1}(S);$ the space $S$ is a closed, connected, orientable topological surface of genus at least two.  Under the stronger condition of $(1,1,2)$-hyperconvexity (\cite[Def.~6.1]{PSW21}), which by results of Pozzetti-Sambarino-Wienhard \cite{PSW21} implies that the limit set is $C^{1},$ Pollicott-Sharp prove that $\zeta_{\Gamma}$ is analytic and zero-free in a strip to the left of $h_{\mathrm{top}}(\Gamma)$.  In our result, we require no hypothesis on $\Gamma$ other than being projective Anosov and torsion-free (and irreducible for the spectral gap), and therefore no hypothesis on the regularity of the limit set.  But, there is still daylight between our results and those in \cite{PS24} since $(1,1,2)$-hyperconvex representations are not necessarily irreducible (\cite[Sec.~9]{PSW21}).
It is important to note that both results ultimately rely on an analysis of Ruelle transfer operators and the thermodynamic formalism.

To clarify further the nature of the hypotheses in \cite{PS24}, the $(1,1,2)$-hyperconvexity condition is satisfied by Hitchin surface subgroups $\Gamma<\mathrm{SL}(d, \mathbb{R})$ acting on $\mathbb{RP}^{d-1}$ \cite{PSW21}, but does not hold for general projective Anosov surface subgroups.  Non-trivial Zariski dense counterexamples are provided by embedding $\mathrm{SO}_{0}(3,1)<\mathrm{SL}(4, \mathbb{R})$ and deforming a strictly quasi-Fuchsian surface group $\Gamma<\mathrm{SO}_{0}(3,1)$ into $\mathrm{SL}(4, \mathbb{R}).$  If $\Gamma$ is not the fundamental group of a closed Riemannian manifold of negative sectional curvature, then $\partial_{\infty}\Gamma$ is often totally disconnected (see \cite{BOW98}), the refraction flow is fractal in nature, and the power of the Axiom A approach is on full display.

Next, we deduce the strengthening of the prime orbit theorem of Sambarino \cite{SAM14} which now comes with the desirable exponentially decaying error term.  Sambarino's result is more general than the irreducible projective Anosov case, but does not yield the exponential error term.  Alternatively, equipped with Theorem \ref{THM A} the result of Parry-Pollicott \cite{PP83} can be applied to obtain a different proof of Sambarino's result in the projective Anosov case, again without the exponential error term.

Given a projective Anosov subgroup $\Gamma< \mathrm{SL}(V),$ define the orbit counting function
$$
N_{\Gamma}(t)=\# \set{[\gamma]\in \left[\Gamma\right]_{\mathrm{prim}}}{ \mathrm{P}_{\Gamma}([\gamma])\leq t}.
$$
The offset Logarithmic integral is defined via
\begin{align}\label{eq:offsetintegral}
\mathrm{Li}(t)=\int_{x=2}^{t} \frac{dx}{\log{x}}
\end{align}
and satisfies $\mathrm{Li}(t)\sim \frac{t}{\log{t}}$ as $t\rightarrow +\infty.$  
\begin{introthm}\label{orbit counting}
Suppose $\Gamma< \mathrm{SL}(V)$ is a non-trivial irreducible projective Anosov subgroup.  Then there exists $0<c<h_{\mathrm{top}}(\Gamma)$ such that
$$
N_{\Gamma}(t)=\mathrm{Li}\left(e^{h_{\mathrm{top}}(\Gamma)t} \right)\left(1+ O\big(e^{-(h_{\mathrm{top}}(\Gamma)-c)t}\big)\right).
$$

\end{introthm}
This result is obtained by Pollicott-Sharp \cite{PS24} for $(1,1,2)$-hyperconvex projective Anosov surface groups $\Gamma$. 

\subsection{Dynamical resonances, (co-)resonant states and flow-invariant distributions}

Let $\Gamma<\mathrm{SL}(V)$ be non-trivial, torsion-free and projective Anosov and $(\mathcal{M}_{\Gamma},\phi^t)$ the Axiom A dynamical system provided by Theorem \ref{THM A} with basic set $\mathcal K_\Gamma$. The  infinitesimal generator
$$
X_{x}=\frac{d}{dt}\phi^{t}(x)|_{t=0}\in T_{x}\mathcal{M}_{\Gamma}
$$
of the flow $\phi^t$ comes with an interesting spectral theory. 
We have no intention in this article to dive into the technical apparatus of microlocal and functional analysis required to describe the spectral properties of $X$: a detailed analysis will appear elsewhere. Recall, the first order differential operator $X$ acts on distributions through its action on compactly supported smooth test functions $U\in C_{c}^{\infty}(\mathcal{M}_{\Gamma}, \C)$. The following preliminary result will be elucidated and stated in a more precise and detailed form as Theorems \ref{Resonant states precise}, \ref{thm:resbounds} and \ref{Resonant states precise2} in Section \ref{sec:dynres}.
\begin{introthm}\label{Resonant states}For each smooth potential $U\in \CT(\mathcal{M}_{\Gamma}, \C)$ supported in a sufficiently tight open neighborhood around the basic hyperbolic set $\mathcal K_\Gamma$, there is an infinite discrete spectrum  $\mathcal{R}^{\mathbf X}_{\Gamma}\subset \C$ of \emph{Ruelle-Pollicott resonances}  associated to the operator $\mathbf X:=-X+U.$ When $U=0$,  $h_{\mathrm{top}}(\Gamma)\in \mathcal{R}^{\mathbf X}_{\Gamma}$ is the unique leading resonance (i.e., the resonance with largest real part).  

In addition, for every resonance $\lambda_0\in \mathcal{R}^{\mathbf X}_{\Gamma}$, there exist pairs of non-zero distributions $u$ (respectively $v$) called \emph{resonant states} (respectively \emph{coresonant states}) defined in a relatively compact open neighborhood of $\mathcal{K}_{\Gamma}$ and satisfying $\mathbf X u=\lambda_0 u$, $\mathbf X^\ast v=\overline{\lambda}_0 v$ where $\mathbf X^\ast:=X+\overline U$. Their product $u\cdot \overline{v}$ is a well-defined $\phi^t$-invariant distribution satisfying $\mathrm{supp}(u\cdot \overline{v})\subset \mathcal{K}_{\Gamma}$.  

 For $h_{\mathrm{top}}(\Gamma)\in \mathcal{R}_{\Gamma}^{\mathbf{X}}$ with $U=0$, there is a unique such distribution $u\cdot \overline{v}$ up to scaling, given by the ergodic Radon measure associated to the Bowen-Margulis measure of maximal entropy.
\end{introthm}
The Ruelle-Pollicott resonances are also called dynamical or classical resonances.
The Bowen-Margulis measure was constructed by Sambarino \cite{SAM14}. For resonances $\lambda_{0}\neq h_{\mathrm{top}}(\Gamma)$ and $\mathrm{dim}(\mathcal{M}_{\Gamma})>3$, the construction of the $\phi^{t}$-invariant distributions $u\cdot \overline{v}$ (and the invariant Ruelle distributions $\mathcal T_{\lambda_0}$ introduced in the more precise Theorem \ref{Resonant states precise2}) have not appeared elsewhere, as they depend on our construction of the Axiom A system $\mathcal{M}_{\Gamma}$ from Theorem \ref{THM A}.

\subsection{Applications to Finsler and pseudo-Riemannian geodesic flows}
In the final examples Section \ref{Examples}, we apply our Axiom A framework in two specific settings: $\bbH^{p,q}$-convex cocompact subgroups $\Gamma<\mathrm{SO}(p,q+1)$ developed by Danciger-Gu\'{e}ritaud-Kassel \cite{dgk18}, and strictly convex divisible domains developed by Benoist \cite{BEN08}.  In particular, the space-like geodesic flow for convex cocompact $\mathbb{H}^{p,q}$-manifolds and the Benoist-Hilbert geodesic flow for strictly convex $\mathbb{RP}^{d}$-manifolds are shown to be exponentially mixing.

Suppose $\Gamma< \mathrm{SL}(V)$ divides a strictly convex divisible domain $\mathcal{C}_{\Gamma}\subset \mathbb{P}(V)$ and 
$$
\mathcal{N}_{\Gamma}=\Gamma\backslash \mathcal{C}_{\Gamma}
$$
is the associated closed strictly convex real projective manifold. The Hilbert metric (see Section \ref{Examples}) induces a $C^{1, \alpha}$-norm on $T\mathcal{N}_{\Gamma}\setminus \mathcal{Z}$ where $\mathcal{Z}\subset T\mathcal{N}_{\Gamma}$ denotes the zero section.  Let
$$
\mathbb{S}\mathcal{N}_{\Gamma}:=(T\mathcal{N}_{\Gamma}\setminus \mathcal{Z})/\mathbb{R}_{+}^{\times}
$$
denote the conformal sphere bundle equipped with the Benoist-Hilbert geodesic flow
$$
\phi_{BH}^{t}: \mathbb{S}\mathcal{N}_{\Gamma}\rightarrow \mathbb{S}\mathcal{N}_{\Gamma}.
$$
As proved by Benoist \cite{BEN04}, the Benoist-Hilbert geodesic flow is a $C^{1, \alpha}$ topologically transitive Anosov flow which is $C^{2}$ if and only if there is an isometry $\mathcal{C}_{\Gamma}\simeq \mathbb{H}^{d}.$  
Among other results, we prove the following in Theorem \ref{THM BHG} in Section \ref{Examples}.

\begin{introthm}\label{introthm:BenoistHilbert}
The Benoist-Hilbert geodesic flow $\phi_{BH}^{t}: \mathbb{S}\mathcal{N}_{\Gamma}\rightarrow \mathbb{S}\mathcal{N}_{\Gamma}$ mixes exponentially for all H\"{o}lder observables with respect to every Gibbs equilibrium state with H\"{o}lder potential.
\end{introthm}

The Benoist-Hilbert geodesic flow has so far been out of reach of quantitative mixing results because of its low regularity (c.f.\ Liverani \cite{Li04} establishing exponential mixing for $C^4$ Anosov flows). Our proof trades the real analytic manifold $\mathbb{S}\mathcal{N}_{\Gamma}$ and the $C^{1,\alpha}$ flow $\phi^t_\mathrm{BH}$ for a Lipschitz submanifold of a real analytic contact Axiom A system whose stable and unstable foliations are the restriction of global real analytic foliations.

As opposed to Riemannian and Finsler geodesic flows, the differentiable dynamics of geodesic flows in pseudo-Riemannian geometry has received comparatively little attention, especially space-like geodesic flows whose physical meaning in relativistic models of the universe is far more uncertain than the trajectories of photons and massive particles described by the light-like and time-like geodesic flows. As a selection of  the relevant mathematical literature concerning Lorentzian and pseudo-Riemannian geometrical and dynamical studies in the context of projective Anosov subgroups we mention:  Mess \cite{mess2007} on the moduli space of globally hyperbolic maximally compact  Anti-de-Sitter (AdS) space-times,  Barbot-M\'{e}rigot \cite{barbot-merigot} on Anosov subgroups  and AdS space-times, Danciger-Gu\'{e}ritaud-Kassel \cite{dgk18,dgk18a} on $\bbH^{p,q}$-convex cocompactness, and Glorieux \cite{GLO17}, Glorieux-Monclair \cite{GM21} on pseudo-Riemannian critical exponents and Hausdorff dimension.

Concerning the pseudo-Riemannian symmetric space $\mathbb{H}^{p,q}$, Danciger-Gu\'{e}ritaud-Kassel \cite{dgk18} defined convex-cocompact subgroups $\Gamma<\SO(p,q+1)$ and constructed corresponding non-empty domains of proper discontinuity $\Omega_{\Gamma}^{\mathrm{DGK}}\subset \mathbb{H}^{p,q}.$ The space-like geodesic flow
$$
\phi^{t}: T^{1}\left(\Gamma\backslash \Omega_{\Gamma}^{\mathrm{DGK}}\right)\rightarrow T^{1}\left(\Gamma\backslash \Omega_{\Gamma}^{\mathrm{DGK}}\right)
$$
is incomplete in general, while the Axiom A system $\mathcal{M}_{\Gamma}$ we construct in Theorem \ref{THM A} satisfies $T^{1}(\Gamma\backslash\Omega_{\Gamma}^{\mathrm{DGK}})\subset \mathcal{M}_{\Gamma}$ and contains the saturation of $T^{1}(\Gamma\backslash\Omega_{\Gamma}^{\mathrm{DGK}})$ with respect to the foliation of $\mathcal{M}_{\Gamma}$ by flow lines.  The following is a sample of what we prove in Theorem \ref{THM CCHPQ} in Section \ref{Examples}.
\begin{introthm}\label{introthm:pseudoRiem}
The restriction of the (possibly incomplete) space-like geodesic flow 
$$
\phi^{t}: T^{1}\left(\Gamma\backslash \Omega_{\Gamma}^{\mathrm{DGK}}\right)\rightarrow T^{1}\left(\Gamma\backslash \Omega_{\Gamma}^{\mathrm{DGK}}\right)
$$
to the basic hyperbolic set $\mathcal{K}_{\Gamma}\subset T^{1}\left(\Gamma\backslash \Omega_{\Gamma}^{\mathrm{DGK}}\right)$ mixes exponentially for all H\"{o}lder observables with respect to every Gibbs equilibrium state with H\"{o}lder potential.
\end{introthm}

\subsection{On the proof of Theorems \ref{THM A} and \ref{SAMHYP}: domains of discontinuity}

Having presented our main results, it is worthwhile to examine how Theorem \ref{THM A} fits into the framework of the ongoing study of proper actions of projective Anosov subgroups $\Gamma<\SL(V)$.  

In \cite{GW12}, Guichard-Wienhard constructed discontinuity domains for actions on projective space $\bbP(V)$, and subsequently on other flag manifolds by an appropriate use of representation theory. This was succeeded by the work of Kapovich-Leeb-Porti \cite{KLP18} giving an intrinsic treatment of the $\Gamma$-action on flag manifolds, which utilized the combinatorial structure of Bruhat-Schubert cells.  The Kapovich-Leeb-Porti ideas were further abstracted away by Carvajales-Stecker \cite{CS23} who constructed examples of open $\Gamma$-invariant domains of proper discontinuity in a broad class of homogeneous spaces which includes the homogeneous space $\mathbb{L}$-appearing in Theorem \ref{THM A}. However, the domain they obtain is not equal to $\widetilde{\mathcal{M}}_{\Gamma}$, and in particular the quotient of their domain does not contain the basic hyperbolic set $\mathcal{K}_{\Gamma}\subset \mathcal{M}_{\Gamma}$ which is essential for all dynamical applications.

Theorem \ref{SAMHYP} is directly deduced from Theorem \ref{THM A} using the Hopf parametrization of the space $\mathbb L$, see Section \ref{sec:dynamicsonquotient}.

\subsection{On the proof of Theorems \ref{EXP MIX}, \ref{zeta fcn} and \ref{orbit counting}: Stoyanov's conditions}\label{methods of proof}
Next, we comment further on the logic of how our aforementioned dynamical results follow from the work of Stoyanov \cite{St11}.  As Stoyanov makes clear, the main result of \cite{St11} is strictly stronger than any of the applications we have listed in this introduction.  To wit, Stoyanov establishes quantitative spectral estimates for complex Ruelle transfer operators with respect to a particular Markov coding of an Axiom A flow satisfying some subtle geometro-dynamical hypotheses. That Theorems \ref{EXP MIX}, \ref{zeta fcn} and \ref{orbit counting}  follow from sufficiently strong quantitative spectral estimates on Ruelle transfer operators was already understood by Dolgopyat \cite{DOL98} and Dolgopyat-Pollicott \cite{DP98}, and appears in the earlier work of Parry-Pollicott \cite{PP83} and Pollicott \cite{POL85}.

The central technical work to achieve Theorems \ref{EXP MIX}, \ref{zeta fcn} and \ref{orbit counting} consists in verifying three geometro-dynamical hypotheses used in the work of  Stoyanov \cite{St11}.  Two of these conditions are now standard in dynamical systems and concern uniform Lipschitz regularity of the local holonomy maps along stable laminations and a quantitative non-integrability for the joint stable and unstable foliations.

The third condition of Stoyanov is called  regular distortion along unstable manifolds over the basic hyperbolic set $\mathcal{K}_{\Gamma}\subset \mathcal{M}_{\Gamma}$ from Theorem \ref{THM A}, and is reminiscent of the well-known Ahlfors-David regularity condition \cite{DAV88, DS93} in the analysis of metric spaces (e.g.\ fractals).    The discovery of an appropriate notion of quantitative self-similarity of the (often fractal) set $\mathcal{K}_{\Gamma}$ so that the appropriate Ruelle transfer operators submit to a Dolgopyat type analysis is a key innovation of Stoyanov \cite{St11}. In particular, we adapt arguments of Stoyanov \cite{St13} to verify the  regular distortion condition in the context of Theorem \ref{THM A}.

To summarize, the exponential mixing, zeta function spectral gap, and prime orbit asymptotics in Theorems \ref{EXP MIX}, \ref{zeta fcn}, \ref{orbit counting} are ultimately consequences of symbolic dynamics and the thermodynamic formalism, and are derived as corollaries of Stoyanov's spectral estimates for complex Ruelle transfer operators.  Meanwhile, the global meromorphic continuation of the zeta function in Theorem \ref{zeta fcn} and existence of (co)-resonant states in Theorem \ref{Resonant states} is a consequence of the work of Dyatlov-Guillarmou \cite{DG16, DG18} and Borns-Weil-Shen \cite{BWS21} which uses microlocal analysis and completely avoids the thermodynamic formalism.

\subsection{Beyond projective Anosov}

This paper is the first in a series where we construct and analyze Axiom A systems associated to Anosov subgroups, so that the powerful techniques of contemporary hyperbolic dynamical systems can be brought to bear in their study.  In particular, a forthcoming paper \cite{DMSII} will establish the analogues of Theorems \ref{THM A} and \ref{SAMHYP} for any Anosov subgroup $\Gamma<G$ of a semisimple Lie group $G$ without relying on representation theoretic arguments to reduce to the projective Anosov case.  Lest one think this is a question of taste, we make some arguments to the contrary below.

The central point is that for general Anosov $\Gamma<G,$ the refraction flow $\phi_{S}^{t}: \chi_{\Gamma}\rightarrow \chi_{\Gamma}$ belongs to a continuous family parameterized by an open convex cone in some vector space, so the appropriate flow space replacing $\bbL$ cannot be constructed using representation theory since integrality conditions for weights of finite-dimensional representations prohibit one from deforming the construction. Instead, we will have to dive deeper into the structure of semisimple Lie groups in order to establish bridges between Lie theoretic and differential geometric constructions of homogeneous spaces, so that the dynamical techniques presented in the current paper can be applied in this broader context.

Once it is established that every refraction flow space is embedded as a basic hyperbolic set in a real analytic Axiom A system, the dynamical consequences (e.g.\ exponential mixing, continuation of zeta functions, orbit counting, and Ruelle-Pollicott resonances) become open to attack via the identical strategy employed in this paper.

\subsection{Roadmap}

In the preliminary Section \ref{Background} we discuss background material concerning hyperbolic groups, projective Anosov subgroups, and Axiom A flows.  Readers familiar with any (or all) of these subjects should just briefly skim this section for the relevant definitions and notation.  In Section \ref{AnosovImpliesAxiomA}, we discuss the geometry of the $\mathrm{SL}(V)$-equivariant flow space $\mathbb{L}$ and
prove Theorems \ref{THM A} and \ref{SAMHYP}.  Section \ref{ExponentialMixing} turns to the dynamical consequences of Theorem \ref{THM A}, and in particular contains the technical arguments verifying the conditions of Stoyanov \cite{St11}.    We include an abbreviated account of Ruelle transfer operators and the precise technical estimates implied by the verification of Stoyanov's conditions, and prove Theorem \ref{EXP MIX}. The detailed and more general versions of Theorems \ref{zeta fcn}, \ref{orbit counting} and \ref{Resonant states} are proved in Section \ref{sec:Resonances}. The final examples Section \ref{Examples} contains an exposition of strictly convex real projective manifolds and convex cocompact pseudo-Riemannian manifolds with the detailed versions of Theorems \ref{BHG RES} and \ref{THM CCHPQ}, in particular explaining the exponential mixing of the Benoist-Hilbert geodesic flow and space-like geodesic flow in each setting.

\subsection*{Acknowledgements} B.~D.\ has received funding from the Deutsche Forschungsgemeinschaft (German Research Foundation, DFG) through the Priority Program (SPP) 2026 ``Geometry at Infinity''. B.~D. thanks Tobias Weich for his  suggestion to apply tools from the study of smooth hyperbolic (multi-)flows and microlocal methods in the field of Anosov representations. D.~M.\  was partially supported by the ANR JCJC grant GAPR (ANR-22-CE40-0001, Geometry and Analysis in the Pseudo-Riemannian setting). D.~M.\ is grateful to Yves Benoist, Carlos Matheus and Andrés Sambarino for encouraging conversations about this work.  A.~S.\ is grateful to Thi Dang and Beatrice Pozzetti for their encouragement and expertise provided during the early development of the ideas that resulted in this collaboration.  A.~S.\ also thanks Nicolas Tholozan for pointing out a previous misunderstanding of the properness criterion of Benoist-Kobayashi and some additional important comments regarding proper actions.  Finally, A.~S.\ is thankful to Martin Bridgeman and the mathematics department at Boston College, 
where he is currently an (unaffiliated) visiting scholar, for providing wonderful working conditions and interactions during the later stages of this work.

\section{Preliminaries and Background}\label{Background}

\subsection*{Notation} Let $V$ be a finite-dimensional real vector space of dimension $d>0$. Throughout this paper, we use the notation $[v]\in \bbP(V)$ for the projection of a non-zero vector $v\in V$. We denote by $\pi$ the projection map
\[\pi:\map{V\setminus\{0\}}{\bbP(V)}{v}{\left[v\right].} \]
Occasionally, we also denote other ``canonical'' projections by $\pi$ if the meaning is clear from the context.
 In the case of a product $V_1\times V_2$, we  denote by $[v_1:v_2]\in\bbP(V_1\times V_2)$ the projection of $(v_1,v_2)\in V_1\times V_2\setminus\{(0,0)\}$. We identify $\bbP(V)$ with the Grassmannian manifold $\cG_1(V)$, identifying $[v]\in\bbP(V)$ with $\R\cdot v\in\cG_1(V)$ for a non-zero vector $v\in V$. In the case of the dual space $V^*$, we will regularly use the natural identification between $\bbP(V^*)$ and the Grassmannian manifold $\cG_{d-1}(V)$ of hyperplanes, identifying $[\alpha]\in \bbP(V^*)$ with $\ker\alpha\in \cG_{d-1}(V)$ for a non-zero linear form $\alpha\in V^*$.

Two elements $[v]\in \mathbb{P}(V)$,  $[\alpha]\in \mathbb{P}(V^\ast)$ are \emph{transverse}, written $[v]\pitchfork [\alpha]$, if $\alpha(v)\neq 0$. In the language of Grassmannians, transversality $\ell\pitchfork H$ simply means that $\ell\not\subset H$. We then consider the open subset of transverse pairs in $\mathbb{P}(V)\times \mathbb{P}(V^{\ast})$
$$
\mathbb{P}(V)\overset{\pitchfork}{\times} \mathbb{P}(V^{\ast}):=\{(\ell, H)\in \mathbb{P}(V)\times \mathbb{P}(V^{\ast}) \ | \ \ell\pitchfork H \}.
$$

For $g\in \SL(V)$, we use the notation $g\cdot \alpha:=\alpha\circ g^{-1}$ for the left action of $\SL(V)$ on $V^*$. Finally, given a vector $v\in V$, we  denote by $\iota_v\in (V^*)^*$ the evaluation form $\alpha\mapsto \alpha(v)$.

\subsection{Dynamics}\label{sec:dynamics}

Let $\M$ be a metrizable topological space and $\phi^t:\M\to \M$ a continuous flow defined for all $t\in \R$ which has no fixed points. 

\begin{definition} \label{def - non-wandering, periodic, trapped}~
\begin{itemize}
\item The \emph{non-wandering set} ${\mathcal{NW}}(\phi^t)$ of the flow  $\phi^t$ is the set of all  points $x\in \M$ for which there are  sequences $x_N\to x$ in $\M$ and $t_N\to +\infty$ in $\R$ such that $\phi^{t_N}(x_N)\to x$. 
\item The \emph{trapped set} ${\mathcal T}(\phi^t)$ of the flow  $\phi^t$  is the set of all points $x\in \M$ whose $\phi^t$-orbits are relatively compact in $\M$.
\item The set $\mathcal{P}(\phi^t)$ of \emph{periodic points}  of the flow  $\phi^t$  consists of all points $x\in \M$ for which there exists $T>0$ with $\phi^T(x)=x$.
\end{itemize}
\end{definition}

\begin{definition}Let $\mathcal K\subset \M$ be a compact $\phi^t$-invariant set and $E$ a continuous vector bundle over $\mathcal K$  equipped with a continuous flow $\phi_E^t:E\to E$ lifting $\phi^t$ over $\mathcal K$. Then $\phi_E^t$ is \emph{uniformly contracting} (resp.\ \emph{expanding}) on $E$ if for some (hence any) continuous bundle norm $\norm{\cdot}$ on $E$ there are constants $C,c>0$ such that for all $p\in \mathcal K$ and all $v\in E_p$ one has
\bqn
\norm{\phi^t_E(v)}_{\phi^t(p)}\leq Ce^{-c|t|}\norm{v}_p
\eqn
for all $t\geq 0$ (resp.\ $t\leq 0$).
\end{definition}
For the remainder of Section \ref{sec:dynamics}, suppose now that $\M$ is a smooth manifold and $\phi^t$ is a smooth flow with generating vector field $X:\M\to T\M$. Note that the latter is nowhere vanishing since we assume that $\phi^t$ has no fixed points.

\begin{definition} \label{def:hyperbolicset}A compact $\phi^t$-invariant set $\mathcal K\subset \M$ is called \emph{hyperbolic} for the flow $\phi^t$ if the restriction of the tangent bundle $T\M$ to $\mathcal K$ admits a Whitney sum decomposition
\bq
T\M|_{\mathcal K}=E^0\oplus E^{\mathrm{s}}\oplus E^{\mathrm{u}}\label{eq:bundlesplittingstableunstableneutral}
\eq
where $E^0_p=\R X(p)$ for all $p\in \mathcal K$ and $E^{\mathrm{s}}$, $E^{\mathrm{u}}$ are $d\phi^t$-invariant continuous subbundles such that $d\phi^t$ is uniformly contracting (resp.\ expanding) on $E^{\mathrm{s}}$ (resp.\ $E^{\mathrm{u}}$).
\end{definition}

\begin{definition}[{c.f.~\cite[§II.5 (5.1)]{smale67}}]The flow $\phi^t$ is an \emph{Axiom A} flow if the non-wandering set ${\mathcal{NW}}(\phi^t)$ is compact and hyperbolic and coincides with the closure  in $\M$ of the set of periodic points $\mathcal{P}(\phi^t)$.
\end{definition}
Note that the manifold $\mathcal M$ itself is not assumed to be compact.
\begin{definition}A compact $\phi^t$-invariant set $\mathcal K\subset \M$ is \emph{locally maximal} for the flow $\phi^t$ if there is a neighborhood $\mathcal U\subset \M$ of $\mathcal K$ such that
\[
\mathcal K=\bigcap_{t\in \R}\phi^t(\mathcal U).
\]
\end{definition}

\begin{definition}\label{def:basicset}A hyperbolic set $\mathcal K\subset \M$ for the flow $\phi^t$  is \emph{basic} if it is locally maximal for $\phi^t$, the flow $\phi^t|_{\mathcal K}$ is topologically transitive (i.e., $\mathcal K$ contains a dense $\phi^t$-orbit), and $\mathcal K$ is the closure in $\M$ of the set of periodic points of $\phi^t|_{\mathcal K}$.
\end{definition}
Smale originally demanded in \cite[§II.5]{smale67} that a basic set should not consist of a single closed orbit, while other standard references such as Ruelle-Bowen's \cite{BR75} do not impose this condition. We follow the convention of the latter and do allow basic sets to be  single closed orbits.
\begin{prop}[Spectral decomposition, {\cite[§II.5 Thm.~ 5.2]{smale67}, c.f.\ \cite[Res.~3.9]{Bowen}}]If $\phi^t$ is an Axiom A flow, then its non-wandering set is a finite disjoint union of basic hyperbolic sets.
\end{prop}
\begin{definition}If $\phi^t$ is an Axiom A flow, we say that $\phi^t$ is \emph{simple} if its non-wandering set is a basic hyperbolic set, i.e., its spectral decomposition consists of only one set.
\end{definition}

\subsection{Hyperbolic groups}

Let $\Gamma$ be a finitely generated group which is hyperbolic \cite{Gr}; other common synonymous terms are \emph{word-hyperbolic} or \emph{Gromov-hyperbolic}. We denote by $\partial_\infty\Gamma$ its Gromov boundary.  The following records some essential properties of $\Gamma$.

\begin{prop}[{\cite{Gr,GdH90}}]\label{prop hyperbolic groups elementary properties}
\begin{enumerate}[label=(\arabic*),ref=\emph{(\arabic*)}]
\item There is a metrisable and compact topology on $\Gamma\cup\partial_\infty\Gamma$  for which $\Gamma$ is open, discrete and dense.
\item \label{propenumi hyp groups continuous action} The left action $\Gamma\curvearrowright\Gamma\cup\partial_\infty\Gamma$ is continuous.
\item \label{propenumi hyp groups minimal action} All $\Gamma$-orbits in $\partial_\infty\Gamma$ are dense.
\item \label{propenumi hyp groups boundary fixed points of elements} Every infinite order element $\gamma\in \Gamma$ has two distinct fixed points $\gamma_+=\lim_{n\to +\infty}\gamma^n\in\partial_\infty\Gamma$ and $\gamma_-=\lim_{n\to -\infty}\gamma^n\in\partial_\infty\Gamma$.
\item \label{propenumi hyp groups density poles} The set of \emph{poles}, i.e. pairs $(\gamma_+,\gamma_-)\in\partial_\infty\Gamma^2$ for some infinite order element $\gamma\in\Gamma$ is dense in $\partial_\infty\Gamma^2$.
\item \label{propenumi hyp groups dense orbits product} The diagonal action $\Gamma\curvearrowright\partial_\infty\Gamma\times\partial_\infty\Gamma$ is topologically transitive.
\item \label{propenumi hyp groups convergence action} If $\gamma_k\to \gamma_+\in\partial_\infty\Gamma$ and  $\gamma_k^{-1}\to \gamma_-\in\partial_\infty\Gamma$, then $\gamma_k\cdot t\to \gamma^+$ for any $t\in \Gamma\cup\partial_\infty\Gamma\setminus\{\gamma^-\}$, uniformly on compact subsets.
\end{enumerate}
\end{prop}

\begin{lem} \label{lem hyp groups convergence limit points}
Let $\gamma_k\in \Gamma$ be a sequence of infinite order elements such that $\gamma_k\to \gamma_+\in \partial_\infty\Gamma$ and $\gamma_k^{-1}\to \gamma_-\in \partial_\infty\Gamma$ with $\gamma_+\neq\gamma_-$. Then the sequence of attracting fixed points $\gamma_{k,+}=\lim_{n\to+\infty}\gamma_k^n\in\partial_\infty\Gamma$ converges to $\gamma_+$.
\end{lem}

\begin{proof}
Let $U\subset \partial_\infty\Gamma$ be an open subset containing $\gamma_+$ whose closure does not contain $\gamma_-$. Then for $k$ large enough one has $\gamma_k\cdot U\subset U$ by Proposition \ref{prop hyperbolic groups elementary properties}\ref{propenumi hyp groups convergence action}, thus $\gamma_{k,+}\in U$.
\end{proof}

\begin{lem}\label{lem distinct limit points}
Let $\delta_N\in \Gamma$ be an unbounded sequence of infinite order elements. There exist an increasing sequence $N_k\in \N$, an element $f\in \Gamma$ and distinct points $\gamma_+,\gamma_-\in\partial_\infty\Gamma$ such that the sequence $\gamma_k=f\delta_{N_k}$ satisfies $\lim \gamma_k=\gamma_+$ and $\lim \gamma_k^{-1}=\gamma_-$.
\end{lem}
\begin{proof}
First consider a subsequence $\delta_{N_k}$ that has limit points $\delta_+=\lim \delta_{N_k}\in \partial_\infty\Gamma$ and $\delta_-=\lim \delta_{N_k}^{-1}\in\partial_\infty\Gamma$. If $\delta_+\neq \delta_-$, we can just choose $f=e$. If $\delta_+=\delta_-$, consider an element $f\in \Gamma$ such that $f\cdot \delta_-\neq \delta_-$ (such an element always exists unless $\Gamma$ is virtually cyclic, in which case we necessarily have $\delta_+\neq\delta_-$ so we can take $f=e$), so that $\gamma_k=f\delta_{N_k}$ satisfies $\lim \gamma_k= f\cdot \delta_-$ (Proposition \ref{prop hyperbolic groups elementary properties} \ref{propenumi hyp groups continuous action}) and $\lim \gamma_k^{-1}=\delta_-$ (Proposition \ref{prop hyperbolic groups elementary properties} \ref{propenumi hyp groups convergence action}).
\end{proof}

The following definition generalizes the concept of the geodesic flow on the unit tangent bundle of a compact negatively curved Riemannian manifold, motivated by the fact that the fundamental group of such a manifold is hyperbolic. 
\begin{definition}[{\cite[Sec.~8.3]{Gr}, \cite[Thm.~60]{Mineyev2005}}] \label{def - geodesic flow gromov hyperbolic group}
A \emph{Gromov-geodesic flow} of $\Gamma$ is a proper hyperbolic metric space $\widehat{\Gamma}$ endowed with a fixed-point free flow $(\Phi^t)_{t\in\R}$, an isometric involution $\iota$, and an isometric $\Gamma$-action with the following properties:
\begin{enumerate}
\item The $\Gamma$-action commutes with $\iota$ and $\Phi^t$.
\item The involution $\iota$ anti-commutes with $\Phi^t$, i.e., $\iota\circ \Phi^t=\Phi^{-t}\circ \iota$.
\item The orbit maps $\Gamma\to \widehat{\Gamma}$ are quasi-isometries. In particular, the $\Gamma$-action on $\widehat{\Gamma}$ is properly discontinuous and cocompact, and there is a $\Gamma$-equivariant homeomorphism $\partial_\infty \widehat{\Gamma}\cong\partial_\infty\Gamma$. The latter is canonical in the sense that it is independent of the choice of the $\Gamma$-orbit.
\item The orbit maps $\R\to \widehat{\Gamma}$ of the flow $\Phi^t$ are quasi-isometric embeddings.
\item The $\Gamma$-equivariant map  $x\mapsto (x_-:=\lim_{t\to -\infty}\Phi^t(x),x_+:=\lim_{t\to +\infty}\Phi^{-t}(x))$ induces a homeomorphism from $\widehat{\Gamma}/\R$ to $\partial_\infty\Gamma^{(2)}:=\partial_\infty\Gamma\times \partial_\infty\Gamma\setminus \mathrm{diag}$. In particular, there is a (non-canonical) $\Gamma$-equivariant homeomorphism $\widehat \Gamma\cong \partial_\infty \Gamma^{(2)}\times\R$.
\end{enumerate}
\end{definition}
By \cite[Thm.~ 8.3.C]{Gr} there exists a Gromov-geodesic flow of $\Gamma$, unique up to $\Gamma$-equivariant involution-preserving quasi-isometries mapping flow orbits to flow orbits. 

%
%
%

\subsection{Anosov subgroups}
Let $V$ be a finite-dimensional real vector space of dimension $d>1$.

\begin{definition} Let $g\in \SL(V)$. We denote by
\[ \lambda_1(g)\geq\lambda_2(g)\geq\cdots\geq\lambda_{d}(g)\]
the logarithms of the moduli of the (complex) eigenvalues of $g$ (with repetitions). 
\end{definition}

Note that any $g\in \SL(V)$ satisfies $\sum_{i=1}^{d}\lambda_i(g)=0$, in particular $\lambda_1(g)\geq0$. 

\begin{definition} \label{def projective Anosov subgroup} A finitely generated subgroup $\Gamma<\SL(V)$ is called \emph{projective Anosov} if it is Gromov-hyperbolic and there are constants $c,c'>0$ such that
\[ \lambda_1(\gamma)-\lambda_2(\gamma)\geq c\vert\gamma\vert_\infty-c' \]

for any $\gamma\in \Gamma$.
\end{definition}

Recall that the stable length is defined as $\vert\gamma\vert_\infty=\lim_{n\to+\infty}\frac{\vert\gamma^n\vert}{n}$, where $|\cdot|$ is the word length with respect to some finite generating set of $\Gamma$, the choice of which is irrelevant to the above condition. Definition \ref{def projective Anosov subgroup} is taken from \cite{KasselPotrie}, in which it is shown to be equivalent to the original definition of a projective Anosov subgroup \cite{labourie,GW12}, in particular it implies the following:
\begin{prop} \label{prop boundary maps}
If $\Gamma<\SL(V)$ is projective Anosov, there exists a unique pair of $\Gamma$-equivariant maps $\xi:\partial_\infty\Gamma\to \bbP(V)$ and $\xi^*:\partial_\infty\Gamma\to \bbP(V^*)$ such that:
\begin{enumerate}[label=(\arabic*),ref=\emph{(\arabic*)}]
\item \label{propenumi boundary maps equivariant homeomorphisms} The map $\xi$ (resp. $\xi^*$) is a bi-Hölder homeomorphism onto its image.
\item \label{propenumi boundary maps transversality} For $t,s\in \partial_\infty\Gamma$, one has $\xi(t)\subset\xi^*(s)$ if and only if $t=s$.
\item \label{propenumi boundary maps elements} For any $\gamma\in \Gamma$ of infinite order, $\lambda_1(\gamma)>0$ and $e^{\lambda_1(\gamma)}$ is the modulus of a unique eigenvalue of $\gamma$. This eigenvalue is real and simple, and the corresponding eigenspace of $\gamma$ is $\xi(\gamma_+)$. The eigenspace  of the inverse transpose $(\alpha\mapsto \gamma\cdot\alpha)\in\SL(V^*)$ for the eigenvalue of modulus $e^{-\lambda_1(\gamma)}$ is $\xi^*(\gamma_-)$.
\end{enumerate}
\end{prop}
\begin{definition} Let $\Gamma<\SL(V)$ be projective Anosov. The maps $\xi,\xi^*$ in Proposition \ref{prop boundary maps} are called the \emph{limit maps} of $\Gamma$. The \emph{limit set} of $\Gamma$ is $\Lambda_\Gamma:=\xi(\partial_\infty\Gamma)\subset \bbP(V)$, the \emph{dual limit set} is $\Lambda_\Gamma^*:=\xi^*(\partial_\infty\Gamma)\subset \bbP(V^*)$.
\end{definition}

\begin{lem} \label{lem top eigenvalue diverges} 
Let $\Gamma<\SL(V)$ be a torsion-free projective Anosov subgroup. If $\gamma_k\in\Gamma$ is a sequence with distinct boundary limits $\gamma_+=\lim\gamma_k\in \partial_\infty\Gamma$ and $\gamma_-=\lim\gamma_k^{-1}\in\partial_\infty\Gamma$, then $\lambda_1(\gamma_k)\to \infty$.
\end{lem}
\begin{proof}
For any $R>0$, there are finitely many conjugacy classes $[\gamma]$ in $\Gamma$ with $\lambda_1(\gamma)\leq R$ \cite[Prop.~2.8]{BCLS}. It means that we only need to check that the sequence $\gamma_k$ cannot belong to a finite collection of conjugacy classes. This is ruled out by the condition that $\gamma_+\neq \gamma_-$.  Indeed, given any $\eta>0$, a conjugacy class $C\subset \Gamma$ can  only contain finitely many elements $\delta\in C$ with $d_\infty(\delta_+,\delta_-)\geq \eta$ (where $d_\infty$ is a distance on $\partial_\infty\Gamma$ defining the Gromov topology), see \cite[Prop.~2.39]{GMT}. In view of Lemma \ref{lem hyp groups convergence limit points}, the claim follows.
\end{proof}

The following result describes the appropriate convergence property for the action of a projective Anosov subgroup on the projective and dual projective spaces (see Theorem 1.1 and Definitions 4.2, 4.25 in \cite{KLP17}).

\begin{prop} \label{prop conv action} Let $\Gamma<\SL(V)$ be projective Anosov. 
For any unbounded sequence $\gamma_N\in \Gamma$ with boundary limit points $\gamma_+=\lim_{N\to +\infty}\gamma_N\in\partial_\infty\Gamma$ and  $\gamma_-=\lim_{N\to +\infty}\gamma_N^{-1}\in\partial_\infty\Gamma$, there is a subsequence $\gamma_{N_k}$ for which the actions on $\bbP(V)$ and $\bbP(V^*)$ obey the following dynamics as $k\to +\infty$:
\begin{enumerate}
\item $\gamma_{N_k}\cdot \ell\rightarrow \xi(\gamma_{+})$ for all $\ell\in\mathbb{P}(V)$ with $\ell\pitchfork \xi^{\ast}(\gamma_{-})$;
\item $\gamma_{N_k}^{-1}\cdot \ell\rightarrow \xi(\gamma_{-})$ for all $\ell\in\mathbb{P}(V)$ with $\ell\pitchfork \xi^{\ast}(\gamma_{+})$;
\item $\gamma_{N_k}\cdot H\rightarrow \xi^{\ast}(\gamma_{+})$ for all $H\in\mathbb{P}(V^\ast)$ with $\xi(\gamma_{-})\pitchfork H$;
\item $\gamma_{N_k}^{-1}\cdot H\rightarrow \xi^{\ast}(\gamma_{-})$ for all $H\in\mathbb{P}(V^\ast)$ with $\xi(\gamma_{+})\pitchfork H$;
\end{enumerate}
and all these convergences are uniform on compact subsets.
\end{prop}

\subsection{Sambarino's refraction flow}

In this section we borrow heavily from the discussion in \cite{SAM14}.  Fix a Euclidean norm $\lVert \cdot \rVert$ on the real vector space $V$ and define the Hopf-Busemann-Iwasawa (HBI) cocycle 
\bqn
\mathscr{H}: \SL(V)\times \mathbb{P}(V)\overset{\pitchfork}\times \mathbb{P}(V^{*})\longrightarrow \mathbb{R}
\eqn
by the real analytic homogeneous expression
$$
\mathscr{H}(g, [v],[\alpha])=\log\left(\frac{\lVert g\cdot v\rVert}{\lVert v\rVert}\right).
$$
The HBI cocycle defines a real analytic $\SL(V)$-action on the trivial real line bundle
\bq\label{triv equi}
\mathbb{P}(V)\overset{\pitchfork}\times \mathbb{P}(V^{*})\times \mathbb{R}\rightarrow \mathbb{P}(V)\overset{\pitchfork}\times \mathbb{P}(V^{*})
\eq
via
\bq	\label{equi action}
g\cdot ([v],[\alpha], \tau)=\left(g\cdot [v], g\cdot [\alpha], \tau+ \mathscr{H}(g, [v],[\alpha])\right).
\eq
The $\SL(V)$-action \eqref{equi action} covers the diagonal action on $\mathbb{P}(V)\overset{\pitchfork}\times \mathbb{P}(V^{*})$ turning \eqref{triv equi} into an $\SL(V)$-equivariant principal $\mathbb{R}$-bundle (i.e. real affine line bundle).  If $\Gamma<\SL(V)$ is projective Anosov and 
$$
\partial_{\infty}\Gamma^{(2)}:=\{(x,y)\in \partial_{\infty}\Gamma\times \partial_{\infty}\Gamma \ | \ x\neq y\},
$$
we combine the Anosov boundary maps $\xi,\xi^\ast$ to the map
\bq
\sigma:\map{\partial_\infty\Gamma^{(2)}}{\trprod}{(x,y)}{(\xi(x),\xi^\ast(y))}\label{eq:sigmamap}
\eq
and consider the commutative diagram
\begin{center}
\begin{tikzcd}
\partial_{\infty}\Gamma^{(2)}\times \R \arrow[d] \arrow[r, "\sigma\times\mathrm{id}"]
&[1.5cm] \mathbb{P}(V)\overset{\pitchfork}\times \mathbb{P}(V^{*})\times \mathbb{R} \arrow[d] \\
\partial_{\infty}\Gamma^{(2)} \arrow[r, "\sigma"]
& \mathbb{P}(V)\overset{\pitchfork}\times \mathbb{P}(V^{*}).
\end{tikzcd}
\end{center}
Since both horizontal arrows are $\Gamma$-equivariant, the $\SL(V)$-action given by \eqref{equi action} induces a $\Gamma$-action on $\partial_{\infty}\Gamma^{(2)}\times \R$ making the trivial bundle
$$
\partial_{\infty}\Gamma^{(2)}\times \R \rightarrow \partial_{\infty}\Gamma^{(2)}
$$
a $\Gamma$-equivariant real affine line bundle.  The following result is proved in \cite{SAM14} and \cite{BCLS}.
\begin{prop}\label{prop: ref flow}
Suppose $\Gamma<\SL(V)$ is projective Anosov.  The $\Gamma$-action on $\partial_{\infty}\Gamma^{(2)}\times \R$ induced by the HBI cocycle is properly discontinuous and cocompact.
\end{prop}
This leads to the following definition (see \cite{SAM24}).
\begin{definition}\label{sam flow}
Suppose $\Gamma<\SL(V)$ is projective Anosov.  Sambarino's \emph{refraction flow space} is the quotient
$$
\chi_{\Gamma}:=\Gamma \backslash \left(\partial_{\infty}\Gamma^{(2)}\times \R\right)
$$
with continuous complete flow induced by $\phi_{S}^{t}(x, y, \tau)=(x, y,\tau + t)$, called the \emph{refraction flow} $\phi_{S}^{t}$.
\end{definition}
By Proposition \ref{prop: ref flow}, the refraction flow space is Hausdorff and compact.  Observe that it satisfies all but one of the properties of a Gromov-Mineyev flow space for $\Gamma$, with the exception that there does not exist an involution $\iota$ anti-commuting with the flow. This is a consequence of the fact that the space of transverse pairs
$$
\mathbb{P}(V)\overset{\pitchfork}\times\mathbb{P}(V^{*})
$$ 
does not admit a natural order reversing involution since $\mathbb{P}(V)$ and $\mathbb{P}(V^{*})$ are not $\SL(V)$-equivariantly diffeomorphic.  In particular, the refraction flow is not reversible and the periods corresponding to non-trivial $\gamma, \gamma^{-1}\in \Gamma$ are distinct (see Lemma \ref{lem periodic points in discontinuity set}).                       

In the next section, we will prove Theorem \ref{THM A} and \ref{SAMHYP} showing that the  refraction flow is conjugate via a H\"{o}lder homeomorphism to the restricted flow on a basic hyperbolic set in a real analytic contact Axiom A dynamical system.

\section{Projective Anosov Geometry}\label{AnosovImpliesAxiomA}

\subsection{The flow space and its geometry}

Let $V$ be a finite-dimensional real vector space with $\dim V>1$ and linear dual $V^{\ast}$. We consider the open set $\bbL\subset \bbP(V\times V^*)$ defined by
\[ \bbL :=\set{[v:\alpha]\in\bbP(V\times V^*)}{\alpha(v)> 0}. \]  
This real analytic manifold comes with a real analytic action of $\R$ defined by
\[ \phi^t([v:\alpha])=[e^tv:e^{-t}\alpha],\quad \forall \ t\in\R,~\forall \ [v:\alpha]\in \bbL.\]
The flow $\phi^{t}$ commutes with the action $\SL(V)\curvearrowright\bbL$ given by $g\cdot [v:\alpha]=[g\cdot v: g\cdot\alpha]$.

In order to describe the geometry of $\bbL$, it is practical to work with the affine quadric hypersurface
$$
\mathbb{L}_1:=\{(v, \alpha)\in V\times V^{\ast} \ | \ \alpha(v)=1\}
$$
which is an $\SL(V)$-equivariant double cover of $\bbL$ through the restriction of the projection $\pi: \left(V\times V^{*}\right)\setminus\{(0,0)\}\to \bbP(V\times V^*)$. The description of the tangent space
\[ T_{(v,\alpha)}\bbL_1=\set{(w,\beta)\in V\times V^*}{\alpha(w)+\beta(v)=0} \]
shows that there is a natural splitting 
\[ T\bbL_1=E^\mathrm{u}\oplus E^0\oplus E^\mathrm{s} \]
where 
\begin{align*}\begin{split}
E^\mathrm{u}_{(v,\alpha)}&=\ker\alpha\times\{0\}, \\
E^\mathrm{s}_{(v,\alpha)}&=\{0\}\times \ker \iota_v, \\
E^0_{(v,\alpha)}&=\R\cdot (v,-\alpha),
\end{split}\qquad \forall (v,\alpha)\in \bbL_1.
\end{align*} 
 These distributions project to an $\SL(V)$-equivariant splitting of the tangent bundle $T\bbL:$
\begin{equation} \label{eqn splitting TL} T\bbL=E^\mathrm{u}\oplus E^0\oplus E^\mathrm{s}
\end{equation}
where $E^\mathrm{i}_{[v:\alpha]}=d_{(v,\alpha)}\pi(E^\mathrm{i}_{(v,\alpha)})$ for $\mathrm{i}\in \{\mathrm{u},0,\mathrm{s}\}$ and $(v,\alpha)\in\bbL_1$. Note that this decomposition is related to the flow $\phi^t$ by the formula $E^0_{[v:\alpha]}=\R\cdot \left.\frac{d\,}{dt}\right\vert_{t=0}\phi^t([v:\alpha])$.

 The manifold $\bbL$ carries a real analytic pseudo-Riemannian metric invariant under the action of $\SL(V)$ given by the projection under $\pi$ of the antipodal-invariant pairing
\begin{equation}\label{eqn pseudo-riemannian metric} \left( (w,\beta),(w',\beta')\right)_{(v,\alpha)} = \beta(w')+\beta'(w) 
\end{equation}
for $(w,\beta),(w',\beta')\in T_{(v,\alpha)}\bbL_1$. It has signature $(d,d-1)$ where $d=\dim V$.

The manifold $\bbL_1$ also carries a real analytic contact form given by the restriction of the tautological $1$-form of the cotangent bundle $T^\ast V=V\times V^\ast$:
\[ \tau_{(v,\alpha)}(w,\beta)=\alpha(w)=-\beta(v) \]
at any $(v,\alpha)\in \bbL_1$ and $(w,\beta)\in T_{(v,\alpha)}\bbL_1$. It descends to a real analytic contact form $\tau$ on $\bbL$ whose associated contact structure is
\[ \ker\tau=E^\mathrm{u}\oplus E^\mathrm{s},\]
and whose Reeb vector field integrates to the real analytic flow $\phi^t$. The space $\bbL$ fibers over $\trprod$ through the projection
\begin{equation}\label{eqn projection of L onto transverse product}
 p:\map{\bbL}{\trprod}{\left[v:\alpha\right]}{(\left[v\right],\left[\alpha\right])}.
 \end{equation}
The fibers of $p$ are equal to the orbits of the flow $\phi^t$, giving $\bbL$ the structure of a real analytic principal $\R$-bundle (i.e. a real affine line bundle).

\vspace{\baselineskip}

In the rest of this section, we fix a torsion-free projective Anosov subgroup $\Gamma<\SL(V)$ with limit maps $\xi:\partial_\infty\Gamma\to \bbP(V)$, $\xi^*:\partial_\infty\Gamma\to \bbP(V^*)$.

\subsection{Linear dynamics of projective Anosov subgroups}
Before we can prove our result on linear dynamics (Lemma \ref{lem prop criterion}), we will need a technical result.  
\begin{lem} \label{lem: divergence in tautological bundles}
 Let $(\ell_k,H_k)\in \trprod$ be a sequence  converging  to $(\ell,H)\in\trprod$. Given sequences $v_k\in \ell_k$ and $w_k\in H_k$, if $v_k\to \infty$, then $v_k+w_k\to \infty$.
\end{lem}
\begin{proof}
The result is straightforward if $w_k$ is bounded, so we may assume that $w_k\to \infty$. Consider an inner product $\pscal{\cdot}{\cdot}$ on $V$ such that $H=\ell^\perp$. Up to a subsequence, assume that $\frac{v_k}{\norm{v_k}}\to \overline v\in \ell$ and $\frac{w_k}{\norm{w_k}}\to \overline w\in H$. Then $\pscal{\overline v}{\overline w}=0$, and 
\begin{align*}
\norm{v_k+w_k}^2 &= \norm{v_k}^2 +\norm{w_k}^2+2\pscal{v_k}{w_k}\\
&=  \norm{v_k}^2 +\norm{w_k}^2+ \mathrm o(\norm{v_k}\norm{w_k})\to +\infty.
\end{align*}
This proves that the sequence $v_k+w_k$ has no bounded subsequence, hence $v_k+w_k\to \infty$.
\end{proof}
It is important to notice that contraction for the action on the projective space translates to expansion for the linear action, as a contracting fixed point in $\bbP(V)$ is the eigendirection for the largest eigenvalue.
\begin{lem}\label{lem prop criterion} 
Let $\gamma_k\in\Gamma$ be a sequence admitting distinct boundary limits $\gamma_+=\lim \gamma_k\in\partial_\infty\Gamma$ and $\gamma_-=\lim \gamma_k^{-1}\in\partial_\infty\Gamma$. For any sequence $v_k\rightarrow v\in V\setminus \{0\}$ such that $[v]\pitchfork \xi^{\ast}(\gamma_{-})$, one has $\gamma_k\cdot v_k\rightarrow \infty$ as $k\rightarrow \infty.$  
\end{lem}
\begin{proof}
Let $\gamma_{k,\pm}\in \partial_\infty \Gamma$ be the attracting and repelling points of $\gamma_k$, and can consider the decomposition 
\[ v_k=v_{k,+}+v_{k,-}\in \xi(\gamma_{k,+})\oplus \xi^*(\gamma_{k,-}).\]
Since $\gamma_\pm=\lim_{k\to +\infty} \gamma_{k,\pm}$ (Lemma \ref{lem hyp groups convergence limit points}), we also have $v_{k,+}\to v_+$ and $v_{k,-}\to v_-$ where $v=v_++v_-\in \xi(\gamma_+)\oplus\xi^*(\gamma_-)$. The assumption that $[v]\pitchfork \xi^{\ast}(\gamma_{-})$ means that $v_+\neq 0$, hence $\gamma_kv_{k,+}=e^{\lambda_1(\gamma_k)}v_{k,+}\to \infty$ by Lemma \ref{lem top eigenvalue diverges}, and $\gamma_kv_k\to \infty$ by Lemma \ref{lem: divergence in tautological bundles}.
\end{proof}

\subsection{The discontinuity domain in \texorpdfstring{$\bbL$}{L}}

The limit maps $\xi,\xi^*$ in conjunction with the projection $p:\bbL\to \trprod$ introduced in \eqref{eqn projection of L onto transverse product} will allow us to construct a domain of proper discontinuity for the action of $\Gamma$ on $\bbL$.
\begin{definition}  We introduce the following subset of $\trprod$:
\[
\Omega_{\Gamma}:=\{(\ell, H)\in \trprod \,|\, \forall\, x\in \partial_{\infty}\Gamma: \ell\pitchfork \xi^{\ast}(x) \text{ or } \xi(x)\pitchfork H\}.
\]  
Then we define
$$
\widetilde{\M}_{\Gamma}:=p^{-1}(\Omega_{\Gamma})\subset \L.
$$
\end{definition}
We are now ready to prove Theorem \ref{THM A} from the introduction: we will split this into a series of smaller results, noting what we have proved as we go.  The following lemma proves the preliminary statements of Theorem \ref{THM A} concerning the openness and invariance properties of $\widetilde{\mathcal{M}}_{\Gamma}\subset \mathbb{L}.$  
\begin{lem}\label{lem:OmegatildeMopen}
 The sets $\Omega_{\Gamma}\subset  \mathbb{P}(V)\overset{\pitchfork}{\times} \mathbb{P}(V^{\ast})$ and $\widetilde{\M}_{\Gamma}\subset \L$ are open and $\Gamma$-invariant, and $\tilde\cM_\Gamma$ is $\phi^t$-invariant.  
\end{lem}
\begin{proof}
The openness follows from the openness of transversality and the compactness of $\partial_{\infty}\Gamma$,  the $\Gamma$-invariance follows from the equivariance of the limit maps, and the $\phi^t$-invariance of $\tilde\cM_\Gamma$ follows from the $\phi^t$-invariance of the fibers of $p$.
\end{proof}
The next lemma establishes that non-trivial elements $\gamma\in \Gamma$ translate along a unique orbit of $\phi^{t}$ in $\widetilde{\mathcal{M}}_{\Gamma}.$
\begin{lem} \label{lem periodic points in discontinuity set} 
Let $[v:\alpha]\in \tilde\cM_\Gamma$ and $\gamma\in \Gamma\setminus \{e\}$ be such that $\gamma\cdot [v:\alpha]=[e^Tv:e^{-T}\alpha]$ for some $T\geq 0$. Then $[v]=\xi(\gamma_+)$, $[\alpha]=\xi^*(\gamma_-)$ and $T=\lambda_1(\gamma)$.
\end{lem}

\begin{proof}
Consider $\beta\in V^*\setminus\{0\}$ such that $[\beta]=\xi^*(\gamma^+)$, so that $\gamma\cdot\beta=\pm e^{\lambda_1(\gamma)}\beta$. We find
\[\beta(v)=(\gamma\cdot\beta)(\gamma v)=\pm e^{\lambda_1(\gamma)+T}\beta(v). \]

If $[v]\pitchfork \xi^*(\gamma_+)$, i.e. $\beta(v)\neq 0$, this would mean that $\lambda_1(\gamma)+T=0$. This is impossible if $T\geq 0$ and $\gamma\neq e$. As $[v:\alpha]\in \tilde\cM_\Gamma$, we must have $\xi(\gamma_+)\pitchfork [\alpha]$.

Now consider $w\in V\setminus\{0\}$ such that $[w]=\xi(\gamma_+)$, so that $\gamma\cdot w=\pm e^{\lambda_1(\gamma)}w$. We find
\[ \alpha(w)=(\gamma\cdot \alpha)(\gamma w)=\pm e^{\lambda_1(\gamma)-T}\alpha(w). \]

As $\alpha(w)\neq 0$, it means that $T=\lambda_1(\gamma)$, hence $[v]=\xi(\gamma_+)$ and $[\alpha]=\xi^*(\gamma_-)$.
\end{proof}
Now, we prove Theorem \ref{THM A}.(\ref{DOD}) affirming the properness of the $\Gamma$-action on $\widetilde{\mathcal{M}}_{\Gamma}.$  
\begin{thm}\label{thm: prop disc}
Let $\Gamma<\SL(V)$ be a torsion-free projective Anosov subgroup. The $\Gamma$-action on $\widetilde{\M}_{\Gamma}$ is free and properly discontinuous.
\end{thm}

\begin{proof}By way of contradiction, suppose there exists a sequence of elements $[v_{k}: \alpha_{k}]\in \widetilde{\M}_{\Gamma}$ such that
$[v_{k}:\alpha_{k}]\rightarrow [v:\alpha]\in \widetilde{\M}_{\Gamma}$ and $\delta_{k}\in \Gamma$ with $\delta_k\rightarrow \infty$, $\delta_{k}\cdot [v_{k}:\alpha_{k}]\rightarrow [w':\beta'].$  The convergence $[v_{k}:\alpha_{k}]\rightarrow [v:\alpha]$ means that there is a sequence $\lambda_k\in \R\setminus\{0\}$ such that $\lambda_k v_k \to v$ and $\lambda_k \alpha_k \to \alpha$. Replacing $v_k$ by $\lambda_k v_k$ and $\alpha_k$ by $\lambda_k \alpha_k$, we arrange that $v_k\to v$ and $\alpha_k\to \alpha$. By Lemma \ref{lem distinct limit points}, up to a subsequence we can consider $f\in \Gamma$ such that  $\gamma_k:=f\delta_k$ has distinct limits $\gamma_+=\lim \gamma_k$ and $\gamma_-=\lim \gamma_k^{-1}$ in $\partial_\infty\Gamma$. Set $[w:\beta]:=f\cdot [w':\beta']$, so that  $\gamma_k\cdot[v_{k}:\alpha_{k}]\to [w:\beta]$. 

By virtue of $[v:\alpha]\in \widetilde{\M}_{\Gamma},$ either $ [v]\pitchfork\xi^{\ast}(\gamma_{-})$ or $ [\alpha]\pitchfork \xi(\gamma_{-}).$  The first condition $[v]\pitchfork\xi^{\ast}(\gamma_{-})$ implies via Lemma \ref{lem prop criterion} that $\gamma_{k}v_{k}\rightarrow \infty,$ which is absurd because of the assumption $\gamma_{k}\cdot[v_{k}:\alpha_{k}]\rightarrow [w: \beta]\in \widetilde{\M}_{\Gamma}$. Indeed, the latter means that there is a sequence $\mu_k\in \R\setminus \{0\}$ such that $\mu_k \gamma_{k}v_{k}\to w$ and $\mu_k \gamma_{k}\alpha_k\to \beta$, and $\gamma_{k}v_{k}\rightarrow \infty$ implies $\mu_k\to 0$. Now 
$$
0<\beta(w)=\lim_{k\to \infty}\mu_k \gamma_{k}\alpha_k(\mu_k \gamma_k v_k)=\lim_{k\to \infty}\mu_k^2 \underbrace{\alpha_k(v_k)}_{\to \alpha(v)}= 0,$$
a contradiction.   
For the other condition $ [\alpha]\pitchfork \xi(\gamma_{-}),$ repeating the same argument for the dual representation leads to the analogous absurdity, and therefore the $\Gamma$-action on $\widetilde{\M}_{\Gamma}$ is properly discontinuous.

Applying Proposition \ref{prop boundary maps}\ref{propenumi boundary maps elements}, every non-trivial element $\gamma\in \Gamma$ satisfies $\lambda_1(\gamma)>0$, hence Lemma \ref{lem periodic points in discontinuity set} shows that no such element can fix a point in $\tilde\cM_\Gamma$.
\end{proof}
Theorem \ref{thm: prop disc} and the real analyticity of the $\SL(V)$-action on $\mathbb{L}$ imply that the quotient map
$$
\widetilde{\mathcal{M}}_{\Gamma}\rightarrow \Gamma\backslash \widetilde{\mathcal{M}}_{\Gamma}=:\mathcal{M}_{\Gamma}
$$
is a real analytic regular covering onto the real analytic (Hausdorff) manifold $\mathcal{M}_{\Gamma}$ of dimension $2\dim V-1$.  Recall that a (real analytic) locally homogeneous $(\SL(V), \mathbb{L})$-structure on a (real analytic) manifold $\mathcal{M}$ is defined by an atlas of (real analytic) charts taking values in $\mathbb{L}$ whose transition maps are locally restrictions of elements of $\SL(V).$  
The inclusion of the open set $\widetilde{\mathcal{M}}_{\Gamma}\subset \mathbb{L}$ is a global chart for a trivial $(\SL(V), \mathbb{L})$-structure.  Since $(\SL(V), \mathbb{L})$-structures descend to the base of covering spaces, and mutatis mutandis for the contact one-form $\tau$ with Reeb flow $\phi^t$, we have completed the proof of Theorem \ref{THM A}.(\ref{DOD}).

\subsection{The lifted basic set}Recall the notation $\sigma=\xi\times\xi^*:\partial_\infty\Gamma^{(2)}\to \trprod$ from \eqref{eq:sigmamap}.

\begin{definition}\label{def:Ktilde}
 The \emph{transverse limit set} is
\[\Lambda^\transverse_\Gamma:=\sigma\left(\partial_\infty\Gamma^{(2)}\right)\subset \trprod.  \] 

 The \emph{lifted basic set}  is 
$$
\widetilde{\mathcal{K}}_{\Gamma}:=p^{-1}\bigl(\Lambda_{\Gamma}^{\pitchfork}\bigr)\subset \mathbb{L}.$$

\end{definition}
We now begin the proof of Theorem \ref{THM A}.(\ref{THM A: AxiomA}) and \ref{THM A}.(\ref{real analytic fol}) which  will occupy the remainder of this section.
\begin{lem}\label{lem:limitsetinmtilde}
The lifted basic set $\widetilde{\mathcal{K}}_{\Gamma}\subset \L$ is a closed $\Gamma\times \{\phi^{t}\}_{t\in \R}$-invariant subset contained in $\widetilde{\M}_{\Gamma}$.
\end{lem}

\begin{proof}Since the projection $p:\L\rightarrow \trprod$  is $\SL(V)$-equivariant and continuous (even real analytic), and $\Lambda_{\Gamma}^{\pitchfork}$ is $\Gamma$-invariant,
the basic set $\widetilde{\mathcal{K}}_{\Gamma}$ in $\L$ is a closed $\Gamma$-invariant subset of $\L$. Let $(v,\alpha)\in\widetilde{\cK}_{\Gamma}$ and $z\in\partial_\infty\Gamma$. Let $(s,t)\in\partial_\infty\Gamma^{(2)}$ be such that $[v]=\xi(s)$ and $[\alpha]=\xi^*(t)$. If $z\neq s$, then $[v]\transverse \xi^*(z)$. If $z=s\neq t$, then $\xi(z)\transverse [\alpha]$; thus $\widetilde{\mathcal{K}}_{\Gamma}\subset \widetilde{\M}_{\Gamma}$. Since the fibers of $p$ are $\{\phi^{t}\}_{t\in \R}$-orbits, the basic set $\widetilde{\mathcal{K}}_{\Gamma}$ is $\phi^{t}$-invariant by definition.
\end{proof}

The $\Gamma$-action on the lifted basic set $\widetilde{\cK}_{\Gamma}$ can be related to Sambarino's refraction flow space through the Hopf parametrization.

\begin{lem} \label{lem - hopf parametrization projective Anosov}
Let $\norm{\cdot}$ be a Euclidean norm on $V$. The map 
\[ \cH:\map{\bbL}{\trprod\times \R}{\left[v:\alpha\right]}{\left([v],[\alpha],\log\left(\frac{\norm{v}}{\sqrt{\alpha(v)}}\right)\right)}\]
is a real analytic isomorphism.
\end{lem}

\begin{proof}
The  bijectivity of $\cH$ is exactly the statement that $\bbL$ is a $\R$-principal bundle over $\trprod$. The real analyticity of $\cH$ is immediate, and the real analyticity of its inverse can be shown by using local sections of the tautological bundles over $\bbP(V)$ and $\bbP(V^*)$: $\cH^{-1}([v],[\alpha],s)=\left[ \frac{e^s}{\norm{v}}v, \frac{\norm{v}}{e^s\alpha(v)}\alpha \right]$.
\end{proof}

The \emph{Hopf parametrization} $\cH$ conjugates the flow $\phi^t$ to the translation flow
\bq \label{eqn hopf coordinates flow} \cH\circ \phi^t\circ \cH^{-1} ([v],[\alpha],\tau)=([v],[\alpha],\tau+t) \quad \forall\, ([v],[\alpha],\tau)\in \trprod\times\R.  \eq
Conjugating the $\SL(V)$-action on $\bbL$ by $\cH$, we recover the Hopf-Buseman-Iwasawa cocycle of \eqref{equi action}:
\bq \label{eqn hopf coordinates action} \cH\circ g\circ \cH^{-1}([v],[\alpha],\tau)  = (g\cdot[v],g\cdot [\alpha],\tau+\mathscr{H}(g,[v],[\alpha]) )\quad \forall\, g\in \SL(V). \eq
The work of Sambarino 
readily implies:

\begin{lem}\label{lem:Kcompact} The action of $\Gamma$ on the lifted basic set $\widetilde{\cK}_{\Gamma}$ is cocompact.
\end{lem}
\begin{proof} The $\Gamma$-equivariant embedding $\sigma\times\id: \partial_\infty\Gamma^{(2)}\times \R\to \trprod\times\R$ maps $\partial_\infty\Gamma^{(2)}\times \R$ homeomorphically onto $\cH(\tilde\cK_\Gamma)$, so the result follows from Proposition \ref{prop: ref flow}.  
\end{proof}

\subsection{Dynamics on the quotient manifold}\label{sec:dynamicsonquotient}
By Theorem \ref{thm: prop disc}, the locally homogeneous quotient manifold
$$
\M_{\Gamma}=\Gamma\backslash \widetilde{\M}_{\Gamma}
$$
carries a complete real analytic flow $\phi^{t}: \M_{\Gamma}\rightarrow \M_{\Gamma}$ inherited from $\L$.  The $\Gamma$-quotient of the lifted basic set  
$$
\mathcal{K}_{\Gamma}:=\Gamma\backslash \widetilde{\mathcal{K}}_{\Gamma}
$$
is a  $\phi^{t}$-invariant subset of $\M_{\Gamma}$ which we call the \emph{basic set}. It is compact by Lemma \ref{lem:Kcompact}, and the restriction of the flow $\phi^t$ to  $\cK_\Gamma$ is conjugate to the refraction flow $\phi^t_S$.

\begin{proof}[Proof of Theorem \ref{SAMHYP}]
The $\Gamma$-equivariant map $\cH^{-1}\circ (\sigma\times\id):\partial_\infty\Gamma^{(2)}\times\R\to\tilde\cK_\Gamma$ descends to a bi-H\"older homeomorphism $F:\chi_\Gamma\to \cK_\Gamma$. The fact that $\cH$ conjugates the flow $\phi^t$ with the translation flow \eqref{eqn hopf coordinates flow} shows that $F$ conjugates the refraction flow and the flow $\phi^t$.
\end{proof}

Next, we relate the basic set  $\cK_\Gamma$  to various subsets of the manifold $\M_\Gamma$ associated with the flow $\phi^t:\M_\Gamma\to \M_\Gamma$  in an intrinsic dynamical manner, in particular those introduced in Definition \ref{def - non-wandering, periodic, trapped}. The result is quite simple and will justify the name \emph{basic set} for $\mathcal K_\Gamma$:
\begin{thm} \label{thm - dynamical subsets coincide}
The following subsets of $\cM_\Gamma$ coincide:
\begin{itemize}
\item The basic set $\cK_\Gamma$,
\item The non-wandering set $\mathcal{NW}(\phi^t)$ of the flow $\phi^t$,
\item The closure of the set $\mathcal P(\phi^t)$ of periodic points of the flow $\phi^t$,
\item The trapped set ${\mathcal T}(\phi^t)$ of the flow $\phi^t$,
\item The set $\overline{\bigcap_{t\in\R}\phi^t(U)}$ for any relatively compact open subset $U\subset \cM_\Gamma$ containing $\cK_\Gamma$.
\end{itemize}
Moreover, the restriction of the flow $\phi^t$ to $\cK_\Gamma$ is topologically transitive.
\end{thm}
We split the proof of Theorem \ref{thm - dynamical subsets coincide} into a series of lemmas.
\begin{lem} \label{lem topological dynamics} Consider sequences $x_k\in \cM_\Gamma$ and $t_k\to+\infty$ and points $x,y\in\cM_\Gamma$ such that $x_k\to x\in \cM_\Gamma$ and $\phi^{t_k}(x_k)\to y\in \cM_\Gamma$. Let $[v:\alpha]\in\widetilde{\cM}_{\Gamma}$ (resp. $[w:\beta]\in\widetilde{\cM}_{\Gamma}$) be a lift of $x$ (resp. of $y$). Then $[v]\in \Lambda_\Gamma$ and $[\beta]\in \Lambda^*_\Gamma$.
\end{lem}

\begin{proof}
Consider lifts $[v_k:\alpha_k]\in\widetilde{\cM}_{\Gamma}$ of $x_k$  such that $[v_k:\alpha_k]\to [v:\alpha]$. There is a sequence $\delta_k\in \Gamma$ such that $\delta_k\cdot [e^{t_k} v_k:e^{-t_k}\alpha_k]\to [w:\beta]$, and up to a subsequence, consider (thanks to Lemma \ref{lem distinct limit points})  $f\in \Gamma$ such that the sequence $\gamma_k=f\delta_k^{-1}$ satisfies $\gamma_k\to\gamma_+\in\partial_\infty \Gamma$ and $\gamma_k^{-1}\to\gamma_-\in\partial_\infty \Gamma$ with $\gamma_+\neq \gamma_-$. Write $[w_k:\beta_k]=\delta_k\cdot [e^{t_k} v_k:e^{-t_k}\alpha_k]$.

Note that $v_k\to v\neq 0$ and $\gamma_k v_k=e^{-t_k}f w_k\to 0$, so $[v]$ cannot be transverse to $\xi^*(\gamma_-)$ (otherwise Lemma \ref{lem prop criterion} would imply that $\gamma_k v_k\to \infty$). Since $[v:\alpha]\in\tilde\cM_\Gamma$, we must have $[\alpha]\pitchfork \xi(\gamma_-)$, therefore $\gamma_k[v_k]\to \xi(\gamma_+)$ (Proposition \ref{prop conv action}). This means that $f[v]=\xi(\gamma^+)$, and in particular $[v]\in\Lambda_\Gamma$.

This implies that $\xi(\gamma^+)\pitchfork f[\alpha]$, so $\gamma_k^{-1}f[\alpha_k]\to \xi^*(\gamma_-)$, i.e. $[\beta]=\xi^*(\gamma_-)\in\Lambda^*_\Gamma$.
\end{proof}

\begin{lem} \label{lem correspondence periodic orbits conjugacy classes} Let $x\in \cM_\Gamma$ and consider a lift $[v:\alpha]\in \tilde\cM_\Gamma$. Then $x\in \mathcal P(\phi^t)$ if and only if there is $\gamma\in\Gamma\setminus\{e\}$ such that $[v]=\xi(\gamma_+)$ and $[\alpha]=\xi^*(\gamma^-)$. In this case, if $\gamma$ is primitive, then the period of $x$ is $\lambda_1(\gamma)$.
\end{lem}

\begin{proof}
The equality $\phi^T(x)=x$ means that there is $\gamma\in\Gamma$ such that $[e^Tv:e^{-T}\alpha]=\gamma\cdot [v:\alpha]$, so the result follows from Lemma \ref{lem periodic points in discontinuity set}.
\end{proof}

\begin{proof}[Proof of Theorem \ref{thm - dynamical subsets coincide}] Let us start by proving the chain of inclusions
\begin{equation}\label{thmeqn chain of inclusions}
 \overline{\mathcal P(\phi^t)}\subset \mathcal{NW}(\phi^t) \subset \cK_\Gamma\subset \overline{\mathcal P(\phi^t)}. 
\end{equation}

Any flow satisfies $\overline{\mathcal P(\phi^t)}\subset \mathcal{NW}(\phi^t)$.  Let $x\in \mathcal{NW}(\phi^t)$, and consider sequences $x_k\to x$ in $\cM_\Gamma$ and $t_k\to +\infty$ in $\R$ such that $\phi^{t_k}(x_k)\to x$. If $[v:\alpha]\in \widetilde{\cM}_{\Gamma}$ is a lift of $x$, then Lemma \ref{lem topological dynamics} shows that $[v]\in\Lambda_\Gamma$ and $[\alpha]\in\Lambda^*_\Gamma$,  hence $x\in \cK_\Gamma$, i.e. $\mathcal{NW}(\phi^t)\subset\cK_\Gamma$.

The set of poles of $\Gamma$, i.e. pairs $(\gamma_+,\gamma_-)$ of fixed points of elements $\gamma\in\Gamma$, is dense in the space $\partial_\infty\Gamma^{(2)}$ of pairs of distinct points \cite[Corollary 8.2.G]{Gr}, hence $\cK_\Gamma\subset\overline{\mathcal P(\phi^t)}$ thanks to Lemma \ref{lem correspondence periodic orbits conjugacy classes}. This concludes the proof of \eqref{thmeqn chain of inclusions}.

Let us now prove that $\mathcal T(\phi^t)=\cK_\Gamma$. Since $\cK_\Gamma$ is compact and $\phi^t$-invariant, we automatically get $\cK_\Gamma\subset {\mathcal T}(\phi^t)$. Let $x\in {\mathcal T}(\phi^t)$ and consider a lift $[v:\alpha]\in\widetilde{\cM}_{\Gamma}$. Let $y\in \cM_\Gamma$ be an $\omega$-limit point of $x$, i.e. $y=\lim_{k\to +\infty}\phi^{t_k}(x)$ for some sequence $t_k\to +\infty$ (it exists because $x$ is trapped). Then Lemma \ref{lem topological dynamics} shows that $[v]\in \Lambda_\Gamma$. The same reasoning applied to an $\alpha$-limit point of $x$ (i.e. some $\lim\phi^{s_k}(x)$ with $s_k\to -\infty$) shows that $[\alpha]\in\Lambda^*_\Gamma$, hence $x\in \cK_\Gamma$.

Let $U\subset \cM_\Gamma$ be a relatively compact open subset containing $\cK_\Gamma$, and consider $y\in \overline{\bigcap_{t\in\R}\phi^t(U)}$ and a lift $[w:\beta]\in\widetilde{\cM}_{\Gamma}$. Consider first sequences $t_k\to +\infty$ and $x_k\in U$ such that $\phi^{t_k}(x_k)\to y$,  and assume without loss of generality that $x_k$ has a limit $x\in \cM_\Gamma$.  Lemma \ref{lem topological dynamics} shows that $[\beta]\in\Lambda^*_\Gamma$. Consider now sequences $s_k\to -\infty$ and $z_k\in U$ such that $y=\lim\phi^{s_k}(z_k)$, and assume that $z_k\to z\in \cM_\Gamma$. Applying Lemma \ref{lem topological dynamics} to  $\phi^{-s_k}\left( \phi^{s_k}(y)\right)$  shows that $[w]\in\Lambda_\Gamma$, hence $y\in \cK_\Gamma$. 

The fact that $\cK_\Gamma$ is the closure $\overline{\mathcal P(\phi^t)}$ of periodic points shows that $\cK_\Gamma\subset \bigcap_{t\in\R}\phi^t(\cK_\Gamma)$, so finally $\overline{\bigcap_{t\in\R}\phi^t(U)}=\cK_\Gamma$.
\end{proof}

\begin{lem} \label{lem basic set is basic}
The restriction of the flow $\phi^t$ to $\cK_\Gamma$ is topologically transitive.
\end{lem}
\begin{proof}
By Proposition \ref{prop hyperbolic groups elementary properties}\ref{propenumi hyp groups dense orbits product}, we may consider $(a,b)\in \partial_\infty\Gamma^{(2)}$ whose $\Gamma$-orbit is dense in $\partial_\infty\Gamma^{(2)}$.  Let $x\in \cK_\Gamma$ be an element admitting a lift $[v:\alpha]\in\tilde\cK_\Gamma$ with $[v]=\xi(a)$ and $[\alpha]=\xi^*(b)$. Let $y\in \cK_\Gamma$ and consider a lift $[w:\beta]\in\tilde\cK_\Gamma$. 

Let $U\subset \cK_\Gamma$ be an open subset containing $y$, and $\tilde U\subset\tilde\cK_\Gamma$ its preimage. As the projection $p:\bbL\to\trprod$ is an open map, there is $\gamma\in\Gamma$ such that $\gamma\cdot ([v],[\alpha])\in p(\tilde U)$, i.e. there is $t\in \R$ such that $\gamma\cdot [e^t v:e^{-t}\alpha] \in\tilde U$, hence $\phi^t(x)\in U$.
\end{proof}

At this point, all that remains to complete the proof of Theorem \ref{THM A}.(\ref{THM A: AxiomA}) and \ref{THM A}.(\ref{real analytic fol}) is to prove that the canonical splitting of the tangent bundle of $\mathbb{L}$ descends to a contracting/expanding splitting along the basic set $\mathcal{K}_{\Gamma}.$

\subsection{Hyperbolicity of the non-wandering set} The $\SL(V)\times\{\phi^t\}_{t\in\R}$-equivariant splitting \eqref{eqn splitting TL} restricts to a $\Gamma\times\{\phi^t\}_{t\in\R}$-equivariant splitting of $T\widetilde{\cM}_{\Gamma}$, descending to a $\phi^t$-equivariant splitting
\begin{equation}\label{eqn decomposition TMGamma} 
 T\cM_\Gamma = E^\mathrm{u}\oplus E^0\oplus E^\mathrm{s}. 
\end{equation}
Since these bundles  are real analytic, the following lemma (in conjunction with Theorem \ref{thm - dynamical subsets coincide} and Lemma \ref{lem basic set is basic}) will complete the proof of Theorem \ref{THM A}.(\ref{THM A: AxiomA}) and \ref{THM A}.(\ref{real analytic fol}) and thus finish the proof of Theorem \ref{THM A}.
\begin{lem}
The non-wandering set $\cK_\Gamma$ of the  flow $\phi^t$ on $\M_\Gamma$ is a hyperbolic set with stable (resp.\ unstable) subbundle given by $ E^\mathrm{s}|_{\cK_\Gamma}$ (resp.\ $E^\mathrm{u}|_{\cK_\Gamma}$).
\end{lem}

\begin{proof} Consider the $\Gamma$-equivariant maps 
\bq \sigma_+:\map{\partial_\infty\Gamma^{(2)}\times\R}{\bbP(V)}{(x,y,\tau)}{\xi(x)} , \quad \sigma_-:\map{\partial_\infty\Gamma^{(2)}\times\R}{\bbP(V^\ast)}{(x,y,\tau)}{\xi^*(y)}.
\label{eqn sigma pm} \eq
The vector bundles $\Gamma\backslash \sigma_+^*T\bbP(V)$ and $\Gamma\backslash \sigma_-^*T\bbP(V^*)$ over the refraction flow space $\chi_\Gamma$ each come equipped with a flow by vector bundle isomorphisms
\begin{align*} \psi^t_+&:\map{\Gamma\backslash \sigma_+^*T\bbP(V)}{\Gamma\backslash \sigma_+^*T\bbP(V)}{\Gamma\cdot((x,y,\tau),\nu)}{\Gamma\cdot((x,y,t+\tau),\nu)} ,\\
 \quad \psi^t_-&:\map{\Gamma\backslash \sigma_-^*T\bbP(V^*)}{\Gamma\backslash \sigma_-^*T\bbP(V^*)}{\Gamma\cdot((x,y,\tau),\nu)}{\Gamma\cdot((x,y,t+\tau),\nu)}.
\end{align*}
The projection $p:\bbL\to\trprod$ \eqref{eqn projection of L onto transverse product} is composed of the two coordinate projections:
\begin{equation*}
 p_L:\map{\bbL}{\bbP(V)}{\left[v:\alpha\right]}{\left[v\right]},\quad  p_R:\map{\bbL}{\bbP(V^*)}{\left[v:\alpha\right]}{\left[\alpha\right]},
\end{equation*}
and there are natural $\SL(V)$-equivariant isomorphisms of vector bundles over $\bbL$
\begin{equation}\label{eqn stable/unstable bundles as pull-backs}
E^\mathrm{u}\simeq p_L^*T\bbP(V),\quad E^\mathrm{s}\simeq p_R^*T\bbP(V^*),
\end{equation}
given by the restrictions to $E^\mathrm{u}$ and $E^\mathrm{s}$ of the differentials of $p_L$ and $p_R$. As $p_L\circ \phi^t=p_L$ and $p_R\circ \phi^t=p_R$, we also have $dp_L\circ d\phi^t=dp_L$ and $dp_R\circ d\phi^t=dp_R$, meaning that the action of $d\phi^t$ on the pull-back bundles in \eqref{eqn stable/unstable bundles as pull-backs} is trivial on fibers. So we have $\Gamma\times \R$-equivariant isomorphisms
\begin{align*}
\left.E^\mathrm{u}\right\vert_{\widetilde{\cK}_{\Gamma}} &\longrightarrow \left. p_L^*T\bbP(V)\right\vert_{\widetilde{\cK}_{\Gamma}} \longleftarrow \sigma_+^*T\bbP(V),\\
\left.E^\mathrm{s}\right\vert_{\widetilde{\cK}_{\Gamma}} &\longrightarrow \left. p_R^*T\bbP(V^\ast)\right\vert_{\widetilde{\cK}_{\Gamma}} \longleftarrow \sigma_-^*T\bbP(V^\ast),
\end{align*} 
that descend to isomorphisms
\begin{align*}
\left.E^\mathrm{u}\right\vert_{\cK_\Gamma} &\simeq \Gamma\backslash \sigma_+^*T\bbP(V), \\
\left.E^\mathrm{s}\right\vert_{\cK_\Gamma} &\simeq \Gamma\backslash \sigma_-^*T\bbP(V^*),
\end{align*}
conjugating the actions of $d\phi^t$ with $\psi_\pm^t$. Using the refraction flow of Definition \ref{sam flow} as a replacement for a Gromov-geodesic flow (Definition \ref{def - geodesic flow gromov hyperbolic group}), the original definition of an Anosov subgroup \cite[Def.~2.10, Rem.~2.11]{GW12} is that the flow $\psi_+^t$ (resp. $\psi_-^t$) is uniformly dilating (resp. contracting).
\end{proof}
\begin{rem}The differential $d\phi^t$ of the flow $\phi^t$ on $\bbL$ has a very simple expression. Let $[v:\alpha]\in\bbL$,  $(w,0)\in \ker\alpha\times\{0\}=E^\mathrm{u}_{[v:\alpha]}$ and $t\in\R$.
\[ d_{[v:\alpha]}\phi^t(w,0)=(e^tw,0). \]
One can be tempted to believe that this formula automatically implies dilation on $E^\mathrm{u}$, as for any norm $\norm{\cdot}$ on $V\times V^*$ we find
\begin{equation}\label{eqn remark not trivial 1} \frac{\norm{d_{[v:\alpha]}\phi^t(w,0)}}{\norm{(w,0)}}=e^t. 
\end{equation}
 However, the notion of dilation involves a Riemannian metric on $\cM_\Gamma$. Considering a lift of such a Riemannian metric to $\widetilde{\cM}_{\Gamma}$, the ratio that should grow exponentially fast is
\begin{equation}\label{eqn remark not trivial 2} \frac{\norm{d_{[v:\alpha]}\phi^t(w,0)}_{\phi^t([v:\alpha])}}{\norm{(w,0)}_{[v:\alpha]}}= e^t\frac{\norm{(w,0)}_{[e^tv:e^{-t}\alpha]}}{\norm{(w,0)}_{[v:\alpha]}}.
\end{equation}
The comparison would be made trivial if it were possible to choose this Riemannian metric to be constant along flow lines. In general, this is not possible because  the $\Gamma$-action on $V\times V^*$ does not preserve any norm. The situation is most evident when working at a periodic point: according to Lemma \ref{lem correspondence periodic orbits conjugacy classes}, a periodic point $x=\phi^T(x)\in \cK_\Gamma$ lifts to  $[v:\alpha]\in \widetilde{\cK}_{\Gamma}$ such that $\phi^T([v:\alpha])=\gamma\cdot [v:\alpha]$ for some $\gamma\in\Gamma$ with $\lambda_1(\gamma)=T$. Now for $\nu=d_{[v:\alpha]}\pi(w,0)\in E^\mathrm{u}_{x}$ and $t=nT$, $n\in\Z$, the ratio \eqref{eqn remark not trivial 2} becomes:
\begin{align*}
\frac{\norm{d_{x}\phi^t(\nu)}_{\phi^t(x)}}{\norm{\nu}_{x}}&= \frac{\norm{d_{[v:\alpha]}\phi^t(w,0)}_{\phi^t([v:\alpha])}}{\norm{(w,0)}_{[v:\alpha]}}\\
 &=\frac{\norm{d_{\phi^t([v:\alpha])}\gamma^{-n}\circ d_{[v:\alpha]}\phi^t(w,0)}_{\gamma^{-n}\cdot\phi^t([v:\alpha])}}{\norm{(w,0)}_{[v:\alpha]}}\\
&= e^t \frac{\norm{(\gamma^{-n}\cdot w,0)}_{[v:\alpha]}}{\norm{(w,0)}_{[v:\alpha]}}\\
&= e^{n\lambda_1(\gamma)} \frac{\norm{(\gamma^{-n}\cdot w,0)}_{[v:\alpha]}}{\norm{(w,0)}_{[v:\alpha]}}.
 \end{align*}
\end{rem}

\section{Exponential Mixing}\label{ExponentialMixing}

In this entire section, we assume that $\Gamma< \SL(V)$ is projective Anosov, torsion-free and non-elementary. The latter assumption is necessary for $\phi^t$ to be mixing, since in the elementary case $\Gamma$ is isomorphic to $\Z$ (because we also assume that $\Gamma$ is torsion-free), and the restriction of $\phi^t$ to $\cK_\Gamma$ is smoothly conjugate to the translation flow of the circle, in particular it is uniquely ergodic and not mixing.    

The bulk of this section consists in verifying the subtle geometro-dynamical hypotheses of Stoyanov \cite{St11} (see the introductory Subsection \ref{methods of proof} for a contextualized discussion) from which Theorems \ref{EXP MIX}, \ref{orbit counting} and the spectral gap of Theorem \ref{zeta fcn} follow as corollaries of well-known techniques developed by Dolgopyat \cite{DOL98} (see Dolgopyat-Pollicott \cite{DP98}) building on the work of Parry-Pollicott \cite{PP83} and Pollicott \cite{POL85}.  We end this section with an explanation of Stoyanov's \cite{St11} spectral estimates on complex Ruelle transfers and the proof of Theorem \ref{EXP MIX} from the introduction.

\subsection{Gibbs equilibria and mixing rates}

In this preliminary section, we rapidly introduce the machinary of Gibbs equilibrium measures and explain various notions of mixing.  For more details and background information, we refer the reader to \cite{BR75}.

Let $(\mathcal{K},d)$ be a compact metric space with a continuous flow 
$$
\phi^{t}: \mathcal{K}\rightarrow \mathcal{K}
$$
defined for all $t\in \mathbb{R}.$  

If $\mu$ is any finite Borel measure on $\mathcal{K},$ let $h_{\mu}$ denote the Kolmogorov-Sinai measure theoretic entropy of the flow $\phi^{t}$ with respect to $\mu.$  Given any $U\in C^{0}(\mathcal{K}, \mathbb{R}),$ the topological pressure may be defined directly via the variational formula
\begin{align}\label{pressure var}
\mathrm{Pr}_{\phi^{t}}(U)=\sup_{\mu}\left(h_{\mu} + \int_{\mathcal{K}} U \ d\mu \right)
\end{align}
where the supremum is over $\phi^{t}$-invariant Borel probability measures on $\mathcal{K}$.  A $\phi^{t}$-invariant Borel probability measure realizing the supremum in \eqref{pressure var} is called a Gibbs equilibrium state associated to the continuous potential $U.$  Given $0<\alpha<1,$ we denote the topological vector space of H\"{o}lder continuous real valued functions by $C^{\alpha}(\mathcal{K}, \mathbb{R})$ equipped with the norm
$$
\lVert F \rVert_{\alpha}=\lVert F \rVert_{\infty} + \sup_{x\neq y} \frac{\lvert F(x)-F(y)\rvert}{d(x,y)^{\alpha}}.
$$
In the following, we fix a complete Riemannian manifold $\mathcal{M}$ with a $C^{1}$-Axiom A flow $\phi^{t}: \mathcal{M}\rightarrow \mathcal{M}$ and fix a basic hyperbolic set $\mathcal{K}\subset \mathcal{M}.$  The following existence result of Gibbs equilibrium states for basic hyperbolic sets of Axiom A flows was established by Bowen-Ruelle \cite[Thm.~3.3]{BR75}.  
\begin{prop}
For every $0<\alpha<1$ and $U\in C^{\alpha}(\mathcal{K}, \mathbb{R}),$ there exists a unique Gibbs equilibrium state $\mu_{U}.$  Moreover, the flow $\phi^{t}$ is ergodic with respect to $\mu_{U}$ and $\mathrm{supp}(\mu_{U})=\mathcal{K}.$  
\end{prop}
Let $U\in C^{\alpha}(\mathcal{K}, \mathbb{R})$ with unique Gibbs measure $\mu_{U}.$  Given $F,G \in C^{\alpha}(\mathcal{K}, \mathbb{R}),$ the correlation function is defined by
\begin{align}\label{corr functions}
\mathrm{c}^{t}(F,G;U)=\left\lvert \int_{z\in \mathcal{K}} F(z)\cdot G(\phi^{t}(z)) \ d\mu_{U}(z) - \int_{z\in \mathcal{K}} F(z) \ d\mu_{U}(z)\int_{z\in \mathcal{K}} G(z) \ d\mu_{U}(z)\right\rvert.
\end{align}
The flow $\phi^{t}$ is mixing with respect to $\mu_{U}$ for all H\"{o}lder observables if for all $F,G\in C^{\alpha}(\mathcal{K}, \mathbb{R})$ one has $\mathrm{c}^{t}(F,G;U)\rightarrow 0$ as $t\rightarrow \infty$ 
and exponentially mixing if there exists $c_{\alpha}(U), C_{\alpha}(U)>0$ such that
$$
\forall\,t\in \R:\; \mathrm{c}^{t}(F,G;U)\leq C_{\alpha}(U)e^{-c_{\alpha}(U)|t|}\lVert F\rVert_{\alpha}\lVert G\rVert_{\alpha}.
$$
Exponential mixing is also called exponential decay of correlations.

\subsection{Global geometry of the stable and unstable foliations}

For $[v:\alpha]\in \mathbb L$, recall the splitting of the tangent space $T_{[v:\alpha]}\bbL=\set{(w,\beta)\in V\times V^*}{\alpha(w)+\beta(v)=0}$ as
\[T_{[v:\alpha]}\bbL= E^\mathrm{u}_{[v:\alpha]}\oplus E^\mathrm{s}_{[v:\alpha]}\oplus \R\cdot X([v:\alpha]) \]
where $E^\mathrm{u}_{[v:\alpha]}=\ker\alpha\times\{0\}$ and $E^\mathrm{s}_{[v:\alpha]}=\{0\}\times\ker\iota_v$. These distributions are tangent to foliations $W^\mathrm{s}$ and $W^\mathrm{u}$ of $\bbL$ whose leaves are
\begin{align} \begin{split}
W^\mathrm{u}([v:\alpha])&= \set{[w:\alpha]}{w\in V,\, \alpha(w)=1},\\
W^\mathrm{s}([v:\alpha])&= \set{[v:\beta]}{\beta\in V^*,\, \beta(v)=1}.
\label{eqn strong (un)stable foliation} \end{split}
\end{align}
We will also consider the central unstable distribution $E^\mathrm{cu}=E^\mathrm{u}\oplus\R\cdot X$ and the central stable distribution $E^\mathrm{cs}=E^\mathrm{s}\oplus\R\cdot X$ as well as the associated foliations $W^\mathrm{cu}$, $W^\mathrm{cs}$ of $\bbL$ whose leaves are
\begin{align} \begin{split}
W^\mathrm{cu}([v:\alpha])&= \set{\left[w:\frac{1}{\alpha(w)}\alpha\right
]}{w\in V,\, \alpha(w)> 0},\\
W^\mathrm{cs}(\left[v:\alpha\right])&= \set{\left[\frac{1}{\beta(v)}v:\beta\right]}{\beta\in V^*,\, \beta(v)>0}.
\label{eqn central (un)stable foliation} \end{split}
\end{align}
Note that for $[v:\alpha],[w:\beta]\in \bbL$ such that $\alpha(w)> 0$ (resp. $\beta(v)> 0$), the intersection $W^\mathrm{cu}([v:\alpha])\cap W^\mathrm{s}([w:\beta])$ (resp. $W^\mathrm{cs}([v:\alpha])\cap W^\mathrm{u}([w:\beta])$) consists of exactly one point:
\begin{align}\begin{split}
W^\mathrm{cu}([v:\alpha])\cap W^\mathrm{s}([w:\beta])&= \left\lbrace\left[ w:\frac{1}{\alpha(w)}\alpha\right]\right\rbrace\, ,\\
W^\mathrm{cs}([v:\alpha])\cap W^\mathrm{u}([w:\beta])&= \left\lbrace\left[ \frac{1}{\beta(v)}v:\beta\right]\right\rbrace\, . \label{eq:exactlyonepoint} \end{split}
\end{align}
Recall from \eqref{eqn pseudo-riemannian metric} that the manifold $\bbL$ carries a pseudo-Riemannian metric invariant under the action of $\SL(V)$. 
We will use the fact that the leaves of the foliation $W^\mathrm{u}$ are totally geodesic for this metric, through the unstable exponential map
\bq
 \exp^\mathrm{u}_{[v:\alpha]}:\map{E^\mathrm{u}_{[v:\alpha]}}{W^\mathrm{u}([v:\alpha])}{(w,0)}{[v+w:\alpha].}\label{eq:unstableexp}
\eq

\subsection{Local geometry of the stable and unstable foliations}

The foliations $W^\mathrm{u}$, $W^\mathrm{s}$, $W^\mathrm{cu}$ and $W^\mathrm{cs}$ project to foliations of $\cM_\Gamma$ denoted by the same names, respectively. However, the leaves of these foliations are only immersed submanifolds, so we need to be somewhat careful when discussing local properties of the leaves. From now on, we fix a complete Riemannian metric on $\cM_\Gamma$, and denote by $\norm{\cdot}$ the associated norm and by $d$ the Riemannian distance. We will use the same notation for the lift of this Riemannian metric to $\tilde\cM_\Gamma\subset \bbL$. In all that follows, distances (both in $\cM_\Gamma$ and in its tangent spaces) will be considered with respect to this Riemannian metric.

\begin{definition}[Local (central) (un)stable manifolds]
For $x\in \cM_\Gamma$ and $\varepsilon>0$, we consider 
\begin{align*}
W_\eps^\mathrm{s}(x)&:=\{y\in \M\,|\,d(\phi^t(y),\phi^t(x))\leq \eps \;\forall\,t\geq 0,\;d(\phi^t(y),\phi^t(x))\stackrel{t\to +\infty}{\longrightarrow} 0\},\\
W_\eps^\mathrm{u}(x)&:=\{y\in \M\,|\,d(\phi^t(y),\phi^t(x))\leq \eps \;\forall\,t\leq 0,\;d(\phi^t(y),\phi^t(x))\stackrel{t\to -\infty}{\longrightarrow} 0\},\\
W^\mathrm{cs}_\varepsilon(x)&:=\bigcup_{\vert t\vert<\varepsilon}W_\varepsilon^\mathrm{s}\left(\phi^t(x)\right),\\
W^\mathrm{cu}_\varepsilon(x)&:=\bigcup_{\vert t\vert<\varepsilon}W_\varepsilon^\mathrm{u}\left(\phi^t(x)\right).
\end{align*}
\end{definition}
In particular, for $\varepsilon>0$ small enough and $i\in\{\mathrm{s,u,cs,cu}\}$, $W^i_\varepsilon(x)$ is an embedded submanifold of $\cM_\Gamma$, diffeomorphic to an open Euclidean ball, and is included in  the intersection $B(x,\varepsilon)\cap W^i(x)$, however this inclusion is not an equality (the intersection can have many connected components).

The manifold $\cM_\Gamma$ inherits from $\tilde\cM_\Gamma\subset\bbL$ a quotient pseudo-Riemannian metric and its exponential map. For $x\in \cM_\Gamma$, we will denote by $\exp_x^{\mathrm{u}}:E^\mathrm{u}_x\to W^\mathrm{u}(x)$ the induced map. This is the only exponential map that we will use on $\cM_\Gamma$, meaning that there will be no reference to the exponential map of the chosen background Riemannian metric. The reason for this is that the pseudo-Riemannian unstable exponential map linearizes  the flow $\phi^t$:
\begin{equation} \label{eqn: unstable exponential linearizes the flow}
 \phi^t\circ\exp_x^{\mathrm{u}}=\exp_{\phi^t(x)}^\mathrm{u}\circ\, d_x\phi^t\vert_{E^\mathrm{u}_x}.
\end{equation}

Such a property (which follows from the explicit expression \eqref{eq:unstableexp}) will not hold for a Riemannian exponential map.  We will also use the parallel transport of unstable vectors along such geodesics. For $x=\pi([v:\alpha])\in \cM_\Gamma$ and $y=\exp^{\mathrm{u}}_{x}(w)\in\cM_\Gamma$, the parallel transport of a vector $u\in E^{\mathrm{u}}_{x}$ from $x$ to $y$ is \begin{equation} \label{eqn parallel transport unstable vector}
u_y=d_{[x+\tilde w:\alpha]}\pi(\tilde u,0)
\end{equation}
 where $w=d_{[v:\alpha]}\pi(\tilde w,0)$ and $u=d_{[v:\alpha]}\pi(\tilde u,0)$.

\vspace{\baselineskip}

The following result applies to any Axiom A flow as a consequence of the Stable Manifold Theorem (see \cite[§II.7 Thm~7.3]{smale67} and \cite[Thm~6.4.9]{KatokHasselblatt} for the discrete time case and \cite[Thm.~5]{dyatlovnotes} for a detailed treatment of the flow case) and compactness of $\cK_\Gamma$, but in our case it can be recovered from  the explicit formulas \eqref{eqn strong (un)stable foliation}, \eqref{eqn central (un)stable foliation}, \eqref{eq:exactlyonepoint}, \eqref{eq:unstableexp} and \eqref{eqn parallel transport unstable vector}.

\begin{lem} \label{lem: scale setting}  
There exist $\varepsilon_0,\varepsilon_1,\delta_0>0$  (with $\varepsilon_0>2\varepsilon_1$) and $L_0>1$ such that the following properties are satisfied at every $x\in \cK_\Gamma$:
\begin{enumerate}[label=(\arabic*),ref=\emph{(\arabic*)}]
\item The manifolds $W^{\mathrm{u}}_{\varepsilon_0}(x)$, $W^{\mathrm{s}}_{\varepsilon_0}(x)$, $W^{\mathrm{cu}}_{\varepsilon_0}(x)$ and $W^{\mathrm{cs}}_{\varepsilon_0}(x)$ are embedded submanifolds of $\cM_\Gamma$.
\item The balls $B(\tilde x,\varepsilon_0)\subset\tilde\cM_\Gamma$ have disjoint closures for distinct lifts $\tilde x\in\tilde\cM_\Gamma$ of $x$, and the projection $\pi:\tilde\cM_\Gamma\to\cM_\Gamma$ restricts to a diffeomorphism $B(\tilde x,\varepsilon_0)\to B(x,\varepsilon_0)$.
\item \label{lemenumi dynamical projections} For any $y\in B(x,2\varepsilon_1)$, the intersections $W^{\mathrm{s}}_{\varepsilon_0}(x)\cap W^{\mathrm{cu}}_{\varepsilon_0}(y)$ and $W^{\mathrm{u}}_{\varepsilon_0}(x)\cap W^{\mathrm{cs}}_{\varepsilon_0}(y)$ each consist in a single point, and these points belong to $\cK_\Gamma$.
\item \label{lemenumi unstable exponential} The unstable exponential map $\exp^{\mathrm{u}}_x$ realizes a $L_0$-bi-Lipschitz diffeomorphism from $E^{\mathrm{u}}_x\cap B(0,\delta_0)$ to an open subset of $W^{\mathrm{u}}_{\varepsilon_0}(x)$ containing $W^{\mathrm{u}}_{\varepsilon_1}(x)$.
\item \label{lemenumi norm parallel transport} For any $y\in B(x,\varepsilon_1)$, the norm of the parallel transport $E_x^{\mathrm{u}}\to E_y^{\mathrm{u}}$ is less than $L_0$.
\item \label{lemenumi Lipschitz flow} $d(\phi^{-t}(x),\phi^{-t}(y))\leq L_0\,d(x,y)$ for any $y\in W^{\mathrm{u}}_{\varepsilon_1}(x)$ and $t\geq 0$.
\end{enumerate}
\end{lem}

\begin{definition} \label{def stable holonomy}
Let $x,y\in \cK_\Gamma$ be such that $d(x,y)<\varepsilon_1$. The \emph{stable holonomy map} is the map $\cH_x^y:W_{\varepsilon_1}^\mathrm{u}(x)\to W_{\varepsilon_0}^\mathrm{u}(y)$ defined by:
\[ \cH_x^y(z)\in W^\mathrm{cs}_{\varepsilon_0}(z)\cap W^{\mathrm{u}}_{\varepsilon_0}(y),\quad \forall z\in W_{\varepsilon_1}^\mathrm{u}(x). \]
\end{definition}

Lemma \ref{lem: scale setting} \ref{lemenumi dynamical projections} shows that $\cH_x^y(z)\in\cK_\Gamma$ if and only if $z\in \cK_\Gamma$. This can be seen in the explicit formula:

\begin{equation} \label{eqn formula stable holonomy}
\cH_x^y(z) = \pi\left(\left[\frac{1}{\beta(v')}v':\beta\right]\right)
\end{equation}

where $x=\pi([v:\alpha])$, $y=\pi([w:\beta])$ and $z=\pi([v':\alpha])$ with all three lifts in a common ball of radius $\varepsilon_0$ in $\tilde\cK_\Gamma$.

\subsection{The infinitesimal unstable limit set}

Following Stoyanov's strategy, we will replace the intersection $\cK_\Gamma\cap W^\mathrm{u}(x)$ with a linearised limit set $\Lambda^\mathrm{u}(x)\subset E^{\mathrm{u}}_x$ at several instances. This will be achieved  thanks to the special geometry of the unstable foliation.

\begin{definition}
For $x\in \cM_\Gamma$, the \emph{infinitesimal unstable limit set} is 
\[\Lambda^{\mathrm{u}}(x)=\set{v\in E^{\mathrm{u}}_x}{\exp_x^{\mathrm{u}}(v)\in \cK_\Gamma}.\]
\end{definition}

In view of \eqref{eq:unstableexp}, an explicit description of the infinitesimal unstable limit set is that for $x=\pi([v:\alpha])\in \mathcal M_\Gamma$ and $w\in \ker\alpha$, one has 
\bq
d_{[v:\alpha]}\pi(w,0)\in\Lambda^\mathrm{u}(x)\iff[v+w]\in\Lambda_\Gamma.\label{eq:iniffplus}
\eq
Note that since the unstable exponential map linearizes the flow by \eqref{eqn: unstable exponential linearizes the flow}, the infinitesimal unstable limit set is stable under the differential of the flow:
\begin{equation} \label{eqn flow preserves infinitesimal limit set}
 \forall x \in\cK_\Gamma~\forall t\in\R\quad d_x\phi^t\left( \Lambda^\mathrm{u}(x)\right) = \Lambda^\mathrm{u}\left(\phi^t(x)\right). 
\end{equation}
Let $\delta_1\in\left(0,\min\left(\frac{\varepsilon_1}{L_0},\delta_0\right)\right)$ and $\varepsilon_2\in (0,\varepsilon_1)$ be small enough so that 
\begin{equation} \label{eqn delta_1 epsilon_2}
\cH_x^y\left(W^\mathrm{u}_{L_0\delta_1}(x)\right)\subset W^\mathrm{u}_{\varepsilon_1}(y)
\end{equation}
for any $x,y\in \cK_\Gamma$ with $d(x,y)\leq \varepsilon_1$.

\begin{definition}[Infinitesimal holonomy]
Let $x,y\in \cK_\Gamma$ be such that $d(x,y)<\varepsilon_2$.  The \emph{infinitesimal holonomy map} is the map
\[ \widehat{\cH_x^y}:E^\mathrm{u}_x\cap B(0,\delta_1)\to E^\mathrm{u}_y \]
such that $\exp^\mathrm{u}_y\circ \widehat{\cH_x^y}=\cH_{x}^y\circ \exp_x^\mathrm{u}$.
\end{definition}

Note that this is well-defined because of \eqref{eqn delta_1 epsilon_2} and Lemma \ref{lem: scale setting}. By \eqref{eq:exactlyonepoint} and \eqref{eq:unstableexp} we have the following explicit formula for the infinitesimal holonomy map: write $x=\pi([v:\alpha])$, $y=\pi([v':\alpha'])$ (lifts in a common ball of radius $\varepsilon_0$ in $\tilde\cK_\Gamma$) and $u=d_{[v:\alpha]}(\tilde u,0)$. Then
\bq
\widehat{\cH_x^y}(u)=d_{[v':\alpha']}\pi\left(\frac{v+\tilde u}{\alpha'(v+\tilde u)}-v',0\right).\label{eq:Hxy}
\eq
\begin{lem} \label{lem infinitesimal holonomy preserves infinitesimal unstable limit set}
Let $x,y\in \cK_\Gamma$ be such that $d(x,y)<\varepsilon_2$, and consider $u\in E^\mathrm{u}_x\cap B(0,\delta_1)$. Then $\widehat{\cH_x^y}(u)\in \Lambda^\mathrm{u}(y)$ if and only if $u\in \Lambda^\mathrm{u}(x)$.
\end{lem}
\begin{proof}It follows from Lemma \ref{lem: scale setting}\ref{lemenumi dynamical projections}, but it is also immediate from \eqref{eq:iniffplus}, \eqref{eq:Hxy} and the equality $\left[v'+\left(\frac{v+\tilde u}{\alpha'(v+\tilde u)}-v'\right) \right]=[v+\tilde u]$.
\end{proof}

Next, we will need to consider the span of the limit set $\Lambda_\Gamma$ in $V$. Given a subset $E\subset \bbP(V)$, let  us denote by $\mathrm{Span}(E)\subset V$ the span of $\set{v\in V}{[v]\in E}$.

\begin{lem} \label{lem span of open subset of the limit set}
For any non-empty open subset $U\subset \Lambda_\Gamma$, one has $\mathrm{Span}(U)=\mathrm{Span}(\Lambda_\Gamma)$.
\end{lem}

\begin{proof}
Consider an element $\gamma\in \Gamma$ such that $\xi(\gamma_+)\in U$ and $\gamma\cdot U\subset U$. Then $\gamma$ must preserve $\mathrm{Span}(U)$, and by applying positive powers of $\gamma$ we see that $\Lambda_\Gamma\setminus\{\xi(\gamma_-)\}\subset \bbP\left(\mathrm{Span}(U)\right)$. By density of $\Lambda_\Gamma\setminus\{\xi(\gamma_-)\}$ in $\Lambda_\Gamma$ (because $\Gamma$ is non-elementary), we find that $\Lambda_\Gamma\subset\bbP\left(\mathrm{Span}(U)\right)$, hence $\mathrm{Span}(\Lambda_\Gamma)\subset \mathrm{Span}(U)$.
\end{proof}

\begin{lem} \label{lem span of infinitesimal limit set}
Let $x=\pi([v:\alpha])\in \cK_\Gamma$ and $u=d_{[v:\alpha]}\pi(\tilde u,0)\in E^\mathrm{u}_x$. Then $u\in \mathrm{Span}\left(\Lambda^\mathrm{u}(x)\right)$ if and only if $\tilde u\in \mathrm{Span}(\Lambda_\Gamma)$.
\end{lem}

\begin{proof}
First assume that $u\in \mathrm{Span}\left(\Lambda^\mathrm{u}(x)\right)$, and write $u=\sum t_iu_i$ with $t_i\in\R$ and $u_i\in \Lambda^\mathrm{u}(x)$, that is $u_i=d_{[v:\alpha]}\pi(\tilde u_i,0)$ with $[v+\tilde u_i]\in \Lambda_\Gamma$. Rewriting $\tilde u=\sum t_i(v+\tilde u_i) - \left(\sum t_i\right) v$ shows that $\tilde u\in \mathrm{Span}(\Lambda_\Gamma)$.

Now assume that $\tilde u\in \mathrm{Span}(\Lambda_\Gamma)$, and write $\tilde u=\sum v_i$ with $[v_i]\in \Lambda_\Gamma$. If we assume that no $v_i$ belongs to $\ker \alpha$, we can then write \[\tilde u=\sum \alpha(v_i)\left( \frac{v_i}{\alpha(v_i)}-v\right)\]
because $\sum\alpha(v_i)=\alpha(\tilde u)=0$. This shows that $u\in \mathrm{Span}\left(\Lambda^\mathrm{u}(x)\right)$ thanks to \eqref{eq:iniffplus}.

More generally, we can decompose $\tilde u=\tilde u_*+\tilde u_0$ where $\tilde u_*=\sum_{\alpha(v_i)\neq 0}v_i$ and $\tilde u_0=\sum_{\alpha(v_i)=0}v_i$, and find that $u_*=d_{[v:\alpha]}\pi(\tilde u_*,0)\in \mathrm{Span}\left(\Lambda^\mathrm{u}(x)\right)$. In order to deal with $\tilde u_0$, recall that there is only one element $\ell\in\Lambda_\Gamma$ such that $\alpha\vert_\ell=0$ (Proposition \ref{prop boundary maps}), so we find that $[\tilde u_0]=\ell$. However, since $\Lambda_\Gamma\setminus \{\ell\}$ is dense in $\Lambda_\Gamma$ (because $\Gamma$ is non-elementary), we find that $\tilde u_0\in \mathrm{Span}\left( \Lambda_\Gamma\setminus \{\ell\}\right)$, i.e. we can also decompose $\tilde u_0=\sum w_j$ with $[w_j]\in \Lambda_\Gamma$ and $\alpha(w_j)\neq 0$, so the corresponding decomposition 
\[\tilde u_0=\sum \alpha(w_j)\left( \frac{w_j}{\alpha(w_j)}-v\right)\]
shows that $u_0=d_{[v:\alpha]}\pi(\tilde u_0,0)\in \mathrm{Span}\left(\Lambda^\mathrm{u}(x)\right)$, and finally $u\in \mathrm{Span}\left(\Lambda^\mathrm{u}(x)\right)$.

\end{proof}

\begin{cor} \label{cor span of infinitesimal limit set forms a vector bundle} 
For any $\delta>0,x\in \cK_\Gamma$, one has $\mathrm{Span}\left(\Lambda^\mathrm{u}(x)\cap B(0,\delta)\right)=\mathrm{Span}\left(\Lambda^\mathrm{u}(x)\right)$. The collection of these subspaces forms a continuous vector subbundle of $E^\mathrm{u}\vert_{\cK_\Gamma}$.
\end{cor}

\begin{proof}
It follows from  Lemma \ref{lem span of open subset of the limit set}, Lemma \ref{lem span of infinitesimal limit set}, and the fact that $\alpha\vert_{\mathrm{Span}(\Lambda_\Gamma)}\neq 0$ whenever $[\alpha]\in \Lambda^*_\Gamma$.
\end{proof}

\begin{rem} If $\Gamma$ is irreducible, then $\mathrm{Span}(\Lambda_\Gamma)=V$ and  $\mathrm{Span}(\Lambda^\mathrm{u})=E^\mathrm{u}$. 
\end{rem}

\subsection{Regular distortion along unstable manifolds}

\begin{definition}[Bowen's dynamical balls] \label{Bowen Balls} Let $x\in \mathcal{K}_\Gamma, T>0$ and $\delta>0$ and define
\[B_{T}^{\mathrm{u}}(x, \delta)=\set{y\in W^{\mathrm{u}}_{\varepsilon_0}(x)}{d(\phi^{t}(x), \phi^{t}(y))\leq \delta, \forall 0\leq t\leq T}.\]
\end{definition}

\begin{definition}[Uniformly regular distortion along unstable manifolds] \label{def: URDU}
 The flow $\phi^{t}$ has \emph{uniformly regular distortion along unstable manifolds} over the basic set $\cK_\Gamma$ if for some constant $\varepsilon_*>0$ and every $\delta\in (0,\varepsilon_*)$, there exists $R_\delta>0$ such that
\[ \mathrm{diam}\left(\cK_\Gamma\cap B_{T}^{\mathrm{u}}(x, \varepsilon)\right)\leq \varepsilon\,R_\delta\cdot\mathrm{diam}\left(\cK_\Gamma\cap B_{T}^{\mathrm{u}}(x, \delta)\right)\]
for every $x\in \cK_\Gamma$, $\varepsilon\in (0,\varepsilon_*)$ and $T>0.$ 
\end{definition}

\begin{rem} \label{rem RDU vs URDU} This is a strengthening of Stoyanov's notion of \emph{regular distortion along unstable manifolds} in \cite{St11}, combining two requirements in one stronger inequality.
\end{rem}

In order to establish uniformly regular distortion along unstable manifolds, we follow the strategy used by Stoyanov in \cite{St13} in the setting of flows satisfying a special pinching hypothesis (satisfied by geodesic flows of rank one locally symmetric spaces, but not by our higher rank setting), and  replace the basic set $\cK_\Gamma$ with the infinitesimal unstable limit sets $\Lambda^{\mathrm{u}}(x)$. We will also replace Bowen's dynamical balls with linearised versions.

\begin{definition}
For $x\in \cK_\Gamma$, $T\geq 0$ and $\delta>0$, the \emph{infinitesimal unstable dynamical ball} is
\[ \Lambda^{\mathrm{u}}_T(x,\delta)=\set{v\in \Lambda^{\mathrm{u}}(x)}{\Vert d_x\phi^T(v)\Vert\leq \delta}.  \]
\end{definition}

We will now see that diameter estimates for Bowen balls can be replaced with diameter estimates for the infinitesimal unstable dynamical balls.

\begin{lem} \label{lem diameters of linearised Bowen balls}
For any $x\in \cK_\Gamma$, $\varepsilon\in (0,\varepsilon_1)$, $\delta\in(0,\delta_0)$ and $T\geq 0$, the following inequalities hold:
\[\mathrm{diam}(\cK_\Gamma\cap B^{\mathrm{u}}_T(x,\varepsilon)) \leq 2L_0\, \mathrm{diam}\left(\Lambda^{\mathrm{u}}_T(x,L_0\,\varepsilon)\right)\]
and
\[ \mathrm{diam}\left(\Lambda^{\mathrm{u}}_T(x,\delta)\right)\leq L_0\,\mathrm{diam}(\cK_\Gamma\cap B^{\mathrm{u}}_T(x,L_0^2\,\delta)).\]
\end{lem}
\begin{proof}
Let $x\in \cK_\Gamma$, $\varepsilon\in (0,\varepsilon_1)$ and $T\geq 0$. Let $y,y'\in \cK_\Gamma\cap B^\mathrm{u}_T(x,\varepsilon)$ be such that $d(y,y')=\mathrm{diam}(\cK_\Gamma\cap B^\mathrm{u}_T(x,\varepsilon))$. Assume without loss of generality that $d(x,y)\geq d(x,y')$ and consider $v\in E^\mathrm{u}_x\cap B(0,\delta_0)$ such that $y=\exp_x^\mathrm{u}(v)$. Now we estimate:
\begin{align*}
\norm{d_x\phi^T(v)}&\leq L_0\, d(\phi^T(x),\exp^\mathrm{u}_{\phi^T(x)}(d_x\phi^T(v))) & \textrm{by Lemma } \ref{lem: scale setting}\,\ref{lemenumi unstable exponential} \\
&\leq L_0\, d(\phi^T(x),\phi^T(y)) & \textrm{by } \eqref{eqn: unstable exponential linearizes the flow}\\
&\leq L_0\,\varepsilon & \textrm{because }  y\in B^\mathrm{u}_T(x,\varepsilon).
\end{align*}
It follows that $v\in \Lambda^\mathrm{u}_T(x,L_0\varepsilon)$, hence $\Vert v\Vert \leq \mathrm{diam}\left(\Lambda^\mathrm{u}_T(x,L_0\varepsilon)\right)$ since $0\in \Lambda^\mathrm{u}_T(x,L_0\varepsilon)$, and we find:
\begin{align*}
\mathrm{diam}(\cK_\Gamma\cap B^\mathrm{u}_T(x,\varepsilon)) &\leq 2d(x,y) & \\
& \leq 2 L_0 \Vert v\Vert & \textrm{by Lemma } \ref{lem: scale setting}\,\ref{lemenumi unstable exponential}\\
& \leq 2L_0\, \mathrm{diam}\left(\Lambda^\mathrm{u}_T(x,L_0\varepsilon)\right). &
\end{align*}
Now let $\delta\in (0,\delta_0)$ and $v\in \Lambda^\mathrm{u}_T(x,\delta)$, and write $y=\exp^\mathrm{u}_x(v)$. Then
\begin{align*}
d(\phi^T(x),\phi^T(y))&=d(\phi^T(x),\exp^\mathrm{u}_{\phi^T(x)}(d_x\phi^T(v))) & \textrm{by } \eqref{eqn: unstable exponential linearizes the flow}\\
&\leq L_0\,\norm{d_x\phi^T(v)} & \textrm{by Lemma } \ref{lem: scale setting}\ref{lemenumi unstable exponential} \\
&\leq L_0\,\delta & \textrm{because } v\in \Lambda^\mathrm{u}_T(x,\delta).
\end{align*}
 For $t\in [0,T]$, we find:
\begin{align*}
d(\phi^t(x),\phi^t(y))&= d\left(\phi^{-(T-t)}(\phi^T(x)),\phi^{-(T-t)}(\phi^T(y))\right) &\\
&\leq L_0\, d(\phi^T(x),\phi^T(y)) &\textrm{by Lemma }\ref{lem: scale setting}\ref{lemenumi Lipschitz flow}\\
&\leq L_0^2\,\delta. &
\end{align*}
This proves that $y\in B^\mathrm{u}_T(x,L_0^2\delta)$. Now given $v,v'\in \Lambda^\mathrm{u}_T(x,\delta)$ and setting $y=\exp^\mathrm{u}_x(v),y'=\exp^\mathrm{u}_x(v')\in B^\mathrm{u}_T(x,L_0^2\delta)$, we know that $\Vert v-v'\Vert \leq L_0\,d(y,y')$ by Lemma \ref{lem: scale setting}\ref{lemenumi unstable exponential}, hence $\mathrm{diam}(\Lambda^\mathrm{u}_T(x,\delta))\leq L_0\,\mathrm{diam}(\cK_\Gamma\cap B^\mathrm{u}_T(x,L_0^2\delta))$.
\end{proof}

We will now prove an infinitesimal version of regular distortion along unstable manifolds, as in this setting the self-similarity of the basic set can be treated linearly.

\begin{lem} \label{lem:regular distortion for infinitesimal limit set} 
For any $x\in \cK_\Gamma$ and $\delta\in(0,\delta_1)$, there exist a constant $D_{x,\delta}>0$ and a neighbourhood $V_{x,\delta}\subset \cK_\Gamma$ of $x$ such that
\[
\mathrm{diam}\left(\Lambda^\mathrm{u}_T\left(\phi^{-T}(y),\varepsilon\right)\right)\leq \varepsilon\,D_{x,\delta}\cdot \mathrm{diam}\left(\Lambda^\mathrm{u}_T\left(\phi^{-T}(y),\delta\right)\right)
\]
for any $y\in V_{x,\delta}$, $\varepsilon\in (0,\delta_1)$ and $T\geq 0$.
\end{lem}
\begin{proof}
Since we are working locally, we can fix an orthonormal frame field $(e_1,\dots,e_m)$ of the vector bundle $\mathrm{Span}(\Lambda^\mathrm{u})$. Let us fix $x\in \cK_\Gamma$ and $\delta\in(0,\delta_1)$, and consider some linearly independent vectors $u_1(x),\dots,u_m(x)\in\Lambda^\mathrm{u}_x\cap B(0,\delta)$ (they exist thanks to Corollary \ref{cor span of infinitesimal limit set forms a vector bundle}). For $y\in \cK_\Gamma\cap B(x,\varepsilon_2)$, we let $u_i(y)=\widehat{\cH_x^y}(u_i(x))$. Now thanks to Lemma \ref{lem infinitesimal holonomy preserves infinitesimal unstable limit set} and the continuous dependence of the map $\widehat{\cH_x^y}$ on the point $y$, we can find a neighbourhood $V_{x,\delta}$ of $x$ in $\cK_\Gamma$ such that every $y\in V_{x,\delta}$ satisfies:
\begin{itemize}
\item $\frac{\norm{u_i(x)}}{2}\leq \norm{u_i(y)}\leq 2\norm{u_i(x)}$, $i=1,\dots,m$;
\item $(u_1(y),\dots,u_m(y))$ is a basis of $\mathrm{Span}(\Lambda^\mathrm{u}_y)$.
\end{itemize}
For each $y\in V_{x,\delta}$ consider now the linear map $L_y:\mathrm{Span}(\Lambda^\mathrm{u}_y)\to \mathrm{Span}(\Lambda^\mathrm{u}_y)$ defined by $L_y(u_i(y))=e_i(y)$ and put $b_{x,\delta}:=\sup_{y\in V_{x,\delta}}\norm{L_y}\in (0,\infty)$, where $\norm{L_y}$ is the operator norm of $L_y$. Consider the vectors $v_i(y)=d_y\phi^{-T}(u_i(y))\in \Lambda^\mathrm{u}_T\left(\phi^{-T}(y),\delta\right)$ (recalling that $u_i(y)\in \Lambda^\mathrm{u}_y$ and $\norm{u_i(y)}\leq \delta$). Since $0\in\Lambda^\mathrm{u}_T\left(\phi^{-T}(y),\delta\right)$, we find that
\begin{equation} \label{lemeqn norm v_i(y)}
 \norm{v_i(y)}\leq \mathrm{diam}\left(\Lambda^\mathrm{u}_T\left(\phi^{-T}(y),\delta\right)\right). 
\end{equation}
Now consider an arbitrary $v\in \Lambda^\mathrm{u}_T\left(\phi^{-T}(y),\varepsilon\right)$, decompose $v=\sum_{i=1}^m c_iv_i(y)$ and consider the vector $u=d_{\phi^{-T}(y)}\phi^T(v)\in \Lambda^\mathrm{u}_y\cap B(0,\varepsilon)$. Its  decomposition as $u=\sum_{i=1}^mc_i u_i(y)$ provides us with the following estimate:
\begin{equation} \label{lemeqn norm c_i}
\sqrt{\sum_{i=1}^mc_i^2} =\norm{\sum_{i=1}^mc_i e_i(y)}= \norm{L_y u}\leq b_{x,\delta}\norm{u}\leq b_{x,\delta}\,\varepsilon.
\end{equation}
It follows that
\begin{align*}
\norm{v}=\norm{\sum_{i=1}^mc_i v_i(y)}& \leq \sum_{i=1}^m\vert c_i\vert\cdot \norm{v_i(y)} & \\
&\leq \sqrt{\sum_{i=1}^mc_i^2} \sqrt{\sum_{i=1}^m\norm{v_i(y)}^2} & \\
 &\leq b_{x,\delta}\,\varepsilon \, \sqrt{\sum_{i=1}^m\norm{v_i(y)}^2} & \textrm{by }\eqref{lemeqn norm c_i}\\
&\leq \,b_{x,\delta}\,\varepsilon\,\sqrt{m}\,\mathrm{diam}\left(\Lambda^\mathrm{u}_T\left(\phi^{-T}(y),\delta\right)\right). & \textrm{by }\eqref{lemeqn norm v_i(y)}
\end{align*}
This proves the claim with $D_{x,\delta}=2\sqrt{m}\,b_{x,\delta}$.
\end{proof}

\begin{cor} \label{cor local regular distortion} 
For some constant $\varepsilon_3>0$ and any $\delta\in(0,\varepsilon_3)$ and $x\in \cK_\Gamma$, there exist a constant $C_{x,\delta}>0$  and a neighbourhood  $W_{x,\delta}\subset\cK_\Gamma$ of $x$  such that
\[ \mathrm{diam}\left( \cK_\Gamma\cap B_T^\mathrm{u}(\phi^{-T}(y),\varepsilon ) \right) \leq \varepsilon\,C_{x,\delta}\cdot \mathrm{diam}\left( \cK_\Gamma\cap B_T^\mathrm{u}(\phi^{-T}(y),\delta ) \right) \]
for any $y\in W_{x,\delta}$, $\varepsilon\in(0,\varepsilon_3)$ and $T\geq 0$.

\end{cor}

\begin{proof}
Using the notations of Lemma \ref{lem: scale setting}, Lemma \ref{lem diameters of linearised Bowen balls} and Lemma \ref{lem:regular distortion for infinitesimal limit set}, we can set $\varepsilon_3=\min\left(\frac{\delta_1}{L_0},\delta_0L_0^2\right)$, $W_{x,\delta}=V_{x,\frac{\delta}{L_0^2}}$ and $C_{x,\delta}=2L_0^3 D_{x,\frac{\delta}{L_0^2}}$. We find:

\begin{align*}
\mathrm{diam}\left( \cK_\Gamma\cap B_T^\mathrm{u}(\phi^{-T}(y),\varepsilon ) \right)&\leq 2L_0 \,\mathrm{diam}\left(\Lambda^\mathrm{u}_T\left(\phi^{-T}(y),L_0\varepsilon\right)\right) & \textrm{by Lemma }\ref{lem diameters of linearised Bowen balls} \\
&\leq 2L_0^2\varepsilon D_{x,\frac{\delta}{L_0^2}} \,\mathrm{diam}\left(\Lambda^\mathrm{u}_T\left(\phi^{-T}(y),\frac{\delta}{L_0^2}\right)\right) &\textrm{by Lemma }\ref{lem:regular distortion for infinitesimal limit set} \\
&\leq  2L_0^3\varepsilon D_{x,\frac{\delta}{L_0^2}} \,\mathrm{diam}\left( \cK_\Gamma\cap B_T^\mathrm{u}(\phi^{-T}(y),\delta ) \right). & \textrm{by Lemma }\ref{lem diameters of linearised Bowen balls}
\end{align*}
\end{proof}

\begin{cor} \label{cor regular distortion along unstable manifolds}
The flow $\phi^t$ has uniform regular distortion along unstable manifolds.
\end{cor}

\begin{proof}
Set $\varepsilon_*=\varepsilon_3$ given by Corollary \ref{cor local regular distortion}, and let $\delta\in (0,\varepsilon_*)$. Consider finitely many elements $x_1,\dots,x_N\in\cK_\Gamma$ such that $\cK_\Gamma=\bigcup_{i=1}^N W_{x_i,\delta}$. Now let $x\in \cK_\Gamma$, $\varepsilon\in (0,\varepsilon_*)$ and $T\geq 0$. By choosing some $i\in\{1,\dots,N\}$ such that $\phi^T(x)\in W_{x_i,\delta}$, we find that 
\[   \mathrm{diam}\left( \cK_\Gamma\cap B_T^\mathrm{u}(x,\varepsilon ) \right) \leq \varepsilon\,C_\delta\cdot \mathrm{diam}\left( \cK_\Gamma\cap B_T^\mathrm{u}(x,\delta ) \right) \]
where $C_\delta=\max\{C_{x_1,\delta},\dots,C_{x_N,\delta}\}$.
\end{proof}

\subsection{Local non-integrability} 

The central notion in the work of Dolgopyat \cite{DOL98} is the \emph{time separation function} introduced by Chernov \cite{CHE98}. Any two nearby points $x,y\in \cK_\Gamma$ can be joined by a path that consists in the concatenation of
\begin{enumerate}
\item a path in $W^\mathrm{s}_{\varepsilon_0}(x)$,
\item flowing out for some time $\Delta(x,y)$,
\item a path in $W^\mathrm{u}_{\varepsilon_0}(y)$.
\end{enumerate}  

This process defines a function $\Delta:\cK_\Gamma\times\cK_\Gamma\to\R$ (see Definition \ref{def time separation} below for a precise definition) called the time separation function. Lower bounds on the time separation function around the diagonal are called \emph{non integrability conditions}, as they quantify the fact that the distribution $E^\mathrm{s}\oplus E^\mathrm{u}$ is not integrable. In our case, the distribution $E^\mathrm{s}\oplus E^\mathrm{u}$ is a contact structure, so any reasonable non-integrability condition should be satisfied. However in the work of Stoyanov \cite{St11}  the precise non-integrability condition one needs involves the fractal geometry of $\cK_\Gamma$.

\begin{definition}[Time separation and projection] \label{def time separation}
Let $x,y\in \cK_\Gamma$ with $d(x,y)<\varepsilon_1$. Define $\pi_y(x)\in \cK_\Gamma$ by $\pi_y(x)\in W^\mathrm{s}_{\varepsilon_0}(x)\cap W^\mathrm{cu}_{\varepsilon_0}(y)$ and $\Delta(x,y)\in(-\varepsilon_0,\varepsilon_0)$ by $\phi^{\Delta(x,y)}(\pi_y(x))\in W^\mathrm{u}_{\varepsilon_0}(y)$.
\end{definition}

These definitions make sense by Lemma \ref{lem: scale setting}. If we consider lifts $[v:\alpha],[w:\beta]\in\tilde\cK_\Gamma$ of $x,y$ in a common ball of radius $\varepsilon_1$, we find:

\begin{equation}
 \pi_y(x)=\pi\left(\left[v:\frac{1}{\beta(v)}\beta\right]\right) , \qquad  \Delta(x,y)=-\log \beta(v).
 \label{eqn time separation and projection} 
\end{equation}


\begin{definition}[Strong Local Non-Integrability Condition (SLNIC)] \label{def LNIC}
There exists $d_0\in (0,1)$ such that, for any $\varepsilon\in(0,\varepsilon_0)$, $z\in \cK_\Gamma$ and any unit vector $w\in E^\mathrm{u}_z$, there exist $y\in \cK_\Gamma\cap W^\mathrm{s}_{\varepsilon}(z)$, $\kappa>0$ and $\varepsilon'\in (0,\varepsilon)$ such that:
\[ \vert \Delta(\exp^{\mathrm{u}}_{x}(u),\pi_{y}(x))\vert \geq \kappa \norm{u}\]
for any $x\in \cK_\Gamma\cap W^\mathrm{u}_{\varepsilon'}(z)$ and $u\in \Lambda^\mathrm{u}(x)$ with $\norm{u}\leq \varepsilon'$ and $\norm{\frac{u_z}{\norm{u_z}}-w}\leq d_0$ where $u_z\in E^\mathrm{u}_z$ is the parallel transport of $u$ along the geodesic in $W^\mathrm{u}(z)$ from $x$ to $z$.
\end{definition}

\begin{rem} The relation between this definition and the geometry of the basic set $\cK_\Gamma$ lies in the requirement that $y\in \cK_\Gamma$. If we were to drop this requirement, the condition would be a straightforward consequence of the fact that the distribution $E^\mathrm{s}\oplus E^\mathrm{u}$ is a contact structure.
\end{rem}

\begin{rem} \label{rem LNIC vs SLNIC} The Strong Local Non-Integrability Condition (SLNIC) is stronger than the Local Non-Integrability Condition (LNIC) used by Stoyanov \cite{St11}. Amongst other things, he only requires the property for vectors $w\in E^\mathrm{u}_z$ that are tangent to $\cK_\Gamma$ at $z$ (in some appropriate sense). Requiring $\Gamma$ to be irreducible frees us from having to use this tangency. \end{rem}

\begin{lem} \label{lem irreducibility argument}
Suppose that $\Gamma$ is  irreducible, and let $K\subset \widetilde{\cK}_{\Gamma}$ be a compact subset. For any $\delta>0$, there exists $d_0>0$ such that, for any $[v:\alpha]\in K$ and unit vector $(w,0)\in E^\mathrm{u}_{[v:\alpha]}$, there is a linear form $\beta\in V^*$ satisfying the following properties:
\begin{enumerate} [label=(\arabic*),ref=\emph{(\arabic*)}]
\item \label{lemenumi irreducibility linear form in limit set} $[\beta]\in\Lambda^*_\Gamma$,
\item \label{lemenumi irreducibility linear form close to original} $\beta(v)=1$ and $\norm{(0,\beta-\alpha)}_{[v:\alpha]}\leq \delta$,
\item \label{lemenumi irreducibility uniform bound} $\beta(w)\neq 0$ and $\vert\beta(u)\vert> \frac{1}{2}\vert\beta(w)\vert$ for any  vector $u\in \ker\alpha$ (i.e. $(u,0)\in E^\mathrm{u}_{[v:\alpha]}$) such that $\norm{(u-w,0)}_{[v:\alpha]}< d_0$.
\end{enumerate}
\end{lem}

\begin{proof}
 Let us first replace \ref{lemenumi irreducibility uniform bound} with the weaker statement that $\beta(w)\neq 0$ (note that $d_0$ is not involved is this statement). Consider the  neighbourhood \[ \mathcal N=\set{[\beta]\in\Lambda^*_\Gamma}{\beta(v)=1, \norm{(0,\beta-\alpha)}_{[v:\alpha]}\leq \delta}\]
of $[\alpha]$ in $\Lambda_\Gamma^*$, and assume by contradiction that $\beta(w)=0$ whenever $[\beta]\in \mathcal N$. Then $\mathcal N\subset \ker\iota_w$, hence $\mathrm{Span}(\mathcal N)\subset \ker\iota_w$. We know by Lemma \ref{lem span of open subset of the limit set} that $\mathrm{Span}(\Lambda_\Gamma^*)= \mathrm{Span}(\mathcal N)$, so $\Gamma$ preserves a proper vector subspace, thus contradicting the irreducibility of $\Gamma$.

Let us now prove the Lemma. By contradiction, assume that for any $k\in\N$ there are $[v_k:\alpha_k]\in K$ and a unit vector $(w_k,0)\in E^\mathrm{u}_{[v_k:\alpha_k]}$ not satisfying the claimed properties with $d_0=\frac{1}{k}$. Up to a subsequence, we may assume that $[v_k:\alpha_k]\to[v:\alpha]\in K$ and $w_k\to w\in V$ where $(w,0)\in E^\mathrm{u}_{[v:\alpha]}$ is a unit vector. From the above discussion we can consider $\beta\in V^*$  such that $\beta(w)\neq 0$, $\beta(v)=1$, $[\beta]\in\Lambda^*_\Gamma$ and $\norm{(0,\beta-\alpha)}_{[v:\alpha]}< \delta$. Let $k$ be large enough so that $\beta(w_k)\neq 0$ and $\beta(v_k)\neq 0$, and set $\beta_k=\beta/\beta(v_k)$. If furthermore $k$ is large enough so that $\norm{(0,\beta_k-\alpha_k)}_{[v_k:\alpha_k]}<\delta$, there must be a  vector $u_k\in\ker\alpha_k$ such that $\norm{(u_k-w_k,0)}_{[v_k:\alpha_k]}< \frac{1}{k}$ and $\left\vert\beta(u_k)\right\vert \leq \frac{1}{2}\left\vert\beta(w_k)\right\vert$ (as \ref{lemenumi irreducibility linear form in limit set} and \ref{lemenumi irreducibility linear form close to original} are satisfied, so \ref{lemenumi irreducibility uniform bound} cannot be). This contradicts the convergence $\vert\beta(u_k)\vert=\vert\beta(w_k) - \beta(w_k-u_k)\vert\to\vert \beta(w)\vert$.
\end{proof}

\begin{rem} Lemma \ref{lem irreducibility argument} is the only place where the irreducibility of $\Gamma$ is required. The result fails for Barbot subgroups of $\SL(3,\R)$ (Anosov subgroups obtained by deforming uniform lattices of $\SL(2,\R)$ in $\SL(3,\R)$ within block triangular matrices introduced and studied in \cite{BarbotFlag01,BarbotFlag10}).

\end{rem}

\begin{lem} \label{lem LNIC}
If $\Gamma$ is  irreducible, the flow $\phi^t$ on $\cM_\Gamma$ satisfies the Strong  Local Non-Integrability Condition of Definition \ref{def LNIC}.
\end{lem}

\begin{proof}As before, we denote by $\delta_0,\eps_0,\varepsilon_1>0$ the constants from Lemma \ref{lem: scale setting}. Let $\eps\in(0,\eps_1)$ and $K\subset\widetilde{\cK}_{\Gamma}$  a compact set such that $\Gamma\cdot K=\cK_\Gamma$. Let $d_0>0$ be given by Lemma \ref{lem irreducibility argument} for this compact set $K$ and some $\delta\leq\delta_0$ chosen small enough so that 
\begin{equation} \label{lemeqn LNIC choice of delta}
\exp_x^\mathrm{u}\left(E^\mathrm{u}_x\cap B(0,\delta)\right)\subset W^\mathrm{u}_\varepsilon(x),\quad \forall x\in\cK_\Gamma.
\end{equation}
  Given $z\in\cK_\Gamma$ and a unit vector $w\in E^\mathrm{u}_z$, consider lifts  $\tilde z=[v:\alpha]\in K$ and $(\tilde w,0)\in E^\mathrm{u}_{\tilde z}$, i.e. $\alpha(\tilde w)=0$. Let $\beta\in V^*$ be given by applying Lemma \ref{lem irreducibility argument} to $[v:\alpha]$ and $\tilde w$, then set  $\tilde y=[v:\beta]\in \widetilde{\cK}_{\Gamma}$ and $y=\pi(\tilde y)$. Then $y\in W^\mathrm{s}_{\varepsilon_0}(z)$ by \eqref{lemeqn LNIC choice of delta}. Let $\kappa:=\frac{1}{8L_0}\vert\beta(\tilde w)\vert$. Also choose some $r>0$ such that any $\theta\in\R$ with $\vert\theta\vert\leq r$ satisfies $\vert \log(1+\theta)\vert \geq \frac{1}{2}\vert\theta\vert$.

  Now let $\varepsilon'\in (0,\varepsilon)$ be small enough so that any $x\in W^\mathrm{u}_{\varepsilon'}(z)$ and  $u\in \Lambda^\mathrm{u}(x)$ with $\norm{u}\leq \varepsilon'$ can be lifted to $\tilde x=[v':\alpha]$ and $(\tilde u,0)\in E^\mathrm{u}_{[v':\alpha]}$ satisfying $\beta(v')\leq 2$ and $\left\vert\frac{\beta(\tilde u)}{\beta(v')}\right\vert\leq r$.

 For such $x$ and $u$ also satisfying $\norm{\frac{u_z}{\norm{u_z}}-w}=\norm{\left(\frac{\tilde u}{\norm{u_z}}-\tilde w,0\right)}_{[v:\alpha]}\leq d_0$, we find:
\begin{align*}
\left\vert\Delta(\exp^\mathrm{u}_{x}(u),\pi_{y}(x))\right\vert &=\left\vert \Delta\left(\pi\left(\left[v'+\tilde u:\alpha\right]\right),\pi\left(\left[v':\frac{1}{\beta(v')}\beta\right]\right) \right)\right\vert & \textrm{by }\eqref{eqn time separation and projection}\\
&=\left\vert\log\left(1+ \frac{\beta(\tilde u)}{\beta(v')}\right)\right\vert & \textrm{by }\eqref{eqn time separation and projection}\\
&\geq \frac{1}{2} \left\vert\frac{\beta(\tilde u)}{\beta(v')}\right\vert  & \textrm{because }\left\vert\frac{\beta(\tilde u)}{\beta(v')}\right\vert\leq r\\
&\geq \frac{1}{4}\left\vert \beta(\tilde u)\right\vert & \textrm{because }\beta(v')\leq 2\\
&= \frac{1}{4}\norm{ u_z}\left\vert \beta\left(\frac{\tilde u}{\norm{ u_z}}\right)\right\vert & \\
&\geq \frac{1}{8}\norm{ u_z} \vert\beta(\tilde w)\vert &\textrm{by Lemma }\ref{lem irreducibility argument}\\
&\geq \kappa\norm{ u} &\textrm{by Lemma }\ref{lem: scale setting} \ref{lemenumi norm parallel transport}.
\end{align*}
\end{proof}

\subsection{Dynamical consequences}\label{dyn cons}

In this section, we explain the results of Stoyanov \cite{St11} which we are now prepared to apply, and explain how Theorem \ref{EXP MIX} from the introduction follow from these results.  We warn the reader that this section assumes familiarity with the thermodynamic formalism and Markov coding of Axiom A flows: our only goal is to provide a concise account of the logic that leads to our result on the exponential mixing of $\phi^{t}: \mathcal{K}_{\Gamma}\rightarrow \mathcal{K}_{\Gamma}.$

Let $(\mathcal{M}, g)$ be a complete Riemannian manifold, $\phi^{t}: \mathcal{M}\rightarrow \mathcal{M}$ a $C^{2}$-Axiom A flow, and   $\mathcal{K}\subset \mathcal{M}$ a basic hyperbolic set.  By results of Bowen-Ruelle \cite{BR75} (see also Ratner \cite{RAT73}), there exists a Markov family $\mathcal{R}=\{R_{i}\}_{i=1}^{k}$ over $\mathcal{K}$ consisting of a finite number of rectangles $R_{i}=[U_{i}, S_{i}]$ with $U_{i}\subset W_{\varepsilon}^{\mathrm{u}}(z_{i})\cap \mathcal{K}$ and $S_{i}\subset W_{\varepsilon}^{\mathrm{s}}(z_{i})\cap\mathcal{K}$ admissable subsets where $\varepsilon>0$ is arbitrary and $z_{i}\in \mathcal{K}$ for all $1\leq i \leq k.$  The scale $\varepsilon>0$ is called the \emph{size} of the Markov family.  

Let $R=\coprod_{i=1}^{k} R_{i}$ with the induced topology.  Assuming the local stable $W_{\varepsilon}^{\mathrm{s}}(z_{i})\cap\mathcal{K}$ and unstable $W_{\varepsilon}^{\mathrm{u}}(z_{i})\cap \mathcal{K}$ laminations are Lipschitz, the Poincar\'{e} first return map $\mathcal{P}: R\rightarrow R$ and first return time $\tau: R\rightarrow (0, +\infty)$ are essentially Lipschitz.  Let $U=\coprod_{i=1}^{k} U_{i}$ and $\pi^{U}: R\rightarrow U$ the projection along the leaves of the local stable manifolds.  The shift map is defined by
$$
\sigma:=\pi^{U}\circ \mathcal{P}: U\rightarrow U.
$$
Let $\widehat{U}\subset U$ denote the subset of points $u\in U$ whose orbits do not have points in common with the boundary of $R.$  

Given a Lipschitz function $f: \widehat{U}\rightarrow \mathbb{R},$ let $P_{f}\in \mathbb{R}$ be the unique real number such that 
$$
\mathrm{Pr}_{\sigma}(f-P_{f}\tau)=0,
$$
where $\mathrm{Pr}_{\sigma}$ is the topological pressure with respect to the shift map $\sigma.$  Given a complex valued Lipschitz function $g: \widehat{U}\rightarrow \mathbb{C},$ the Ruelle transfer operator
$$
L_{g}: C^{\mathrm{Lip}}(\widehat{U}, \mathbb{C})\rightarrow C^{\mathrm{Lip}}(\widehat{U}, \mathbb{C})
$$
is defined by
\begin{align}\label{transfer op}
L_{g}(h)(x)=\sum_{\sigma(y)=x} e^{g(y)}h(y).
\end{align}
For a Lipschitz function $h\in C^{\mathrm{Lip}}(\widehat{U}, \mathbb{R})$ and $b\in \mathbb{R}\setminus \{0\},$ introduce the norm
$$
\lVert h\rVert_{\mathrm{Lip}, b}=\lVert h\rVert_{\infty} + \frac{\mathrm{Lip}(h)}{\lvert b\rvert}
$$
where $\mathrm{Lip}(h)$ is the optimal Lipschitz constant of $h$ and $\lVert h\rVert_{\infty}$ is the $L^{\infty}$-supremum norm.
\begin{definition}
Let $f: \widehat{U}\rightarrow \mathbb{R}$ be a Lipschitz function.  The Ruelle transfer operators related to $f$ are eventually contracting if for every $\varepsilon>0,$ there exists $0<\rho<1$ and $a_{0}, C>0$ such that if $a,b\in \mathbb{R}$ satisfy $\lvert a \rvert\leq a_{0}$ and $\lvert b \rvert \geq \frac{1}{a_{0}},$ then for every positive integer $m>0$ and every $h\in C^{\mathrm{Lip}}(\widehat{U}, \mathbb{R})$ we have
$$
\lVert L_{f-(P_{f}+a+ib)\tau}^{m}(h)\rVert_{\mathrm{Lip}, b}\leq C\cdot \rho^{m}\cdot \lvert b\rvert^{\varepsilon}\cdot \lVert h\rVert_{\mathrm{Lip}, b}.
$$
\end{definition}
If the Ruelle transfer operators related to $f\in C^{\mathrm{Lip}}(\widehat{U}, \mathbb{R})$ are eventually contracting, then in particular the spectral radius of the bounded linear operator $L_{f-(P_{f}+a+ib)\tau}$ on $C^{\mathrm{Lip}}(\widehat{U}, \mathbb{R})$ is at most $\rho<1.$  
We can finally state the main result of this section.
\begin{thm}\label{contracting ruelle}
Let $\Gamma<\mathrm{SL}(V)$ be an irreducible, torsion-free projective Anosov subgroup with associated Axiom A dynamical system $(\mathcal{M}_{\Gamma}, \phi^{t})$ and basic hyperbolic set $\mathcal{K}_{\Gamma}\subset \mathcal{M}_{\Gamma}$.  Then for every $\varepsilon>0,$ there exists a Markov family $\mathcal{R}=\{R_{i}\}_{i=1}^{k}$ for $\phi^{t}$ over $\mathcal{K}_{\Gamma}$ of size $\varepsilon>0$ such that for any $f\in C^{\mathrm{Lip}}(\widehat{U}, \mathbb{R}),$ the Ruelle transfer operators related to $f$ are eventually contracting.
\end{thm}
\begin{proof}
We apply the results of Stoyanov \cite{St11}: if a $C^{2}$-Axiom A flow $\phi^{t}: \mathcal{M}\rightarrow \mathcal{M}$ along a basic hyperbolic set $\mathcal{K}\subset \mathcal{M}$ satisfies the condition (LNIC) (Definition \ref{def LNIC} and Remark \ref{rem LNIC vs SLNIC}), has  regular distortion along unstable manifolds (Definition \ref{def: URDU} and Remark \ref{rem RDU vs URDU}), and the local holonomy maps along the stable laminations through $\mathcal{K}$ (i.e. the restriction of the map $\cH_x^y$ introduced in Definition \ref{def stable holonomy} to $W^{\mathrm{u}}_{\varepsilon_1}(x)\cap\cK_\Gamma$) are uniformly Lipschitz, then the Ruelle transfer operators related to $f\in C^{\mathrm{Lip}}(\widehat{U}, \mathbb{R})$ are eventually contracting.  By Corollary \ref{cor regular distortion along unstable manifolds}, the real analytic (hence $C^{2})$ Axiom A flow $\phi^{t}: \mathcal{M}_{\Gamma}\rightarrow \mathcal{M}_{\Gamma}$ has regular distortion along unstable manifolds over $\mathcal{K}_{\Gamma}.$  Furthermore, Lemma \ref{lem LNIC} verifies the (LNIC) condition.  
Finally, the local holonomy maps along the stable laminations through $\cK_\Gamma$ are restrictions of   smooth maps to compact subsets (smoothness being a consequence of the explicit formula \eqref{eqn formula stable holonomy}), hence uniformly Lipschitz.   Having verified the hypotheses of the result of Stoyanov \cite{St11}, we have completed the proof.
\end{proof}
Now we can finally prove Theorem \ref{EXP MIX} from the introduction.

\begin{thm}\label{thm: cor exp decay}
Suppose $\Gamma<\SL(V)$ is a torsion-free irreducible projective Anosov subgroup with associated $\mathcal{K}_{\Gamma}\subset \mathcal{M}_{\Gamma}$ from Theorem \ref{THM A}.  For every $\alpha>0$ and H\"{o}lder potential $U\in C^{\alpha}(\mathcal{K}_{\Gamma}, \R),$ there exists $c_{\alpha}(U), C_{\alpha}(U)>0$ such that the correlation functions (see Definition \ref{corr functions}) decay exponentially:
$$
\forall\,t\in \R:\;c^{t}(F, G, U)\leq C_{\alpha}(U)e^{-c_{\alpha}(U)|t|}\lVert F \rVert_{\alpha}\lVert G \rVert_{\alpha}
$$
for all H\"{o}lder observables $F,G\in C^{\alpha}(\mathcal{K}_{\Gamma}, \R).$  In other words, the restriction of the Axiom A flow $\phi^{t}: \mathcal{K}_{\Gamma}\rightarrow \mathcal{K}_{\Gamma}$ mixes exponentially for all H\"{o}lder observables $F,G$ with respect to every Gibbs equilibrium state with H\"{o}lder potential $U$.
\end{thm}
\begin{proof}
Using Stoyanov \cite[Cor.~1.5]{St11}, the above exponential mixing result is a direct consequence of Theorem \ref{contracting ruelle}.
\end{proof}

\section{Zeta Functions, Orbit Counting and Dynamical Resonances}\label{sec:Resonances}

Throughout this section, $\Gamma<\SL(V)$ denotes a non-trivial torsion-free projective Anosov subgroup and $(\mathcal{M}_{\Gamma},\phi^t)$ is the Axiom A system with basic set $\mathcal K_\Gamma$ provided by Theorem \ref{THM A}.  

\subsection{Zeta functions}
 By Lemma \ref{lem periodic points in discontinuity set}, the periodic orbits of the Axiom A flow $\phi^{t}: \mathcal{M}_{\Gamma}\rightarrow \mathcal{M}_{\Gamma}$ are in bijection with the non-trivial conjugacy classes $[\gamma]\in [\Gamma]$ and have periods
\begin{align}
\lambda_{1}: [\Gamma]&\rightarrow [0, +\infty) \\
  [\gamma] &\mapsto \lambda_{1}([\gamma]),
\end{align}
where $\lambda_{1}([\gamma])$ is the natural logarithm of the modulus of the largest eigenvalue of the linear transformation $\gamma\in \Gamma<\SL(V).$  A non-trivial element $\gamma\in \Gamma$ is primitive if whenever $\gamma=\eta^{k}$ then $k=1,$ and we denote the set of non-trivial conjugacy classes of primitive elements by $[\Gamma]_{\mathrm{prim}}.$  The Ruelle zeta function is formally defined by the Euler product expansion
$$
\zeta_{\Gamma}(s)=\prod_{[\gamma]\in [\Gamma]_{\mathrm{prim}}}\left(1-e^{-s\lambda_{1}([\gamma])}\right)^{-1}.
$$
The Euler product converges for all $s\in \mathbb{C}$ satisfying $\mathrm{Re}(s)>h_{\mathrm{top}}(\Gamma)>0$.  It was first proved by Parry-Pollicott \cite{PP83} that any zeta function associated to a non-empty basic set of a weakly mixing Axiom A flow admits a meromorphic continuation to a neighborhood of $h_{\mathrm{top}}(\Gamma)\in \mathbb{C}$ with a simple pole at $h_{\mathrm{top}}(\Gamma).$

If $U\in C^{\alpha}(\mathcal{K}_{\Gamma}, \R)$ is a H\"{o}lder potential, let 
$$
U_{[\gamma]}: [\Gamma]\rightarrow \R
$$
be defined by the averages
$$
U_{[\gamma]}=\frac{1}{\lambda_{1}([\gamma])}\int_{0}^{\lambda_{1}([\gamma])} U(\phi^{t}(x)) \ dt
$$
where $x\in \mathcal{K}_{\Gamma}$ belongs to the period corresponding to $[\gamma]\in [\Gamma].$  The associated \emph{weighted} zeta function is defined via
\begin{align}\label{euler potential}
\zeta_{\Gamma, U}(s)=\prod_{[\gamma]\in [\Gamma]_{\mathrm{prim}}}\left(1-\exp{\left[-\lambda_{1}([\gamma])(s - U_{[\gamma]})\right]}\right)^{-1}.
\end{align}
As before, the Euler product \eqref{euler potential} converges for all $\{s\in \C \ | \ \mathrm{Re}(s)>\mathrm{Pr}_{\Gamma}(U)\}$ where $\mathrm{Pr}_{\Gamma}(U)$ is the topological pressure of the H\"{o}lder potential $U$.   In light of Theorem \ref{thm: cor exp decay}, the following partial meromorphic continuation result is an immediate consequence of Pollicott \cite[Thm.~2]{POL85}.
\begin{thm}\label{zeta potential strip}
Let $U\in C^{\alpha}(\mathcal{K}_{\Gamma}, \R)$ for some $0<\alpha<1$  and suppose that $\Gamma$ is irreducible.  Then there exists $\varepsilon>0$ such that $\zeta_{\Gamma, U}$ has an analytic extension to the punctured half-plane $\{s\in \mathbb{C}\setminus \{\mathrm{Pr}_{\Gamma}(U)\}\ | \ \mathrm{Pr}_{\Gamma}(U)-\varepsilon<\mathrm{Re}(s)\}$ with a simple pole at $\mathrm{Pr}_{\Gamma}(U)\in \C.$
\end{thm}
\begin{proof}
For ease of reading, we sketch the basic idea and refer to \cite[Thm.~2]{POL85} for a complete exposition.  Consider the Fourier transform of the correlation function
$$
\widehat{c}^{t}(F,G,U)(s)=\frac{1}{2\pi}\int_{t\in \R} e^{ist}c^{t}(F,G,U) \ dt.
$$
The key fact (proved in \cite{POL85}) is that $\zeta_{\Gamma, U}$ has an analytic extension to the punctured half-plane $\{s\in \mathbb{C}\setminus \{\mathrm{Pr}_{\Gamma}(U)\}\ | \ \mathrm{Pr}_{\Gamma}(U)-\varepsilon<\mathrm{Re}(s)\}$ with a simple pole at $\mathrm{Pr}_{\Gamma}(U)\in \C$ if and only if $\widehat{c}^{t}(F,G,U)(s)$ admits a holomorphic continuation to the strip $\lvert \mathrm{Im}(s)\rvert<\varepsilon.$  Since the correlation functions decay exponentially (Theorem \ref{thm: cor exp decay}), an application of the Payley-Wiener theorem guarantees that $\widehat{c}^{t}(F,G,U)(s)$ admits a holomorphic continuation to the strip $\lvert \mathrm{Im}(s)\rvert<\varepsilon$ for some $\varepsilon>0,$ and the result follows.
\end{proof}

Let $C_{\mathcal{K}_{\Gamma}}^{\infty}(\mathcal{M}_{\Gamma}, \R)$ denote the $\R$-vector space of germs of smooth functions along the closed subset $\mathcal{K}_{\Gamma}\subset \mathcal{M}_{\Gamma}.$  There is an obvious injective map $C_{\mathcal{K}_{\Gamma}}^{\infty}(\mathcal{M}_{\Gamma}, \R)\hookrightarrow C^{\mathrm{Lip}}(\mathcal{K}_{\Gamma}, \R).$
By Theorem \ref{THM A}, the zeta function $\zeta_{\Gamma,U}$ is associated to an Axiom A flow, therefore the following global meromorphic continuation result is an immediate application of the resolution of Smale's conjecture by Dyatlov-Guillarmou \cite{DG16, DG18} and Borns-Weil-Shen \cite{BWS21}. 
\begin{thm}\label{thm:thmzetasmooth}
Suppose that $U\in C_{\mathcal{K}_{\Gamma}}^{\infty}(\mathcal{M}_{\Gamma}, \R)$.  Then $\zeta_{\Gamma, U}$ admits a meromorphic continuation to $\C$ with a simple pole at $\mathrm{Pr}_{\Gamma}(U)\in \C.$
\end{thm}
Note that in Theorem \ref{thm:thmzetasmooth}, there is no requirement that $\Gamma< \SL(V)$ is irreducible.

Next, we establish that Theorem \ref{contracting ruelle} implies that the strip in Theorem \ref{zeta potential strip} for the ordinary zeta function $\zeta_{\Gamma}$ is zero-free.  The reader may consult \cite{DP98}, \cite{BAL98}, \cite{PS98} and \cite{NAU05} for a general overview and background information.
\begin{thm}\label{RH}
Suppose that $\Gamma$ is irreducible.  Then there exists $\varepsilon>0$ such that the zeta function $\zeta_{\Gamma}$ is analytic and zero-free in the strip $h_{\mathrm{top}}(\Gamma)-\varepsilon<\mathrm{Re}(s)<h_{\mathrm{top}}(\Gamma).$
\end{thm}
\begin{proof}
Again, for ease of reading we provide a sketch of the main ideas borrowed from \cite{DP98}.  Throughout this discussion, we allow $\varepsilon>0$ to be recalibrated as necessary.  It already follows from Theorem \ref{zeta potential strip} that $\zeta_{\Gamma}$ is analytic in a strip $h_{\mathrm{top}}(\Gamma)-\varepsilon<\mathrm{Re}(s)<h_{\mathrm{top}}(\Gamma)$ for some $\varepsilon>0$.  As in section \ref{dyn cons}, we fix a Markov partition derived from rectangles $R=\coprod_{i=1}^{k} R_{i}$ with first return map $\mathcal{P}: R\rightarrow R$ and first return time $\tau: R\rightarrow (0, +\infty),$ both of which are essentially Lipschitz.  Define the formal \emph{symbolic} zeta function
\begin{align}\label{symbolic zeta}
\zeta_{0}(s)=\exp \left[\sum_{n=1}^{\infty} \frac{1}{n} \sum_{\mathcal{P}^{n}(x)=x} \exp\left(-s \sum_{j=0}^{n-1} \tau\left(\mathcal{P}^{j}(x)\right)\right)\right].
\end{align}
It was shown by Bowen \cite[Sec.~5]{BOW75} that $\zeta_{0}(s)=\zeta_{\Gamma}(s)\eta(s)$ where $\eta(s)$ is a nowhere vanishing holomorphic function for $\mathrm{Re}(s)>h_{\mathrm{top}}(\Gamma)-\varepsilon$ for some $\varepsilon>0.$  The function $\eta$ is a correction to account for the overcounting of periodic orbits which intersect the boundary of the rectangles $R_{i}$ for $1\leq i \leq k.$  

Using a trick due to Ruelle (see \cite[pg.~4-6]{DP98}), the (suspended) Markov chain $(R, \mathcal{P}, \tau)$ admits a reduction to:
\begin{enumerate}
\item A finite disjoint union of intervals $I=\coprod_{i=1}^{k} I_{i}$ with $I_{i}\subset \R.$ 
\item An expanding essentially Lipschitz map $T: I\rightarrow I.$
\item An essentially Lipschitz function $r: I\rightarrow (0, +\infty).$
\end{enumerate} 
These reductions do not alter the symbolic zeta function \ref{symbolic zeta}.  
The corresponding ergodic sums satisfy the identity
\begin{align}\label{erg sums}
\sum_{T^{n}(x)=x} \exp\left(-s \sum_{j=0}^{n-1} r\left(T^{j}(x)\right)\right)=\sum_{i=1}^{k} \left(L_{-sr}^{n}(\chi_{i})\right)(x_{i})
\end{align}
for all $n\geq 0$ and $x_{i}\in I_{i}$ well-chosen.  In \eqref{erg sums}, $\chi_{i}$ is the characteristic function of the interval and $L_{-sr}$ is the appropriate Ruelle transfer operator defined by the equation \ref{transfer op}.

Using the identity \ref{erg sums} and the estimates on the transfer operators from Theorem \ref{contracting ruelle}, the symbolic zeta function \ref{symbolic zeta} converges uniformly to a nowhere vanishing holomorphic function for $s\in \C$ satisfying $h_{\mathrm{top}}(\Gamma)-\varepsilon<\mathrm{Re}(s)<h_{\mathrm{top}}(\Gamma)$ and $\lvert \mathrm{Im}(s) \rvert>2.$   Recall the relation $\zeta_{0}(s)=\zeta_{\Gamma}(s)\eta(s)$, this implies $\zeta_{\Gamma}(s)$ converges uniformly to a nowhere vanishing holomorphic function on the region $h_{\mathrm{top}}(\Gamma)-\varepsilon<\mathrm{Re}(s)<h_{\mathrm{top}}(\Gamma)$ and $\mathrm{Im}(s)>2.$  By Theorem \ref{zeta potential strip}, the zeta function $\zeta_{\Gamma}$ is holomorphic on $\mathrm{Re}(s)>h_{\mathrm{top}}(\Gamma)-\varepsilon$ except for a simple pole at $h_{\mathrm{top}}(\Gamma)\in \C,$ and therefore there are finitely many zeros and a single pole in the compact region $\{s\in \C \ | \ h_{\mathrm{top}}(\Gamma)-\varepsilon\leq \mathrm{Re}(s)\leq h_{\mathrm{top}}(\Gamma),\ \lvert \mathrm{Im}(s)\rvert\leq 2\}.$  One final shrinking of $\varepsilon>0$ shows that $\zeta_{\Gamma}$ is nowhere vanishing for $\mathrm{Re}(s)>h_{\mathrm{top}}(\Gamma)-\varepsilon.$
\end{proof}
\begin{rem}Given that the indispensible ingredient above is the estimate on Ruelle transfer operators from Theorem \ref{contracting ruelle}, it is likely that the above result can be extended to the weighted zeta functions $\zeta_{\Gamma, U}$ for all $U\in C^{\mathrm{Lip}}(\mathcal{K}_{\Gamma}, \R),$ and by an approximation argument to all $U\in C^{\alpha}(\mathcal{K}_{\Gamma}, \R)$.
\end{rem}  
\subsubsection{Orbit counting}
Define the orbit counting function
$$
N_{\Gamma}(t):=\{[\gamma]\in [\Gamma]_{\mathrm{prim}} \ | \ \lambda_{1}([\gamma])\leq t\}
$$
using the offset logarithmic integral \eqref{eq:offsetintegral}. The following prime orbit theorem with exponentially decaying error term is a consequence of Theorems \ref{contracting ruelle} and \ref{RH} (see \cite[Cor.~1.4]{St11}).
\begin{thm}\label{thm:orbitcounting}
Suppose that $\Gamma$ is  irreducible.  Then there exists $0<c<h_{\mathrm{top}}(\Gamma)$ such that
$$
N_{\Gamma}(t)=\mathrm{Li}\big(e^{h_{\mathrm{top}}(\Gamma)t}\big)\left(1+ O(e^{-(h_{\mathrm{top}}(\Gamma)-c)t})\right).
$$
\end{thm}

\subsection{Dynamical resonances}\label{sec:dynres}
Although the generator $X$ of the flow $\phi^t$ is not an elliptic differential operator (because its principal symbol vanishes on $E^\mathrm{s}\oplus E^\mathrm{u}$), it has a rich and well-studied spectral theory. This applies more generally to lifts of $X$ to vector bundles over $\mathcal{M}_{\Gamma}$, which we do not consider here for the sake of brevity. For a general overview, we refer the reader to the introduction of the paper \cite{DG16} by Dyatlov-Guillarmou. We will present only a selection of special cases of the results obtained by applying the work of Dyatlov-Guillarmou \cite{DG16, DG18} to our operator $X$,  a  detailed and systematic study being the subject of forthcoming works.

Define the complex vector space $L^2(\mathcal{M}_{\Gamma})$ using the volume form provided by the contact form on $\mathcal{M}_{\Gamma}$, and embed the space $\CT(\mathcal M_\Gamma)$ of $\C$-valued test functions into its dual $\D'(\mathcal M_\Gamma)$ using the sesquilinear $L^2$-pairing. Let $U\in \CT(\mathcal M_\Gamma)$, acting on $L^2(\mathcal{M}_{\Gamma})$  by multiplication. 

The spectral analysis of the differential operator 
$$\mathbf X:=-X+U$$
begins with the observation that for each $\lambda \in \C$ with $\Re \lambda>0,$ there is an  $L^2$-resolvent
\[
(\mathbf X-\lambda)^{-1}:L^2(\mathcal{M}_{\Gamma})\to L^2(\mathcal{M}_{\Gamma})
\]
given by the explicit formula 
\[
(\mathbf X-\lambda)^{-1}f(x)= -\int_{0}^\infty \exp\bigg(\int_{0}^{t}U(\phi^{-s}(x))\d s\bigg) \,e^{-\lambda t}f(\phi^{-t}(x))\d t,
\]
where $dt$ is the Lebesgue measure. When applying the works \cite{DG16, DG18}, one localizes the resolvent of $\mathbf X$ to an appropriate neighborhood of the basic set $\mathcal{K}_{\Gamma}$ and in addition modifies $\mathbf X$ slightly near the boundary of the neighborhood to ensure that the strict convexity assumption \cite[(A3)]{DG16}  is satisfied. This yields the following result: 
\begin{thm}\label{Resonant states precise}There exists a relatively compact open set $\mathcal{B}_{\Gamma}\subset \mathcal{M}_{\Gamma}$ with $\mathcal{K}_{\Gamma}\subset \mathcal{B}_{\Gamma}$ and a nowhere vanishing smooth vector field $X_0$ on $\mathcal{B}_{\Gamma}$, which agrees with $X$ in a neighborhood of $\mathcal{K}_{\Gamma}$, such that the following holds: If $U\in \CT(\mathcal{B}_{\Gamma})$ and we put $\mathbf X_0:=-X_0+U$, then  the resolvent $(\mathbf X_0-\lambda)^{-1}$, regarded as an operator $\CT(\mathcal{B}_{\Gamma})\to \D'(\mathcal{B}_{\Gamma})$, extends to $\lambda\in \C$ as an operator-valued meromorphic function
\[
R(\lambda):\CT(\mathcal{B}_{\Gamma})\to \D'(\mathcal{B}_{\Gamma}).
\]
For each pole $\lambda_0\in \C$ of the meromorphic family of operators $R(\lambda)$, the residue operator
\[
\Pi_{\lambda_0}:=\mathrm{Res}_{\lambda=\lambda_0}R(\lambda):\CT(\mathcal{B}_{\Gamma})\to \D'(\mathcal{B}_{\Gamma})
\]
is a finite rank projection commuting with $\mathbf{X}_0$ satisfying $\mathrm{Ran}(\Pi_{\lambda_0})\subset \ker (\mathbf X_0-\lambda_0)^{J(\lambda_0)}$ for some $J(\lambda_0)\geq 1,$ and $\mathrm{Ran}(\Pi_{\lambda_0})\cap \ker (\mathbf X_0-\lambda_0)\neq \{0\}$.
\end{thm}
The poles $\lambda_0$ of the meromorphic operator family $R(\lambda)$ are called \emph{Ruelle-Pollicott resonances} (or \emph{dynamical resonances}/\emph{classical resonances}), the rank $m_{\lambda_0}\geq 1$ of the projection $\Pi_{\lambda_0}$ is called the \emph{multiplicity} of the resonance $\lambda_0$,  the distributions in $\mathrm{Ran}(\Pi_{\lambda_0})$ are called \emph{generalized resonant states}, and those in $\mathrm{Ran}(\Pi_{\lambda_0})\cap \ker (\mathbf X_0-\lambda_0)$ are called \emph{resonant states}.  

 The (non-canonical) relatively compact open set $\mathcal{B}_{\Gamma}\subset \mathcal{M}_{\Gamma}$ is an \emph{isolating block} for the basic set $\mathcal{K}_{\Gamma}$ in the sense of Conley-Easton \cite{CE71} (see also \cite{DG18} and \cite{GMT21}). As in \cite[proof of Prop.~6.2]{DG16}, one can show that the set of Ruelle-Pollicott resonances and their multiplicities are independent of the choice of the open set $\mathcal{B}_{\Gamma}$ containing $\mathcal{K}_{\Gamma}$ and the modification $\mathbf{X}_0$ of the operator $\mathbf{X}$, so that these resonances form an intrinsic discrete spectrum $\mathcal{R}^{\mathbf X}_{\Gamma}\subset \C$ of complex numbers canonically associated to the Axiom A flow $\phi^t$ and thus to the projective Anosov subgroup $\Gamma<\SL(V)$. 
 
Note that the work of Meddane \cite{Med21} avoids working with tight neighborhoods of the basic set and modifications of the flow generator $X$. However, it is formulated in the setting of compact manifolds without boundary, so it is not directly applicable to our situation.

Applying the recent results of Jin-Tao \cite[Thm.~1.2]{JT23} (again in combination with \cite{DG16, DG18}) to our Axiom A flow provides quantitative bounds for the number of Ruelle-Pollicott resonances:
\begin{thm}\label{thm:resbounds}For $U=0$, the spectrum of Ruelle-Pollicott resonances $\mathcal{R}^{\mathbf X}_{\Gamma}$ satisfies the following estimates:
\begin{enumerate}
\item For every $\beta>0$ there exists $C>0$ such that for every $R>0$ one has
\[
\# \mathcal{R}^{\mathbf X}_{\Gamma}\cap\{\lambda\in \C \,|\, |\lambda|\leq R, \Im \lambda>-\beta\}\leq CR^{\dim \mathcal{M}_{\Gamma}+1}+C.
\]
\item For every $\delta\in (0,1)$ there exist $\beta>0$ and $C>0$ such that for every $R>0$ one has  
\[
\# \mathcal{R}^{\mathbf X}_{\Gamma}\cap\{\lambda\in \C \,|\, |\lambda|\leq R, \Im \lambda>-\beta\}\geq \frac{1}{C}R^{\delta}-C.
\]
\end{enumerate}
In particular, there are infinitely many Ruelle-Pollicott resonances.
\end{thm}

\subsubsection{Distributions supported on the basic set $\mathcal{K}_{\Gamma}$}

Theorem \ref{Resonant states} applies in an analogous form to the formal adjoint $\mathbf X^\ast= X+\overline U$ of $\mathbf X$ (where $\overline{\phantom{x}}$ denotes complex conjugation), which associates to every Ruelle-Pollicott resonance $\lambda_0$ another finite-dimensional space of distributions called  \emph{(generalized) coresonant states}, and the works \cite{DG16, DG18} imply:
\begin{thm}\label{Resonant states precise2}
Let  $\lambda_0\in\mathcal{R}^{\mathbf X}_{\Gamma}$ be a Ruelle-Pollicott resonance. Then the following holds:
\begin{enumerate}
\item If $u,v\in \D'(\mathcal{B}_{\Gamma})$ are a resonant and a coresonant state of the resonance $\lambda_0$, respectively, then the wave front sets of the distributions $u$ and $v$ satisfy the H\"ormander criterion \cite[Thm.~8.2.10]{HOR03}, so that the product
\[
u\cdot \overline{v}\in \D'(\mathcal{B}_{\Gamma})
\]
is well-defined. The distribution $u\cdot \overline{v}$ satisfies $\supp(u\cdot \overline{v})\subset\mathcal{K}_{\Gamma}$ and $X(u\cdot \overline{v})=0$, i.e., it is $\phi^t$-invariant. 
\item The Schwartz kernel $K_{\Pi_{\lambda_0}}\in \D'(\mathcal{B}_{\Gamma}\times\mathcal{B}_{\Gamma})$ of the projection $\Pi_{\lambda_0}$ has the property that its wave front set is disjoint from the conormal bundle of the diagonal submanifold in $\mathcal{B}_{\Gamma}\times\mathcal{B}_{\Gamma}$, so that by  \cite[Thm.~8.2.4]{HOR03} its pullback along the diagonal map $\iota_{\mathrm{diag}}:\mathcal{B}_{\Gamma}\to\mathcal{B}_{\Gamma}\times\mathcal{B}_{\Gamma}$ is well-defined. The obtained distribution
\[
\mathcal T_{\lambda_0}:=\iota_{\mathrm{diag}}^\ast K_{\Pi_{\lambda_0}}\in \D'(\mathcal{B}_{\Gamma})
\]
is $\phi^t$-invariant and satisfies $\mathrm{supp}(\mathcal T_{\lambda_0})\subset\mathcal{K}_{\Gamma}$.
\end{enumerate}
\end{thm}
 The distribution $\mathcal T_{\lambda_0}$ is called the \emph{invariant Ruelle distribution} associated to the resonance $\lambda_0$, c.f.\ \cite{GHW21}.  If $\lambda_0$ is a simple resonance (i.e., $m(\lambda_0)=1$), so that up to scalar multiples there are a unique resonant state $u_{\lambda_0}$ and a unique coresonant state $v_{\lambda_0}$ for this resonance, then by rescaling one can arrange that $\mathcal T_{\lambda_0}= u_{\lambda_0}\cdot \overline{v}_{\lambda_0}$.

\subsubsection{Relation to the Ruelle zeta function}

The zeroes and poles of the Ruelle zeta function $\zeta_{\Gamma,U}$ considered in Theorem \ref{thm:thmzetasmooth} are Ruelle-Pollicott resonances of lifts of $X$ to operators $\mathbf X$ acting on appropriate spaces of differential forms, to which generalized versions of the above results apply (again, see \cite{DG16, DG18,BWS21}). For the simple leading pole $\lambda_0=\mathrm{Pr}_{\Gamma}(U)$ of the zeta function,  it is a folklore result (mentioned for example in \cite[p.~6]{DG16} in the case $U=0$) that the invariant Ruelle distribution $\mathcal T_{\lambda_0}$ is a constant multiple of the Gibbs equilibrium measure with respect to the potential $U$, which in the case $U=0$ is the Bowen-Margulis measure of maximal topological entropy.  See the recent work \cite{TH24} for a proof in the Anosov case.

\section{Examples}\label{Examples}

Let us now describe two settings in which the Axiom A flow $\phi^t$ on $\cM_\Gamma$ provided by Theorem \ref{THM A} is directly related to a geodesic flow: $\mathbb H^{p,q}$-convex cocompact groups  and Benoist groups. Both are particular instances of projectively strongly convex-cocompact groups as introduced by Danciger-Guéritaud-Kassel \cite{dgk18a}.

To begin, we fix some terminology for the following subsections. A hyperplane $H\subset V$ defines an \emph{affine patch} $\mathbb A_H=\pi(V\setminus H)$, where $\pi:V\setminus\{0\}\to \bbP(V)$ is the projection. 
\begin{definition} \label{def properly convex} A subset $C\subset \bbP(V)$ is called \emph{properly convex} if it is contained, convex and bounded in an affine patch.
\end{definition}

\subsection{\texorpdfstring{$\mathbb H^{p,q}$-}{Pseudo-hyperbolic } convex cocompact groups}

Consider two integers $p,q$ with $p\geq 1$ and $q\geq 0$, and equip the vector space $V=\R^{p+q+1}$ with the non-degenerate bilinear form $\eklm{\cdot,\cdot}$ of signature $(p,q+1)$ defined by
\[
\eklm{x,y}_{p,q+1}:=\sum_{j=1}^p x_jy_j - \sum_{j=p+1}^{d} x_jy_j\quad \forall x,y\in \R^{p+q+1}.\]
The pseudo-hyperbolic space $\bbH^{p,q}$ is defined as the projectivization of \emph{timelike} vectors, i.e.
\[\bbH^{p,q}:=\set{[x]\in \bbP(V)}{\eklm{x,x}_{p,q+1}<0}.\] 
The geometry of $\bbH^{p,q}$ is best described by considering its double cover
\[ \cH^{p,q}:=\set{x\in V}{\eklm{x,x}_{p,q+1}=-1}.\]
The restriction of $\eklm{\cdot,\cdot}_{p,q+1}$ to tangent spaces  $T_x\cH^{p,q}=x^\perp :=\{v\in V\,|\,\eklm{x,v}_{p,q+1}=0\}$ endows $\cH^{p,q}$ with a pseudo-Riemannian metric of signature $(p,q)$ invariant under both the $\SO(p,q+1)$-action and the antipodal map. It therefore descends to an $\SO(p,q+1)$-invariant pseudo-Riemannian metric on $\bbH^{p,q}$. Since the differential $d\pi:T\cH^{p,q}\to T\bbH^{p,q}$ is also a double cover, we may describe the tangent bundle of $\bbH^{p,q}$ as
%
%
%
\[ T\bbH^{p,q}=\set{[x:v]\in \bbP(V\times V)}{\eklm{x,x}_{p,q+1}<0,\eklm{x,v}_{p,q+1}=0}. \]
With this description, the pseudo-Riemannian metric is expressed as
\[ \left(\rule{0pt}{12pt}[x:v_1],[x:v_2]\right)_{[x]}=-\frac{\eklm{v_1,v_2}_{p,q+1}}{\eklm{x,x}_{p,q+1}}. \]
If $q=0$, we recover the (Klein model of the) real hyperbolic space $\bbH^p=\bbH^{p,0}$. If $q=1$, we find the projective model of the anti-de Sitter space $\mathbb A\mathrm d\mathbb S^{p+1}=\bbH^{p,1}$.
\begin{definition} The \emph{spacelike unit tangent bundle} of $\bbH^{p,q}$ is
\begin{align*} T^1\bbH^{p,q}&=\set{[x:v]\in T\bbH^{p,q}}{([x:v],[x:v])_{[x]}=1}\\
&=\set{[x:v]\in \bbP(V\times V)}{\eklm{x,x}_{p,q+1}<0, \eklm{x,v}_{p,q+1}=\eklm{v,v}_{p,q+1}+\eklm{x,x}_{p,q+1}=0}.
\end{align*}
\end{definition}
The geodesic flow $\varphi^t:T\bbH^{p,q}\to T\bbH^{p,q}$ leaves $T^1\bbH^{p,q}$ invariant, and we find
\[ \varphi^t([x:v])=[ \cosh t\, x+\sinh t\, v : \sinh t\, x+ \cosh t\, v]\quad \forall [x:v]\in T^1\bbH^{p,q}. \]
Flow lines of this \emph{space-like geodesic flow} have endpoints in the boundary 
\[ \partial \bbH^{p,q}=\set{[x]\in \bbP(V)}{\eklm{x,x}_{p,q+1}=0}. \]
They are given, for $[x:v]\in T^1\bbH^{p,q}$, by
\bqn 
[x:v]_\pm:=\lim_{t\to \pm\infty}\pi(\varphi^t([x:v]))=[x\pm v]\in \partial \bbH^{p,q}.
\eqn
Note that we always have $[x:v]_+\neq [x:v]_-$.

\subsubsection{Integration into the general setting}
We will now see that the spacelike geodesic flow $(T^1\bbH^{p,q},\varphi^t)$ of $\bbH^{p,q}$ is a subsystem of $(\bbL,\phi^t)$. The isomorphism
\[ \Phi^{p,q+1}:\map{V}{V^*}{\left[v\right]}{\left[\eklm{v,\cdot}_{p,q+1}\right]} \]
allows us to replace the flow space
\[\bbL=\set{[v:\alpha]\in\bbP(V\times V^*)}{\alpha(v)>0} \]
and the flow
\[ \phi^t([v:\alpha])=[e^tv:e^{-t}\alpha] \]
with the flow space
\[
\bbL^{p,q+1}=\set{[v_1:v_2]\in \mathbb{P}(V\times V)}{\eklm{v_1,v_2}_{p,q+1}>0 }
\]
equipped with the flow
 \[ \phi^t([v_1:v_2])=([e^tv_1:e^{-t}v_2]),\qquad t\in \R. \]
The map 
\[ \Phi^{p,q+1}_\partial:\map{T^1\bbH^{p,q}}{\bbL^{p,q+1}}{\left[x:v\right]}{\left[ x+v:x-v\right]} \]
is an $\SO(p,q+1)$-equivariant embedding that intertwines the flows $\Phi^{p,q+1}_\partial\circ \varphi^t=\phi^t\circ \Phi^{p,q+1}_\partial$. Its image is the subspace $\bbL^{p,q+1}_\partial\subset\bbL$ defined by
\[
\bbL^{p,q+1}_\partial=\set{[v_1:v_2]\in\bbL^{p,q+1}}{[v_1],[v_2]\in \partial\bbH^{p,q}}
\]
and the inverse $\left(\Phi^{p,q+1}_\partial\right)^{-1}:\bbL^{p,q+1}_\partial\to T^1\bbH^{p,q}$ is given by
\[ \left(\Phi^{p,q+1}_\partial\right)^{-1}([v_1:v_2])=[\eklm{v_1,v_2}_{p,q+1}v_1-v_2:\eklm{v_1,v_2}_{p,q+1}v_1+v_2]. \]

\subsubsection{Projective Anosov subgroups of \texorpdfstring{$\SO(p,q+1)$}{SO(p,q+1)}}

There is a straightforward relation between projective Anosov subgroups $\Gamma<\SO(p,q+1)<\SL(V)$ and the boundary $\partial\bbH^{p,q}$.

\begin{lem}[{\cite[Rem.~2.4]{dgk18a}}] \label{lem limit set of projective Anosov subgroup of SO(p,q+1)}
If $\Gamma<\SO(p,q+1)$ is projective Anosov, then $\Lambda_\Gamma\subset \partial\bbH^{p,q}$. Furthermore, the limit sets $\Lambda_\Gamma$, $\Lambda^*_\Gamma$ can be  identified through the intertwining relation $\xi^*=\Phi_{p,q+1}\circ \xi$.
\end{lem}

\begin{definition}
If $\Gamma<\SO(p,q+1)$ is projective Anosov, we consider
\begin{align*} \tilde\cM_{\Gamma,\partial}&=\tilde\cM_\Gamma\cap \bbL_\partial^{p,q+1}\\
&= \set{[v_1:v_2]\in \bbL^{p,q+1}_\partial}{\eklm{v_1,v}_{p,q+1}\neq 0 \textrm{ or } \eklm{v_2,v}_{p,q+1}\neq 0~\forall\, [v]\in \Lambda_\Gamma}.
\end{align*}
\end{definition}
As a consequence of Lemma \ref{lem limit set of projective Anosov subgroup of SO(p,q+1)}, we find that any projective Anosov subgroup $\Gamma<\SO(p,q+1)$ satisfies
\begin{equation*}
\tilde\cK_\Gamma\subset\tilde\cM_{\Gamma,\partial}.
\end{equation*}

\begin{rem} In the Riemannian case $q=0$, any distinct points $[v_1],[v_2]\in \partial\bbH^{p}$ satisfy $\eklm{v_1,v_2}_{p,1}\neq 0$, hence $\tilde\cM_{\Gamma,\partial}=\bbL^{p,1}_\partial\approx T^1\bbH^p$ is the whole Riemannian sphere bundle.
\end{rem}

\subsubsection{Projective Anosov subgroups and \texorpdfstring{$\mathbb H^{p,q}$-}{pseudo-hyperbolic }convex cocompactness}

In \cite{dgk18a}, Danciger, Gué\-ri\-taud and Kassel relate the projective Anosov property to the geometry of $\bbH^{p,q}$, generalizing the work of Barbot and Mérigot \cite{barbot-merigot} in the anti-de Sitter case.

\begin{definition}[{\cite[Def.~1.23]{dgk18a}}]A discrete subgroup $\Gamma<\SO(p,q+1)$ is called \emph{$\mathbb H^{p,q}$-convex cocompact} if it acts properly discontinuously and cocompactly on some properly convex closed subset $C\subset \mathbb H^{p,q}$ with non-empty interior whose ideal boundary $\partial_i C:=\overline C\setminus C$ does not contain any non-trivial projective line segment. Here $\overline{C}$ denotes the closure of $C$ in $\mathbb{P}(V)$.
\end{definition}

\begin{prop}[{\cite[Thm.~1.24]{dgk18a}}]
If $\Gamma< \SO(p,q+1)$ is $\bbH^{p,q}$-convex cocompact, then it is projective Anosov.
\end{prop}

\begin{definition} A \emph{hyperbolic plane} $P\subset \bbH^{p,q}$ is a totally geodesic copy of $\bbH^2$, that is $P=\bbH^{p,q}\cap \bbP(W)$ where $W\subset V$ is a three-dimensional vector subspace of signature $(2,1)$ for $\eklm{\cdot,\cdot}_{p,q+1}$.

A triple $(\ell_1,\ell_2,\ell_3)\in (\bbH^{p,q}\cup \partial\bbH^{p,q})^3$ is called \emph{negative} if there is a hyperbolic plane $P\subset \bbH^{p,q}$ such that $\ell_1,\ell_2,\ell_3\in \overline{P}$, where $\overline P$ denotes the closure of $P$ in $\mathbb{P}(V)$.
\end{definition}

\begin{prop}[{\cite[Prop.~8.1]{dgk18a}}]
If $\Gamma< \SO(p,q+1)$ is $\bbH^{p,q}$-convex cocompact, then it acts properly discontinuously on a non-empty properly convex open set
\[ \Omega_\Gamma^{\mathrm{DGK}}\subset\set{\ell\in\bbH^{p,q}}{(\ell,\ell_1,\ell_2) \textrm{ is negative } \forall \ell_1,\ell_2\in \Lambda_\Gamma}. \]
If the right-hand side is properly convex, the inclusion above can be made an equality. In particular, this is the case when $\Gamma$ is irreducible.
\end{prop}
Naturally, one can also consider the unit spacelike tangent bundle $T^1\Omega_\Gamma^\mathrm{DGK}\subset T^1\bbH^{p,q}$ on which $\Gamma$ acts properly discontinuously. Our approach produces a larger discontinuity domain for the $\Gamma$-action on  $T^1\bbH^{p,q}$:
\begin{lem}\label{sl con}
If $\Gamma< \SO(p,q+1)$ is $\bbH^{p,q}$-convex cocompact, then
\[ \tilde\cK_\Gamma\subset \Phi^{p,q+1}_\partial\left(T^1 \Omega_\Gamma^{\mathrm{DGK}}\right)\subset \tilde\cM_{\Gamma,\partial}. \]
\end{lem}
\begin{proof}
The set $(\Phi^{p,q+1}_\partial)^{-1}(\tilde\cK_\Gamma)$ is included in the spacelike unit tangent bundle over the convex hull of $\Lambda_\Gamma$ in $\bbH^{p,q}$, which is itself included in $T^1\Omega_\Gamma^{\mathrm{DGK}}$ (because $\Omega_\Gamma^{\mathrm{DGK}}$ is convex).

Now let $[x:v]\in \Omega_\Gamma^{\mathrm{DGK}}$ and $[w]\in \Lambda_\Gamma$. Since $[x]\in \Omega_\Gamma^{\mathrm{DGK}}$, we have $\eklm{x,w}_{p,q+1}\neq 0$, so either $\eklm{x+v,w}_{p,q+1}\neq 0$ or $\eklm{x-v,w}_{p,q+1}\neq 0$, i.e. $[x+v:x-v]=\Phi^{p,q+1}_\partial([x:v])\in \tilde\cM_{\Gamma,\partial}$.
\end{proof}

Note that if $q\neq 0$ and $\Gamma$ is infinite, we have $\Phi^{p,q+1}_\partial\left(T^1 \Omega_\Gamma^{\mathrm{DGK}}\right)\neq \tilde\cM_{\Gamma,\partial}$ since $T^1 \Omega_\Gamma^{\mathrm{DGK}}$ is not invariant under the geodesic flow. This means that the discontinuity domain of $\Gamma$ $$\left(\Phi^{p,q+1}_\partial\right)^{-1}\big(\tilde\cM_{\Gamma,\partial}\big)\subset T^1\bbH^{p,q}$$ is larger than $T^1 \Omega_\Gamma^{\mathrm{DGK}}$.

We can now give the full statement and proof of Theorem \ref{introthm:pseudoRiem} from the introduction.
\begin{thm}\label{THM CCHPQ}
Suppose $\Gamma<\mathrm{SO}(p, q+1)$ is a non-trivial torsion-free $\bbH^{p,q}$-convex cocompact subgroup.  The (possibly incomplete) space-like geodesic flow
$$
\phi^{t}: T^{1}\left(\Gamma\backslash \Omega_{\Gamma}^{\mathrm{DGK}}\right)\rightarrow T^{1}\left(\Gamma\backslash \Omega_{\Gamma}^{\mathrm{DGK}}\right)
$$
is the restriction of a complete Axiom A flow
$$
\phi^{t}: \mathcal{M}_{\Gamma}\rightarrow \mathcal{M}_{\Gamma}
$$
with a unique basic hyperbolic set $\mathcal{K}_{\Gamma}$ such that
$\mathcal{K}_{\Gamma}\subset T^{1}\left(\Gamma\backslash \Omega_{\Gamma}^{\mathrm{DGK}}\right)\subset \mathcal{M}_{\Gamma}.$
Moreover,
\begin{enumerate}
\item If $\Gamma$ is irreducible, the restriction of the space-like geodesic flow $\phi^{t}$ to the basic hyperbolic set $\mathcal{K}_{\Gamma}\subset T^{1}(\Gamma\backslash \Omega_{\Gamma}^{\mathrm{DGK}})$  mixes exponentially for all H\"{o}lder observables with respect to every Gibbs equilibrium state with H\"{o}lder potential.

\item The Ruelle zeta function $\zeta_{\Gamma}$ constructed using the periods of the space-like geodesic flow admits global meromorphic continuation to $\mathbb{C}$ with a simple pole at $h_{\mathrm{top}}(\phi^{t})$, and assuming $\Gamma$ is irreducible, $\zeta_{\Gamma}$ is nowhere vanishing and analytic in a strip $h_{\mathrm{top}}(\phi^{t})-\varepsilon< \mathrm{Re}(z)< h_{\mathrm{top}}(\phi^{t})$ for some $\varepsilon>0.$

\item There exist Ruelle-Pollicott resonances of the generator $X$ of the flow $\phi^t$ with associated (co-)resonant states and invariant Ruelle distributions. In particular, $h_{\mathrm{top}}(\phi^{t})$ is a  resonance and the associated invariant Ruelle distribution is the Bowen-Mar\-gu\-lis measure of maximal entropy.

\item If $\Gamma$ is irreducible, the spacelike geodesic flow satisfies the prime orbit theorem with exponentially decaying error term:
$$
N_{\Gamma}(\tau)=\mathrm{Li}\big(e^{h_{\mathrm{top}}(\phi^{t})\tau} \big)\Big(1+ O\big(e^{-(c-h_{\mathrm{top}}(\phi^{t}))\tau}\big)\Big)
$$
for some $0<c<h_{\mathrm{top}}(\phi^{t}).$
\end{enumerate}
 \end{thm}
\begin{proof}[Proof of Theorem \ref{THM CCHPQ}]
That the space-like geodesic flow on $T^{1}(\Gamma\backslash \Omega_\Gamma^{\mathrm{DGK}})$ is the restriction of an Axiom A flow on $\mathcal{M}_{\Gamma}$ is a consequence of Lemma \ref{sl con}.  The statements (1)-(4) are direct applications of Theorem \ref{thm: cor exp decay} on exponential mixing,  Theorems \ref{thm:thmzetasmooth}, \ref{RH} on zeta functions, Theorems \ref{Resonant states precise}, \ref{Resonant states precise2} on resonances, and Theorem \ref{thm:orbitcounting} on orbit counting.
\end{proof}

\begin{rem}[$\bbH^{p,q}$ and $\bbS^{p,q}$] 
One can also consider the pseudo-spherical space
\[ \bbS^{p,q}=\set{[x]\in\bbP(V)}{\eklm{x,x}_{p,q+1}>0}.\]
It also comes with a pseudo-Riemannian metric and a spacelike unit tangent bundle
\[ T^1\bbS^{p,q}=\set{[x:v]\in \bbP(V\times V)}{\eklm{x,x}_{p,q+1}>0, \eklm{x,v}_{p,q+1}=\eklm{v,v}_{p,q+1}+\eklm{x,x}_{p,q+1}=0}. \]
The map $[x:v]\to [v:x]$ is an $\SO(p,q+1)$-equivariant  diffeomorphism between  $T^1\bbH^{p,q}$ and $T^1\bbS^{p,q}$ that intertwines the geodesic flows, so the study of the spacelike geodesic flow does not differentiate between $\bbH^{p,q}$ and $\bbS^{p,q}$.
\end{rem}

\subsection{Benoist subgroups and the Benoist-Hilbert flow}

\subsubsection{Strictly convex divisible domains}

\begin{definition} \label{BenSub}
A discrete subgroup $\Gamma<\SL(V)$ \emph{divides} a properly convex open subset $\cC\subset \bbP(V)$ if the action of $\Gamma$ on $\bbP(V)$ preserves $\cC$ and $\Gamma\curvearrowright\cC$ is properly discontinuous and cocompact.

A \emph{strictly convex domain} $\cC\subset\bbP(V)$ is a properly convex open subset whose boundary $\partial\cC$ does not contain any non-trivial projective segment.

A discrete subgroup $\Gamma<\SL(V)$ is called a \emph{Benoist subgroup} if it divides a non-empty strictly convex domain.
\end{definition}

\begin{prop}[{\cite[Théorème 1.1]{BEN04}}]
Let $\Gamma<\SL(V)$ be a Benoist sugroup. Then $\Gamma$ is projective Anosov, it divides a unique strictly convex domain $\cC_\Gamma\subset\bbP(V)$ and one has
\[ \Lambda_\Gamma=\partial\cC_\Gamma. \]
Furthermore, there exists $0<\alpha<1$ such that the boundary $\partial\cC_\Gamma\subset \bbP(V)$ is a $C^{1, \alpha}$ submanifold, and  the limit maps $\xi,\xi^*$ are related by
\[ T_{\xi(t)}\partial\cC_\Gamma=d_x\pi(\xi^*(t)),\quad \forall t\in\partial_\infty\Gamma,\forall x\in \xi(t). \]
\end{prop}
Benoist also proved in \cite{BEN04} that the boundary $\partial\cC_\Gamma$ is never $C^2$, unless $\Gamma$ is conjugate to a uniform lattice in $\SO(d-1,1)$ (in which case $\cC_\Gamma$ is an ellipsoid).

\subsubsection{The Benoist-Hilbert geodesic flow} The cross-ratio of four points $\ell_1,\ell_2,\ell_3,\ell_4\in \bbP(V)$ lying on a common projective line $L\subset \bbP(V)$ is defined as
\[ [\ell_1,\ell_2,\ell_3,\ell_4] = \frac{t_1-t_3}{t_1-t_2}\cdot\frac{t_2-t_4}{t_3-t_4}, \]
where $t_1,t_2,t_3,t_4\in \R\cup\{\infty\}$ are preimages of $\ell_1,\ell_2,\ell_3,\ell_4$ under any projective parametrization $\R\bbP^1\to L$ (i.e. of the form $t\mapsto [v+tw]$ for some linearly independent vectors $v,w\in V$).

\begin{definition} Let $\cC\subset \bbP(V)$ be a properly convex open subset. The \emph{Hilbert distance} between distinct points $x,y\in \cC$ is defined as
\[ d_\cC(x,y)=\frac{1}{2}\log [x^-,x,y,y^+], \]
where $x^-$ and $y^+$ are the intersections of $\partial\cC$ with the projective line through $x$ and $y$, with $x^-,x,y,y^+$ in this order.
\end{definition}
The Hilbert distance makes $(\cC, d_{\cC})$ a geodesic metric space. If $\cC$ is strictly convex, then geodesics are intersections of projective lines with $\cC$. In order to define a geodesic flow, we will work with \emph{sphere bundles} of manifolds, i.e.  the fiber bundle $\bbS M$ over a manifold $M$ with fiber $\bbS_xM= \left(T_xM\setminus\{0\}\right) / \R^*_+$ over $x\in \M$, where $\R^*_+$ acts by multiplication. The ray spanned by a non-zero tangent vector $\nu\in T_xM$ will be denoted by $[\nu)\in \bbS_xM$.

If $\cC\subset\bbP(V)$ is a properly convex open subset, a pair $(\ell,[\nu))\in \bbS\cC$ defines a parametrization $c_{\ell,[\nu)}:\R\to \cC$ of the intersection of $\cC$ with the projective line going through $\ell$ and tangent to $\nu$, satisfying $d_\cC(\ell,c_{\ell,[\nu)}(t))=\vert t\vert$ for any $t\in\R$, and $\dot c_{\ell,[\nu)}(0)\in\R^*_+\cdot \nu$.

\begin{definition}
The \emph{Benoist-Hilbert flow} is the flow $\phi_{\mathrm{BH}}^t:\bbS\cC\to\bbS\cC$ defined by 
\[ \phi_{\mathrm{BH}}^t(\ell,[\nu)):=(c_{\ell,[\nu)}(t),[\dot c_{\ell,[\nu)}(t))),\quad \forall (\ell,[\nu])\in \bbS\cC.\]
\end{definition}
Let us now focus on the case of a torsion-free Benoist subgroup $\Gamma<\SL(V)$, and denote by $\cN_\Gamma=\Gamma\backslash\cC_\Gamma$ the quotient manifold. As the action $\Gamma\curvearrowright\bbS\cC_\Gamma$ commutes with the Benoist-Hilbert flow, we can also define a flow $\phi_{\mathrm{BH}}^t$ on the quotient manifold $\bbS\cN_\Gamma=\Gamma\backslash\bbS\cC_\Gamma$.

Note that $\bbS\cN_\Gamma$ is a smooth manifold, but the regularity of $\phi_{\mathrm{BH}}^t$ is exactly that of $\partial\cC_\Gamma$.

\begin{prop}[{\cite[Théorème 1.1, Théorème 1.2]{BEN04}}]
Let $\Gamma<\SL(V)$ be a torsion-free Benoist subgroup. There exists $0<\alpha<1$ such that the Benoist-Hilbert flow $\phi_{\mathrm{BH}}^t:\bbS\cN_\Gamma\to\bbS\cN_\Gamma$ is a topologically transitive $C^{1+\alpha}$ Anosov flow.
\end{prop}

\subsubsection{Integration into our setting}

The Benoist-Hilbert flow $\phi_{\mathrm{BH}}^t:\bbS\cN_\Gamma\to\bbS\cN_\Gamma$ of a torsion-free Benoist subgroup $\Gamma<\SL(V)$ cannot be directly related to our flow space $(\bbL,\phi^t)$. Indeed, each non-trivial element $\gamma\in\Gamma$  corresponds to a periodic orbit of both flows, but they have different periods: $\lambda_1(\gamma)$ for the flow $\phi^t$ on $\cM_\Gamma$ and $\frac{1}{2}(\lambda_1(\gamma)-\lambda_d(\gamma))$ for the Benoist-Hilbert flow.

In order to make the two coincide, we will work with the adjoint representation $\Ad:\SL(V)\to \SL(W)$, where $W=\mathfrak{sl}(V)\subset \mathrm{End}(V)=V\otimes V^\ast$, as it satisfies $\lambda_1(\Ad(g))=\lambda_1(g)-\lambda_d(g)$ for any $g\in \SL(V)$ (meaning that $\lambda_1-\lambda_d$ is the highest weight of the adjoint representation).

\begin{lem}[{\cite[Prop.~4.3]{GW12}}] \label{lem adjoint of projective Anosov subgroup}
If $\Gamma<\SL(V)$ is a projective Anosov subgroup, then so is $\Ad(\Gamma)<\SL(W)$, and its limit map  is given by $\xi_\Ad(t)=[v\otimes\alpha]\in \mathbb P(W)$ at $t\in\partial_\infty\Gamma$, where $[v]=\xi(t)$ and $[\alpha]=\xi^*(t)$.
\end{lem}

Note that since the adjoint representation of $\SL(V)$ preserves the non-degenerate symmetric bilinear form $(\varphi,\psi)\mapsto \Tr(\varphi\circ\psi)$ (which is a multiple of the Killing form of $\mathfrak{sl}(V)$), we can identify $W$ with $W^*$ through the map 
\[ \Phi_\Ad:\map{W}{W^*}{\varphi}{(\psi\mapsto\Tr(\varphi\circ\psi))}. \]
As in the study of $\bbH^{p,q}$-convex cocompact subgroups, we will replace the flow space
\[\bbL=\set{[v:\alpha]\in\bbP(W\times W^*)}{\alpha(v)>0} \]
and the flow
\[ \phi^t([v:\alpha])=[e^tv:e^{-t}\alpha] \]
with the flow space
\[
\bbL_\Ad=\set{[\varphi:\psi]\in \mathbb{P}(W\times W)}{\Tr(\varphi\circ\psi)>0 }
\]
equipped with the flow
 \[ \phi^t([\varphi:\psi])=([e^t\varphi:e^{-t}\psi]),\qquad t\in \R. \]
There is a difference in regularities between the Benoist-Hilbert flow $\phi^t_{\mathrm{BH}}$ on $\bbS\cN_\Gamma$ and the flow $\phi^t$ on $\cM_{\Ad(\Gamma)}$, as the basic set $\cK_{\Ad(\Gamma)}$ is not a $C^1$ submanifold of $\cM_{\Ad(\Gamma)}$. It is however better than the expected H\"older regularity.

\begin{lem}\label{lem basic set Ad(Benoist subgroup) is Lipschitz}
Let $\Gamma<\SL(V)$ be a Benoist subgroup. The basic set $\cK_{\Ad(\Gamma)}$ is a Lipschitz submanifold of $\cM_{\Ad(\Gamma)}$.
\end{lem}
\begin{proof}
By \cite[Thm.~2.55]{GMT}, the set
\[ \Lambda_{\Gamma}^\mathrm{Sym}=\set{ (\xi(t),\xi^*(t))\in \bbP(V)\times \bbP(V^*)}{ t\in \partial_\infty\Gamma}\]
is a Lipschitz submanifold of $\bbP(V)\times \bbP(V^*)$. It follows by Lemma \ref{lem adjoint of projective Anosov subgroup} that the limit set
\[\Lambda_{\Ad(\Gamma)}=\set{[v\otimes\alpha]\in\bbP(W)}{([v],[\alpha])\in \Lambda_{\Gamma}^\mathrm{Sym}} \]
is a Lipschitz submanifold of $\bbP(W)$, therefore 
 \[\tilde\cK_{\Ad(\Gamma)}=\set{[\varphi:\psi]\in \bbP(W\times W)}{[\varphi],[\psi]\in \Lambda_{\Ad(\Gamma)},\, \Tr(\varphi\circ\psi)=1}\]
 is a Lipschitz submanifold of $\bbL_\Ad$.
\end{proof}
The following result allows us to apply our work to the study of the Benoist-Hilbert flow.
\begin{lem} \label{lem embedding of Benoist-Hilbert flow} Let $\Gamma<\SL(V)$ be a torsion-free Benoist subgroup. 
There is a H\"older homeomorphism $\Psi:\bbS\cN_\Gamma\to \cK_{\Ad(\Gamma)}$ with Lipschitz inverse such that $\Psi\circ \phi_{\mathrm{BH}}^{t}=\phi^{2t}\circ\Psi$ for all $t\in\R$.
\end{lem}
\begin{proof}
In this construction we will consider a connected component $\widehat\cC_\Gamma\subset V$ of the cone $\{v\in V\setminus\{0\}\,|\,[v]\in\cC_\Gamma\}$ and the corresponding component $\widehat \cC_\Gamma^*=\{\alpha\in V\,|\,\alpha\vert_{\widehat\cC_\Gamma}>0\}$ of the dual convex set $\cC_\Gamma^*=\set{[\alpha]\in\bbP(V^*)}{\bbP(\ker\alpha)\cap\cC_\Gamma=\emptyset}$. 

Let $([x],[\nu))\in \bbS\cC_\Gamma$, and consider $t^\pm\in\partial_\infty\Gamma$ such that $\lim_{t\to\pm\infty}c_{[x],[\nu)}(t)=\xi(t^\pm)\in\partial\cC_\Gamma$. Consider vectors $v^\pm\in V$ such that $[v^\pm]=\xi(t^\pm)$ and linear forms $\alpha^\pm\in V^*$ such that $[\alpha^\pm]=\xi^*(t^\pm)$. Let
\[ \Psi([x],[\nu)):=\left[ \frac{\alpha^-(x)}{\alpha^+(x)}\frac{v^+\otimes\alpha^+}{\alpha^-(v^+)} : \frac{\alpha^+(x)}{\alpha^-(x)}\frac{v^-\otimes\alpha^-}{\alpha^+(v^-)} \right]\in \bbL_\Ad. \]
This expression defines a map $\Psi:\bbS\cC_\Gamma\to\bbL_\Ad$. The description of the limit set $\Lambda_{\Ad(\Gamma)}$ in Lemma \ref{lem adjoint of projective Anosov subgroup} shows that $\Psi(\bbS\cC_\Gamma)\subset \tilde\cK_{\Ad(\Gamma)}$.

Let $[\varphi:\psi]\in\tilde\cK_{\Ad(\Gamma)}$. Consider a vector $v_0\notin \ker\varphi$, then write 
\[ \Phi([\varphi:\psi])=\left([x],[\nu)\right)\in\bbS(\bbP(V)) \]
where $x=\varphi(v_0)+\psi\circ\varphi(v_0)$ and $\nu(x)=d_x\pi(\varphi(v_0))$.  Since the expression does not depend on the choice of $v_0$, it defines a map $\Phi:\tilde\cK_{\Ad(\Gamma)}\to\bbS(\bbP(V))$ such that $\Phi\circ\Psi=\mathrm{id}$.

By Lemma \ref{lem adjoint of projective Anosov subgroup} and Lemma \ref{lem limit set of projective Anosov subgroup of SO(p,q+1)} we can write $\varphi=v^+\otimes\alpha^+$ and $\psi=v^-\otimes\alpha^-$  where $[v^\pm]=\xi(t^\pm)$ and $[\alpha^\pm]=\xi^*(t^\pm)$ for a pair of distinct points $t^+,t^-\in \partial_\infty\Gamma$. We can also choose these lifts so that $v^\pm\in\partial\widehat\cC_\Gamma$ and $\alpha^\pm\in\partial\widehat\cC_\Gamma^*$. Then $x=\alpha^+(v^-)v^++v^-\in \widehat\cC_\Gamma$, which shows that $[x]\in\cC_\Gamma$, i.e. we have constructed a function $\Phi:\tilde\cK_{\Ad(\Gamma)}\to\bbS\cC_\Gamma$. The fact that $[\nu)=[d_x\pi(v^+))=[-d_x\pi(v^-))$ shows that $\lim_{t\to\pm\infty}c_{[x],[\nu)}(t)=[v^\pm]$, therefore $[\varphi:\psi]=\Psi([x],[\nu))$, i.e. $\Psi\circ\Phi=\mathrm{id}$.

For $t\in \R$, consider $x_t\in V$ and $\nu_t\in T_{[x_t]}\bbP(V)$ such that $\phi_{\mathrm{BH}}^t([x],[v))=([x_t],[\nu_t))$. There are numbers $x^\pm_t\in\R$ such that $x_t=x^+_tv^++x^-_tv^-$. From the cross-ratio computation
\begin{align*} [[v^-],[x],[x_t],[v^+]]&=\left[ 0,\frac{x^+_0}{x^-_0},\frac{x^+_t}{x^-_t},\infty\right] \\
& =\frac{x^+_t}{x^-_t}\frac{x^-_0}{x^+_0} 
\end{align*}
we find that
\begin{align*} \frac{\alpha^-(x_t)}{\alpha^+(x_t)}&=\frac{x^+_t}{x^-_t}\frac{\alpha^-(v^+)}{\alpha^+(v^-)}\\
&=[[v^-],[x],[x_t],[v^+]] \frac{x^+_0}{x^-_0} \frac{\alpha^-(v^+)}{\alpha^+(v^-)}\\
&= e^{2t}\frac{x^+_0}{x^-_0}  \frac{\alpha^-(v^+)}{\alpha^+(v^-)}\\
&= e^{2t}\frac{\alpha^-(x)}{\alpha^+(x)}.
\end{align*}
This shows that $\Psi\circ \phi_{\mathrm{BH}}^t=\phi^{2t}\circ\Psi$.  As the bijection $\Psi:\bbS\cC_\Gamma\to\tilde\cK_{\Ad(\Gamma)}$ is $\Gamma$-equivariant, it induces a bijection $\Psi:\bbS\cN_\Gamma\to\cK_{\Ad(\Gamma)}$ also satisfying $\Psi\circ \phi_{\mathrm{BH}}^t=\phi^{2t}\circ\Psi$. The construction of $\Psi$ involves the H\"older maps $\xi$ and $\xi^*$, so it is H\"older, but $\Phi$ is constructed as the restriction to a Lipschitz submanifold of a smooth map, which makes it Lipschitz.
\end{proof}
The loss of regularity from $C^{1,\alpha}$ to Lipschitz when going from $\phi^t_\mathrm{BH}$ to $\phi^t\vert_{\cK_{\Ad(\Gamma)}}$ is compensated by the gain of regularity in the stable/unstable foliations.  We are now ready for the full statement and proof of Theorem \ref{introthm:BenoistHilbert} from the introduction.
\begin{thm}\label{THM BHG}
If $\Gamma<\mathrm{SL}(V)$ is a non-trivial torsion-free Benoist subgroup and $W=\mathfrak{sl}(V),$ then $\mathrm{Ad}(\Gamma)<\mathrm{SL}(W)$ is projective Anosov with associated Axiom A system $(\mathcal{M}_{\mathrm{Ad}(\Gamma)},\phi^t)$ and basic hyperbolic set $\mathcal{K}_{\mathrm{Ad}(\Gamma)}$ which is a Lipschitz submanifold of $\mathcal{M}_{\mathrm{Ad}(\Gamma)}.$
Additionally, there is a H\"{o}lder continuous homeomorphism
$$
\Psi: \mathbb{S}\mathcal{N}_{\Gamma}\rightarrow \mathcal{K}_{\mathrm{Ad}(\Gamma)}
$$
with Lipschitz inverse such that $\Psi\circ \phi_{BH}^{t}=\phi^{2t}\circ \Psi.$
Moreover,
\begin{enumerate}
\item $\phi_{BH}^{t}$ mixes exponentially for all H\"{o}lder observables with respect to every Gibbs equilibrium state with H\"{o}lder potential.

\item The Ruelle zeta function $\zeta_{BH, \Gamma}$ constructed using the periods of $\phi_{BH}^{t}$ admits global meromorphic continuation to $\mathbb{C}$ with a simple pole at $h_{\mathrm{top}}(\phi_{BH}^{t})$, and $\zeta_{BH, \Gamma}$ is nowhere vanishing and analytic in a strip $h_{\mathrm{top}}(\phi_{BH}^{t})-\varepsilon< \mathrm{Re}(z)< h_{\mathrm{top}}(\phi_{BH}^{t})$ for some $\varepsilon>0.$

\item \label{BHG RES} There exist Ruelle-Pollicott resonances of the generator $X$ of the flow $\phi^t_{BH}$ with associated (co-)resonant states and invariant Ruelle distributions. In particular, $h_{\mathrm{top}}(\phi^{t}_{BH})$ is a  resonance and the associated invariant Ruelle distribution is the Bowen-Margulis measure of maximal entropy.

\item $\phi_{BH}^{t}$ satisfies the prime orbit theorem with exponentially decaying error term:
$$
N_{BH, \Gamma}(\tau)=\mathrm{Li}\big(e^{h_{\mathrm{top}}(\phi_{BH}^{t})\tau} \big)\Big(1+ O\big(e^{-(c-h_{\mathrm{top}}(\phi_{BH}^{t}) )\tau}\big)\Big)
$$
for some $0<c<h_{\mathrm{top}}(\phi_{BH}^{t}).$
\end{enumerate}
\end{thm}
\begin{proof}[Proof of Theorem \ref{THM BHG}]
That the Benoist-Hilbert geodesic flow admits a H\"{o}lder conjugacy and constant speed reparameterization onto a Lipschitz basic hyperbolic set $\mathcal{K}_{\Ad(\Gamma)}\subset \mathcal{M}_{\Ad(\Gamma)}$ is the content of Lemma \ref{lem embedding of Benoist-Hilbert flow}.  By a result of Benoist \cite[Théorème 1.2]{BenoistCD2}, either $\Gamma<\SL(V)$ is conjugate to a uniform lattice in $\SO(d-1,1)$, in which case the flow $\phi_{\mathrm{BH}}^t$ is the geodesic flow of the hyperbolic manifold $\Gamma\backslash\bbH^{d-1},$ or $\Gamma$ is Zariski-dense in $\SL(V)$; in particular $\Ad(\Gamma)$ is irreducible.  Hence, the exponential mixing result follows from Theorem \ref{thm: cor exp decay} as exponential mixing for H\"{o}lder observables with respect to Gibbs states is invariant under H\"older conjugacy and constant speed reparametrizations.  Since the periods are preserved by the conjugacy, the definition of the zeta function is unchanged and therefore the results concerning its analytic continuation, spectral gap and zero-free strip are a consequence of Theorems \ref{thm:thmzetasmooth} and \ref{RH}.  The results on (co-)resonant states are an application of Theorems \ref{Resonant states precise} and \ref{Resonant states precise2}.  Finally, the prime orbit theorem with exponentially decaying error term, again only depending on the periods, follows by appealing to Theorem \ref{thm:orbitcounting} which completes the proof.
\end{proof}

\providecommand{\bysame}{\leavevmode\hbox to3em{\hrulefill}\thinspace}
\providecommand{\MR}{\relax\ifhmode\unskip\space\fi MR }
\providecommand{\MRhref}[2]{%
  \href{http://www.ams.org/mathscinet-getitem?mr=#1}{#2}
}
\providecommand{\href}[2]{#2}

\bigskip

\begin{thebibliography}{{Mes}07}

\bibitem[Bal98]{BAL98}
V.~Baladi, \emph{Periodic orbits and dynamical spectra}, Ergodic Theory Dynam.
  Systems \textbf{18} (1998), no.~2, 255--292.

\bibitem[Bar01]{BarbotFlag01}
T.~Barbot, \emph{Flag {Structures} on {Seifert} {Manifolds}}, Geom. Topol.
  \textbf{5} (2001), 227--266.

\bibitem[Bar10]{BarbotFlag10}
\bysame, \emph{Three-dimensional {Anosov} flag manifolds}, Geom. Topol.
  \textbf{14} (2010), no.~1, 153--191.

\bibitem[BCLS15]{BCLS}
M.~Bridgeman, R.~Canary, F.~Labourie, and A.~Sambarino, \emph{The pressure
  metric for {Anosov} representations}, Geom. Funct. Anal. \textbf{25} (2015),
  no.~4, 1089--1179.

\bibitem[Ben97]{BEN97}
Y.~Benoist, \emph{Propri\'{e}t\'{e}s asymptotiques des groupes lin\'{e}aires},
  Geom. Funct. Anal. \textbf{7} (1997), no.~1, 1--47 (French).

\bibitem[Ben00]{BEN00}
\bysame, \emph{Propri\'{e}t\'{e}s asymptotiques des groupes lin\'{e}aires.
  {II}}, Analysis on homogeneous spaces and representation theory of {L}ie
  groups, {O}kayama--{K}yoto (1997), Adv. Stud. Pure Math., vol.~26, Math. Soc.
  Japan, Tokyo, 2000, pp.~33--48 (French).

\bibitem[Ben03]{BenoistCD2}
\bysame, \emph{Convexes divisibles. {II}}, Duke Math. J. \textbf{120} (2003),
  no.~1, 97--120 (French).

\bibitem[Ben04]{BEN04}
\bysame, \emph{Convexes divisibles. {I}}, Algebraic groups and arithmetic, Tata
  Inst. Fund. Res., Mumbai, 2004, pp.~339--374 (French).

\bibitem[Ben08]{BEN08}
\bysame, \emph{A survey on divisible convex sets}, Geometry, analysis and
  topology of discrete groups, Adv. Lect. Math. (ALM), vol.~6, Int. Press,
  Somerville, MA, 2008, pp.~1--18.

\bibitem[BM12]{barbot-merigot}
T.~Barbot and Q.~M{\'e}rigot, \emph{Anosov {A}d{S} representations are
  quasi-{F}uchsian}, Groups Geom. Dyn. \textbf{6} (2012), no.~3, 441--483.

\bibitem[Bow75]{BOW75}
R.~Bowen, \emph{Symbolic dynamics for hyperbolic flows}, Proceedings of the
  {I}nternational {C}ongress of {M}athematicians ({V}ancouver, {B}.{C}., 1974),
  {V}ol. 2, Canad. Math. Congr., Montreal, QC, 1975, pp.~299--302.

\bibitem[Bow98]{BOW98}
B.~H. Bowditch, \emph{Cut points and canonical splittings of hyperbolic
  groups}, Acta Math. \textbf{180} (1998), no.~2, 145--186.

\bibitem[Bow08]{Bowen}
R.~Bowen, \emph{Equilibrium states and the ergodic theory of {Anosov}
  diffeomorphisms}, 2nd revised ed. ed., Lect. Notes Math., vol. 470, Berlin:
  Springer, 2008.

\bibitem[BR75]{BR75}
R.~Bowen and D.~Ruelle, \emph{The ergodic theory of {A}xiom {A} flows}, Invent.
  Math. \textbf{29} (1975), no.~3, 181--202.

\bibitem[BWS21]{BWS21}
Y.~Borns-Weil and S.~Shen, \emph{Dynamical zeta functions in the nonorientable
  case}, Nonlinearity \textbf{34} (2021), no.~10, 7322--7334.

\bibitem[CE71]{CE71}
C.~Conley and R.~Easton, \emph{Isolated invariant sets and isolating blocks},
  Trans. Amer. Math. Soc. \textbf{158} (1971), 35--61.

\bibitem[Che98]{CHE98}
N.~I. Chernov, \emph{Markov approximations and decay of correlations for
  {A}nosov flows}, Ann. Math. (2) \textbf{147} (1998), no.~2, 269--324.

\bibitem[CS23]{CS23}
L.~Carvajales and F.~Stecker, \emph{Anosov representations acting on
  homogeneous spaces: domains of discontinuity}, 2023, arXiv:2308.08607.

\bibitem[Dav88]{DAV88}
G.~David, \emph{Op\'{e}rateurs d'int\'{e}grale singuli\`ere sur les surfaces
  r\'{e}guli\`eres}, Ann. Sci. \'{E}cole Norm. Sup. (4) \textbf{21} (1988),
  no.~2, 225--258.

\bibitem[DG16]{DG16}
S.~Dyatlov and C.~Guillarmou, \emph{{Pollicott--Ruelle resonances for open
  systems}}, Ann. Henri Poincar{\'e} \textbf{17} (2016), no.~11, 3089--3146.

\bibitem[DG18]{DG18}
\bysame, \emph{{Afterword: Dynamical zeta functions for Axiom A flows}}, Bull.
  Amer. Math. Soc. (2018), no.~55, 337--342.

\bibitem[DGK18]{dgk18}
J.~Danciger, F.~Gu\'{e}ritaud, and F.~Kassel, \emph{{Convex cocompactness in
  pseudo-Riemannian hyperbolic spaces.}}, {Geom. Dedicata} \textbf{192} (2018),
  87--126.

\bibitem[DGK23]{dgk18a}
\bysame, \emph{Convex cocompact actions in real projective geometry}, 2023,
  arXiv:1704.08711, to appear in Ann. Sci. Éc. Norm. Supér.

\bibitem[DMS24]{DMSII}
B.~Delarue, D.~Monclair, and A.~Sanders, \emph{Locally homogeneous {A}xiom {A}
  flows {II}}, in preparation, 2024.

\bibitem[Dol98]{DOL98}
D.~Dolgopyat, \emph{On decay of correlations in {A}nosov flows}, Ann. Math. (2)
  \textbf{147} (1998), no.~2, 357--390.

\bibitem[DP98]{DP98}
D.~Dolgopyat and M.~Pollicott, \emph{Addendum to: ``{P}eriodic orbits and
  dynamical spectra'' [{E}rgodic {T}heory {D}ynam. {S}ystems {\bf 18} (1998),
  no. 2, 255--292] by {V}. {B}aladi}, Ergodic Theory Dynam. Systems \textbf{18}
  (1998), no.~2, 293--301.

\bibitem[DS93]{DS93}
G.~David and S.~Semmes, \emph{Analysis of and on uniformly rectifiable sets},
  Mathematical Surveys and Monographs, vol.~38, American Mathematical Society,
  Providence, RI, 1993.

\bibitem[Dya18]{dyatlovnotes}
S.~Dyatlov, \emph{Notes on hyperbolic dynamics}, 2018, arXiv:1805.11660.

\bibitem[Gd90]{GdH90}
E.~Ghys and P.~{de la Harpe} (eds.), \emph{Sur les groupes hyperboliques
  d'apr{\`e}s {Mikhael} {Gromov}. ({On} the hyperbolic groups {\`a} la {M}.
  {Gromov})}, Prog. Math., vol.~83, Boston, MA: Birkh{\"a}user, 1990 (French).

\bibitem[GHW21]{GHW21}
C.~Guillarmou, J.~Hilgert, and T.~Weich, \emph{High frequency limits for
  invariant {Ruelle} densities}, Ann. Henri Lebesgue \textbf{4} (2021),
  81--119.

\bibitem[Glo17]{GLO17}
O.~Glorieux, \emph{Counting closed geodesics in globally hyperbolic maximal
  compact {A}d{S} 3-manifolds}, Geom. Dedicata \textbf{188} (2017), 63--101.

\bibitem[GM21]{GM21}
O.~Glorieux and D.~Monclair, \emph{Critical exponent and {H}ausdorff dimension
  in pseudo-{R}iemannian hyperbolic geometry}, Int. Math. Res. Not. IMRN
  (2021), no.~18, 13661--13729.

\bibitem[GMT21]{GMT21}
C.~Guillarmou, M.~Mazzucchelli, and L.~Tzou, \emph{Boundary and lens rigidity
  for non-convex manifolds}, Amer. J. Math. \textbf{143} (2021), no.~2,
  533--575.

\bibitem[GMT23]{GMT}
O.~Glorieux, D.~Monclair, and N.~Tholozan, \emph{Hausdorff dimension of limit
  sets for projective {Anosov} representations}, J. {\'E}c. Polytech., Math.
  \textbf{10} (2023), 1157--1193.

\bibitem[Gro87]{Gr}
M.~Gromov, \emph{Hyperbolic groups}, Essays in Group Theory (S.~M. Gersten,
  ed.), Springer New York, New York, NY, 1987, pp.~75--263.

\bibitem[GW12]{GW12}
O.~{Guichard} and A.~{Wienhard}, \emph{{Anosov representations: domains of
  discontinuity and applications.}}, {Invent. Math.} \textbf{190} (2012),
  no.~2, 357--438.

\bibitem[H\"03]{HOR03}
L.~H\"{o}rmander, \emph{The analysis of linear partial differential operators.
  {I}}, Classics in Mathematics, Springer-Verlag, Berlin, 2003, Reprint of the
  second (1990) edition.

\bibitem[Hum24]{TH24}
T.~Humbert, \emph{{First Ruelle resonance for an Anosov flow with smooth
  potential}}, 2024, arXiv:2402.04948.

\bibitem[JT24]{JT23}
L.~Jin and Z.~Tao, \emph{{Counting Pollicott--Ruelle resonances for Axiom A
  flows}}, 2024, arXiv:2306.02297.

\bibitem[KH95]{KatokHasselblatt}
A.~Katok and B.~Hasselblatt, \emph{Introduction to the modern theory of
  dynamical systems}, Encycl. Math. Appl., vol.~54, Cambridge: Cambridge Univ.
  Press, 1995.

\bibitem[KL06]{KL06}
B.~Kleiner and B.~Leeb, \emph{Rigidity of invariant convex sets in symmetric
  spaces}, Invent. Math. \textbf{163} (2006), no.~3, 657--676.

\bibitem[KLP17]{KLP17}
M.~Kapovich, B.~Leeb, and J.~Porti, \emph{Anosov subgroups: dynamical and
  geometric characterizations}, Eur. J. Math. \textbf{3} (2017), no.~4,
  808--898.

\bibitem[KLP18a]{KLP18}
\bysame, \emph{Dynamics on flag manifolds: domains of proper discontinuity and
  cocompactness}, Geom. Topol. \textbf{22} (2018), no.~1, 157--234.

\bibitem[KLP18b]{KLP2}
\bysame, \emph{A {Morse} lemma for quasigeodesics in symmetric spaces and
  {Euclidean} buildings}, Geom. Topol. \textbf{22} (2018), no.~7, 3827--3923.

\bibitem[KP22]{KasselPotrie}
F.~Kassel and R.~Potrie, \emph{Eigenvalue gaps for hyperbolic groups and
  semigroups}, J. Mod. Dyn. \textbf{18} (2022), 161--208.

\bibitem[{Lab}06]{labourie}
F.~{Labourie}, \emph{{Anosov flows, surface groups and curves in projective
  space.}}, {Invent. Math.} \textbf{165} (2006), no.~1, 51--114.

\bibitem[Led95]{LED95}
F.~Ledrappier, \emph{Structure au bord des vari\'{e}t\'{e}s \`a courbure
  n\'{e}gative}, S\'{e}minaire de {T}h\'{e}orie {S}pectrale et
  {G}\'{e}om\'{e}trie, {N}o. 13, {A}nn\'{e}e 1994--1995, S\'{e}min. Th\'{e}or.
  Spectr. G\'{e}om., vol.~13, Univ. Grenoble I, Saint-Martin-d'H\`eres, 1995,
  pp.~97--122 (French).

\bibitem[Liv04]{Li04}
C.~Liverani, \emph{On contact {Anosov} flows}, Ann. Math. (2) \textbf{159}
  (2004), no.~3, 1275--1312.

\bibitem[Med21]{Med21}
A.~Meddane, \emph{{A Morse complex for Axiom A flows}}, 2021, arXiv:2107.08875.

\bibitem[{Mes}07]{mess2007}
G.~{Mess}, \emph{{Lorentz spacetimes of constant curvature.}}, {Geom. Dedicata}
  \textbf{126} (2007), 3--45.

\bibitem[Min05]{Mineyev2005}
I.~Mineyev, \emph{Flows and joins of metric spaces.}, Geometry \& Topology
  \textbf{9} (2005), 403--482.

\bibitem[Nau05]{NAU05}
F.~Naud, \emph{Expanding maps on {C}antor sets and analytic continuation of
  zeta functions}, Ann. Sci. \'{E}cole Norm. Sup. (4) \textbf{38} (2005),
  no.~1, 116--153.

\bibitem[Pol85]{POL85}
M.~Pollicott, \emph{On the rate of mixing of {A}xiom {A} flows}, Invent. Math.
  \textbf{81} (1985), no.~3, 413--426.

\bibitem[Pol87]{POL87}
\bysame, \emph{Symbolic dynamics for {S}male flows}, Amer. J. Math.
  \textbf{109} (1987), no.~1, 183--200.

\bibitem[PP83]{PP83}
W.~Parry and M.~Pollicott, \emph{An analogue of the prime number theorem for
  closed orbits of {A}xiom {A} flows}, Ann. Math. (2) \textbf{118} (1983),
  no.~3, 573--591.

\bibitem[PS98]{PS98}
M.~Pollicott and R.~Sharp, \emph{Exponential error terms for growth functions
  on negatively curved surfaces}, Amer. J. Math. \textbf{120} (1998), no.~5,
  1019--1042.

\bibitem[PS24]{PS24}
\bysame, \emph{Zeta functions in higher {T}eichm\"{u}ller theory}, Math. Z.
  \textbf{306} (2024), no.~37.

\bibitem[PSW21]{PSW21}
M.B. Pozzetti, A.~Sambarino, and A.~Wienhard, \emph{Conformality for a robust
  class of non-conformal attractors}, J. Reine Angew. Math. \textbf{774}
  (2021), 1--51.

\bibitem[Qui02a]{QUI02C}
J.-F. Quint, \emph{C\^{o}nes limites des sous-groupes discrets des groupes
  r\'{e}ductifs sur un corps local}, Transform. Groups \textbf{7} (2002),
  no.~3, 247--266 (French).

\bibitem[Qui02b]{QUI02D}
\bysame, \emph{Divergence exponentielle des sous-groupes discrets en rang
  sup\'{e}rieur}, Comment. Math. Helv. \textbf{77} (2002), no.~3, 563--608
  (French).

\bibitem[Qui05]{Quint05}
\bysame, \emph{Groupes convexes cocompacts en rang sup\'{e}rieur}, Geom.
  Dedicata \textbf{113} (2005), 1--19 (French).

\bibitem[Rat73]{RAT73}
M.~Ratner, \emph{Markov partitions for {A}nosov flows on {$n$}-dimensional
  manifolds}, Israel J. Math. \textbf{15} (1973), 92--114.

\bibitem[Sam14]{SAM14}
A.~Sambarino, \emph{Quantitative properties of convex representations},
  Comment. Math. Helv. \textbf{89} (2014), no.~2, 443--488.

\bibitem[Sam24]{SAM24}
\bysame, \emph{A report on an ergodic dichotomy}, Ergodic Theory Dynam. Systems
  \textbf{44} (2024), no.~1, 236--289.

\bibitem[Sma67]{smale67}
S.~Smale, \emph{Differentiable dynamical systems. {With} an appendix to the
  first part of the paper: ``{Anosov} diffeomorphisms'' by {John} {Mather}},
  Bull. Am. Math. Soc. \textbf{73} (1967), 747--817.

\bibitem[Sto11]{St11}
L.~Stoyanov, \emph{{Spectra of Ruelle transfer operators for Axiom A flows}},
  {Nonlinearity} \textbf{24} (2011), no.~4, 1089--1120.

\bibitem[Sto13]{St13}
\bysame, \emph{Pinching conditions, linearization and regularity of axiom {A}
  flows}, Discrete Contin. Dyn. Syst. \textbf{33} (2013), no.~2, 391--412.

\bibitem[Thu22]{THU22}
W.~P. Thurston, \emph{{The geometry and topology of three-manifolds. Vol. IV}},
  American Mathematical Society, Providence, RI, 2022.
\end{thebibliography}
\end{document}